\documentclass[12pt]{amsart}
\usepackage{amssymb}
\usepackage[mathcal]{eucal}
\usepackage{longtable,array}
\usepackage{enumitem}
\usepackage{hyperref}
\usepackage{tikz-cd}
\usepackage[total={6.25truein,9truein},centering]{geometry}

\allowdisplaybreaks

\usepackage{cite}
\makeatletter
\newcommand{\citecomment}[2][]{\citen{#2}#1\citevar}
\newcommand{\citeone}[1]{\citecomment{#1}}
\newcommand{\citetwo}[2][]{\citecomment[,~#1]{#2}}
\newcommand{\citevar}{\@ifnextchar\bgroup{;~\citeone}{\@ifnextchar[{;~\citetwo}{]}}}
\newcommand{\citefirst}{\@ifnextchar\bgroup{\citeone}{\@ifnextchar[{\citetwo}{]}}}
\newcommand{\cites}{[\citefirst}
\makeatother

\makeatletter
\let\c@equation\c@subsubsection

\makeatother

\newtheorem{theorem}[subsubsection]{Theorem}
\newtheorem{lemma}[subsubsection]{Lemma}
\newtheorem{proposition}[subsubsection]{Proposition}
\newtheorem{corollary}[subsubsection]{Corollary}

\newtheorem{theoremintro}{Theorem}

\newtheorem{corollaryintro}[theoremintro]{Corollary}

\theoremstyle{definition}

\theoremstyle{remark}
\newtheorem{example}[subsubsection]{Example}
\newtheorem{remark}[subsubsection]{Remark}
\newtheorem{definition}[subsubsection]{Definition}

\newcommand{\AAA}{\mathbb{A}}
\newcommand{\FF}{\mathbb{F}}
\newcommand{\ZZ}{\mathbb{Z}}
\newcommand{\CC}{\mathbb{C}}
\newcommand{\TT}{\mathbb{T}}
\newcommand{\EE}{\mathbb{E}}
\newcommand{\LL}{\mathbb{L}}
\newcommand{\bA}{\mathbf{A}}
\newcommand{\be}{\mathbf{e}}

\newcommand{\bff}{\mathbf{f}}
\newcommand{\bg}{\mathbf{g}}
\newcommand{\bh}{\mathbf{h}}

\newcommand{\bL}{\mathbf{L}}
\newcommand{\bm}{\mathbf{m}}
\newcommand{\bn}{\mathbf{n}}
\newcommand{\bP}{\mathbf{P}}
\newcommand{\bQ}{\mathbf{Q}}
\newcommand{\br}{\mathbf{r}}

\newcommand{\bs}{\mathbf{s}}
\newcommand{\bu}{\mathbf{u}}
\newcommand{\bv}{\mathbf{v}}
\newcommand{\bx}{\mathbf{x}}

\newcommand{\by}{\mathbf{y}}

\newcommand{\bz}{\mathbf{z}}

\newcommand{\cB}{\mathcal{B}}
\newcommand{\cE}{\mathcal{E}}
\newcommand{\cI}{\mathcal{I}}

\newcommand{\cM}{\mathcal{M}}
\newcommand{\cN}{\mathcal{N}}
\newcommand{\cR}{\mathcal{R}}
\newcommand{\cT}{\mathcal{T}}

\newcommand{\rC}{\mathrm{C}}

\newcommand{\rH}{\mathrm{H}}
\newcommand{\rI}{\mathrm{I}}
\newcommand{\rR}{\mathrm{R}}

\newcommand{\sA}{\mathsf{A}}
\newcommand{\sC}{\mathsf{C}}

\newcommand{\hbeta}{\hat{\beta}}
\newcommand{\hL}{\hat{L}}

\newcommand{\bsalpha}{\boldsymbol{\alpha}}
\newcommand{\bsbeta}{\boldsymbol{\beta}}
\newcommand{\bslambda}{\boldsymbol{\lambda}}
\newcommand{\bsmu}{\boldsymbol{\mu}}
\newcommand{\bsnu}{\boldsymbol{\nu}}

\newcommand*{\isoarrow}[1]{\arrow[#1,"\rotatebox{90}{\(\sim\)}"
]}

\DeclareMathOperator{\Aut}{Aut}
\DeclareMathOperator{\Exp}{Exp}
\DeclareMathOperator{\Log}{Log}
\DeclareMathOperator{\GL}{GL}
\DeclareMathOperator{\Lie}{Lie}
\DeclareMathOperator{\Mat}{Mat}
\DeclareMathOperator{\Char}{Char}
\DeclareMathOperator{\Span}{Span}
\DeclareMathOperator{\Gal}{Gal}

\DeclareMathOperator{\rank}{rank}

\DeclareMathOperator{\sgn}{sgn}
\DeclareMathOperator{\ord}{ord}
\DeclareMathOperator{\Sym}{Sym}
\makeatletter
\newcommand{\Alt}{\@ifnextchar^\@Alt{\@Alt^{\,}}}
\def\@Alt^#1{\mathop{\bigwedge\nolimits^{\!#1}}}
\makeatother

\newcommand{\phitenpsi}{\phi\otimes\psi}
\newcommand{\ophitenpsi}{\mkern2.5mu\overline{\mkern-2.5mu \phitenpsi}}
\newcommand{\tenphi}{\phi^{\otimes 2}}
\newcommand{\symphi}{\Sym^2\phi}
\newcommand{\osymphi}{\mkern2.5mu\overline{\mkern-2.5mu \Sym^2\phi}}
\newcommand{\altphi}{\Alt^2\phi}
\newcommand{\oaltphi}{\mkern2.5mu\overline{\mkern-2.5mu \Alt^2\phi}}

\newcommand{\oF}{\mkern2.5mu\overline{\mkern-2.5mu F}}
\newcommand{\oK}{\mkern2.5mu\overline{\mkern-2.5mu K}}

\newcommand{\ochi}{\overline{\chi}}
\newcommand{\ophi}{\mkern2.5mu\overline{\mkern-2.5mu \phi}}
\newcommand{\opsi}{\mkern2.5mu\overline{\mkern-2.5mu \psi}}
\newcommand{\otheta}{\mkern2.5mu\overline{\mkern-2.5mu \theta}}
\newcommand{\Reg}{\mathrm{Reg}}
\newcommand{\sep}{\mathrm{sep}}
\newcommand{\nr}{\mathrm{nr}}
\newcommand{\tB}{\widetilde{B}}
\newcommand{\tC}{\widetilde{C}}
\newcommand{\tL}{\widetilde{L}}

\newcommand{\tPhi}{\widetilde{\Phi}}
\newcommand{\tE}{\tilde{E}}
\newcommand{\tU}{\widetilde{U}}
\newcommand{\tbeta}{\tilde{\beta}}

\newcommand{\C}{\CC_{\infty}}
\newcommand{\tbsmu}{\tilde{\bsmu}}

\newcommand{\soE}{\mkern1mu\overline{\mkern-1mu E}}

\newcommand{\iso}{\stackrel{\sim}{\longrightarrow}}

\newcommand{\power}[2]{{#1 [\![ #2 ]\!]}}
\newcommand{\laurent}[2]{{#1 (\!( #2 )\!)}}

\newcommand{\norm}[1]{\lvert #1 \rvert}
\newcommand{\dnorm}[1]{\lVert #1 \rVert}
\newcommand{\inorm}[1]{{\lvert #1 \rvert}_{\infty}}
\newcommand{\Aord}[2]{{[ #1 ]}_{#2}}
\newcommand{\rbracket}[1]{{\left( #1 \right)}}
\newcommand{\bracket}[1]{{\left[ #1 \right]}}
\newcommand{\bigAord}[2]{{\left[ #1 \right]}_{#2}}
\newcommand{\pd}{\partial}
\newcommand{\tr}{{\mathsf{T}}}

\newcommand{\assign}{\mathrel{\vcenter{\baselineskip0.5ex \lineskiplimit0pt
                     \hbox{\scriptsize.}\hbox{\scriptsize.}}}%
                     =}
\newcommand{\rassign}{=%
                     \mathrel{\vcenter{\baselineskip0.5ex \lineskiplimit0pt
                     \hbox{\scriptsize.}\hbox{\scriptsize.}}}%
                     }

\usepackage{comment}

\begin{document}

\title{Tensor products of Drinfeld modules and convolutions of Goss $L$-series}

\author{Wei-Cheng Huang}
\address{Department of Mathematics, Texas A{\&}M University, College Station, TX 77843, U.S.A.}
\email[W.-C. Huang]{wchuang@tamu.edu}


\date{August 10, 2023}

\begin{abstract}
Following the same framework of the special value results of convolutions of Goss and Pellarin $L$-series attached to Drinfeld modules that take values in Tate algebras by Papanikolas and the author, we establish special value results of convolutions of two Goss $L$-series attached to Drinfeld modules that take values in $\laurent{\FF_q}{\frac{1}{\theta}}.$ Applying the class module formula of Fang to tensor products of two Drinfeld modules, we provide special value formulas for their $L$-functions. By way of the theory of Schur polynomials these identities take the form of specializations of convolutions of Rankin-Selberg type. Finally, we show an explicit computation of the regulators appearing in Fang's class module formula for tensor products as well as symmetric and alternating squares of Drinfeld modules.
\end{abstract}


\maketitle

\tableofcontents

\section{Introduction}

Let $\FF_q$ be a field with $q=p^m$ elements for $p$ a prime. For a variable $\theta$ we let $A \assign \FF_q[\theta]$ be a polynomial ring in~$\theta$ over~$\FF_q$, and let $K \assign \FF_q(\theta)$ be its fraction field. We take $K_{\infty} \assign \laurent{\FF_q}{\theta^{-1}}$ for the completion of $K$ at $\infty$, and let $\C$ be the completion of an algebraic closure of $K_{\infty}$. We normalize the $\infty$-adic norm $\inorm{\,\cdot\,}$ on $\C$ so that $\inorm{\theta} = q$, and letting $\deg \assign -\ord_{\infty} = \log_q \inorm{\,\cdot\,}$, we see that $\deg a = \deg_{\theta} a$ for any $a \in A$. Finally, we let $A_+$ denote the monic elements of $A$.

\subsection{Motivation}

Guided by a series of articles by Angl\`es, Demeslay, Gezmi\c{s}, Pellarin, Taelmann, Tavares Ribeiro~{\cites{APT16}{AT17}{Demeslay15}{DemeslayPhD}{Gezmis19}{Pellarin12}{Taelman09}{Taelman10}{Taelman12}}, Papanikolas and the author \cite{HP22} defined a $t$-module $\EE(\phi \times \psi)$ that is the twist of $\phi$ by the rigid analytic trivialization of~$\psi$. Then its associated $L$-function includes a Rankin-Selberg type convolution of a Goss $L$-series and a Pellarin $L$-series (see \cite[Thm.~ 6.2.3, 6.3.5]{HP22}) and can be evaluated using Demeslay's class module identity (see \cite[Cor.~6.2.4, 6.3.6]{HP22}).

Inspired by these convolutions, it is natural to consider the Rankin-Selberg type convolution of two Goss $L$-series and ask to what extent the regulators are related to the special values of logarithms. This leads us to study Goss $L$-series of tensor products, symmetric squares and alternating squares of Drinfeld modules (see Theorem~\ref{T:tenconv} and Corollary~\ref{C:Lmuvalue}). We also provide explicit expressions of their regulators for the rank $2$ case (see Corollary~\ref{Cintro:log} and Theorem~\ref{T:computeReg}) in terms of these $L$-values and logarithms. We now summarize these results.

\subsection{Convolution Goss \texorpdfstring{$L$}{L}-functions and special values}
Let $\sA = \FF_q[t]$, and let $\phi$, $\psi : \sA \to A[\tau]$ be Drinfeld modules defined over~$A$ by
\begin{equation} \label{E:phipsidefintro}
\phi_t = \theta + \kappa_1 \tau + \cdots + \kappa_r \tau^r, \quad
\psi_t = \theta + \eta_1 \tau + \cdots + \eta_{\ell} \tau^{\ell}, \quad \kappa_r, \eta_\ell \in \FF_q^{\times}.
\end{equation}
Thus $\phi$ has rank $r$ and $\psi$ has rank $\ell$, and moreover because their leading coefficients are in $\FF_q^{\times}$, both $\phi$ and $\psi$ have everywhere good reduction.

\subsubsection{Characteristic polynomials of Frobenius}
For our Drinfeld module $\phi$ in \eqref{E:phipsidefintro}, if we fix $f \in A_+$ irreducible of degree $d$ and let $\lambda \in \sA_+$ be irreducible with $\lambda(\theta) \neq f$, then by work of Gekeler, Hsia, Takahashi, and Yu~\cites{Gekeler91, HsiaYu00, Takahashi82}, the characteristic polynomial $P_{\phi,f}(X) = \Char(\tau^d,T_{\lambda}(\ophi),X) = X^r + c_{r-1} X^{r-1} + \cdots + c_0 \in A[X]$ of $\tau^d$ acting on $T_{\lambda}(\ophi)$ satisfies $c_0 = (-1)^r \ochi_{\phi}(f) f$, where $\chi_{\phi}(a) \assign ((-1)^{r+1} \kappa_r)^{\deg a}$ and $\ochi_{\phi} = \chi_{\phi}^{-1}$. We further let $P_{\phi,f}^\vee(X)\in K[X]$ be the characteristic polynomial of $\tau^d$ acting on the dual space of $T_\lambda(\ophi)$. See \S\ref{SS:munu} for more details.

\subsubsection{\texorpdfstring{$L$}{L}-functions of tensor products}
For the Drinfeld module $\phi$ from \eqref{E:phipsidefintro}, Goss defined the $L$-function,
\[
L(\phi^{\vee},s) = \prod_f Q_{\phi,f}^{\vee} \bigl( f^{-s} \bigr)^{-1}=\sum_{a\in A_+}\frac{\mu_\phi(a)}{a^{s+1}},\]
where $Q_{\phi,f}^{\vee}(X)$ is the reciprocal polynomial of $P_{\phi,f}^{\vee}(X)$. The multiplicative function $\mu_{\phi}: A_+ \to A$ is defined by the generating series,
\[
\sum_{m=1}^{\infty} \mu_{\phi}(f^m) X^m = Q_f^{\vee}(fX)^{-1}.
\]
We similarly define $\mu_{\psi}$. See \S\ref{SS:munu} for details. 

One of our main goals is to express Dirichlet series for $L$-function of tensor products in a similar explicit fashion. If $E = \phitenpsi$, Cauchy's identity (see Theorem~\ref{T:Cauchy}) implies that the $L$-function $L(E^\vee,s)$ has a convolution interpretation, following the situation for Maass forms on $\GL_n$ (see \cites{Bump89, Goldfeld}). When $E = \symphi$ or $\altphi$, Littlewood's identities (see Theorem~\ref{T:Littlewood}) imply that the $L$-function $L(E^\vee,s)$ can be factored into a twisted Carlitz zeta function and intriguing $L$-functions involving $\bsmu_{\phi}$ defined in \cite{HP22}.  In particular, if $\alpha_1, \dots, \alpha_r \in \oK$ are the roots of $P_{\phi,f}^{\vee}(X)$, then for $k_1, \dots, k_{r} \geqslant 0$, we define
\begin{align}
\bsmu_{\phi}\bigl( f^{k_1}, \dots, f^{k_{r-1}} \bigr) &\assign S_{k_1, \dots, k_{r-1}} (\alpha_1, \dots, \alpha_r) \cdot f^{k_1 + \cdots + k_{r-1}},
\end{align}
where $S_{k_1, \dots, k_{r-1}}(x_1, \dots, x_r)$ is the Schur polynomial defined in~\eqref{E:Sk}. We extend $\bsmu_{\phi}$ to a function on $(A_+)^{r-1}$ multiplicatively, and then find that for $a_1, \dots, a_{r-1} \in A_+$, we have $\bsmu_{\phi}(a_1, \dots, a_{r-1}) \in A$. Furthermore, for any $a \in A_+$,
\[
\bsmu_{\phi}(a,1, \dots, 1) = \mu_{\phi}(a).
\]
The function $\bsmu_{\phi} : (A_+)^{r-1} \to A$ satisfies various relations induced by relations on Schur polynomials. See \S\ref{SS:bsmubsnu} for details.

When $r$, $\ell\geqslant 2$, we define an $L$-function $L(\bsmu_{\phi} \times \bsmu_{\psi},s)$ as follows. If $r=\ell$
\begin{equation}
L(\bsmu_{\phi} \times \bsmu_{\psi},s) \assign \sum_{a_1 \in A_+} \cdots \sum_{a_{r-1} \in A_+} \frac{\bsmu_{\phi}(a_1, \dots, a_{r-1}) \bsmu_{\psi}(a_1, \dots, a_{r-1})}{(a_1 \cdots a_{r-1})^2 (a_1 a_2^2 \cdots a_{r-1}^{r-1})^s}.
\end{equation}
If $r< \ell$, then
\begin{equation}
L(\bsmu_{\phi} \times \bsmu_{\psi},s) \assign
\sum_{a_1, \ldots, a_r \in A_+}
\frac{\chi_{\phi}(a_r)\bsmu_{\phi}(a_1, \dots, a_{r-1}) \bsmu_{\psi}(a_1, \dots, a_r,1, \ldots, 1)}{(a_1 \cdots a_r)^2 (a_1 a_2^2 \cdots a_r^r)^s}.
\end{equation}
In this way we interpret $L(\bsmu_{\phi} \times \bsmu_{\psi},s)$ as a convolution of two Goss $L$-series. Since $\phitenpsi\cong\phi\otimes\phi$ (see \cite[Prop.~2.5]{Hamahata}, Rmk.~\ref{R:teniso}), the $r>\ell$ case is the same as the $r<\ell$ case in the sense of switching the rolls of $\phi$ and $\psi$. We further define two $L$-functions as follows.
\begin{equation}
L(\tbsmu_{\phi},s) \assign \sum_{a_1 \in A_+} \cdots \sum_{a_{r-1} \in A_+} \frac{\bsmu_{\phi}(a_1^2, \dots, a_{r-1}^2)}{(a_1 \cdots a_{r-1})^2 (a_1 a_2^2 \cdots a_{r-1}^{r-1})^s},
\end{equation}
and 
\begin{equation}
L(\hat{\bsmu}_{\phi},s) \assign \underset{a_i = 1 \text{ if $2\nmid i$}}{\sum_{a_1,\dots a_{r-1} \in A_+}} \frac{\bsmu_{\phi}(a_1, \dots, a_{r-1})}{a_1 \cdots a_{r-1} (a_1 a_2^2 \cdots a_{r-1}^{r-1})^s}.
\end{equation}
For the cases $E=\tenphi$, $\symphi$, $\altphi$. The $L$-series above are related to $L(E^\vee,s)$ by the following result (see Theorems~\ref{T:Ltenrxr}, \ref{T:Ltenrxl}, \ref{T:Lsymrxr} and \ref{T:Laltrxr}). We let $L(A,\chi,s) = \sum_{a\in A_+}\chi(a)\cdot a^{-s}$ be the twist of the Carlitz zeta function $L(A,s) = \sum_{a\in A_+}a^{-s}$ by a completely mutiplicative function $\chi:A_+\to \FF^\times_q.$

\begin{theoremintro}\label{T:tenconv}
Let $\phi$, $\psi : \sA \to A[\tau]$ be Drinfeld modules of ranks $r$ and $\ell$ respectively with everywhere good reduction, as defined in~\eqref{E:phipsidefintro}. Assume that $r$, $\ell \geqslant 2$.
\begin{enumerate}
\item If $r=\ell$, then
\[
L((\phi \otimes \psi)^{\vee},s) = L(A,\chi_{\phi}\chi_{\psi},rs+2) \cdot
L(\bsmu_{\phi} \times \bsmu_{\psi},s).
\]
\item If $r<\ell$, then
\[
L((\phi \otimes \psi)^{\vee},s) = L(\bsmu_{\phi} \times \bsmu_{\psi},s).
\]    

\item Assume further that $p\neq 2$, then
\[
L((\symphi)^{\vee},s) = L(A,\chi_{\phi}^2,rs+2) \cdot
L(\tbsmu_{\phi},s).
\]
and
\[
 L((\altphi)^{\vee},s) = L(A,\chi_{\phi},\frac{rs}{2}+1)^{\frac{(-1)^r+1}{2}} \cdot
L(\hat{\bsmu}_{\phi},s).
\]
\end{enumerate}
\end{theoremintro}

Substituting $s=0$ in Theorem~\ref{T:tenconv} provides special value identities for $L(\bsmu_{\phi} \times \bsmu_{\psi},0)$, $L(\tbsmu_{\phi},0)$ and $L(\hat{\bsmu}_{\phi},0)$. Fang's class module identity (Theorem~\ref{FangClass}) implies the following corollary (stated as Corollaries \ref{C:Lmumurxr}, \ref{C:Lmumurxl}, \ref{C:Lsymrxr}, \ref{C:Laltrxr}).

\begin{corollaryintro}\label{C:Lmuvalue}
Let $\phi$, $\psi : \sA \to A[\tau]$ be Drinfeld modules of ranks $r$ and $\ell$ respectively with everywhere good reduction, as defined in~\eqref{E:phipsidefintro}. Assume that $r$, $\ell \geqslant 2$.
\begin{enumerate}
\item If $r=\ell$, then
\begin{align*}
L(\bsmu_{\phi} \times \bsmu_{\psi},0)
&= \sum_{a_1 \in A_+} \cdots \sum_{a_{r-1} \in A_+} \frac{\bsmu_{\phi}(a_1, \dots, a_{r-1}) \bsmu_{\psi}(a_1, \dots, a_{r-1})}{(a_1 \cdots a_{r-1})^2} \\[10pt]
&= \frac{ \Reg_{\phitenpsi} \cdot \Aord{\rH(\phitenpsi)}{A}}{L(A,\chi_\phi\chi_\psi,2)}.
\end{align*}
\item If $r<\ell$, then
\begin{align*}
L(\bsmu_{\phi} \times \bsmu_{\psi},0)
&= \sum_{a_1, \ldots, a_r \in A_+}
\frac{\chi_{\phi}(a_r)\bsmu_{\phi}(a_1, \dots, a_{r-1}) \bsmu_{\psi}(a_1, \dots, a_r,1, \ldots, 1)}{(a_1 \cdots a_r)^2} \\[10pt]
&= \Reg_{\phitenpsi} \cdot \Aord{\rH(\phitenpsi)}{A}.
\end{align*}
\item Assume further that $p\neq 2$, then
\begin{align*}
L(\tbsmu_{\phi},0)
= \sum_{a_1 \in A_+} \cdots \sum_{a_{r-1} \in A_+} \frac{\bsmu_{\phi}(a_1^2, \dots, a_{r-1}^2)}{(a_1 \cdots a_{r-1})^2} 
= \frac{ \Reg_{\symphi} \cdot \Aord{\rH(\symphi)}{A}}{L(A,\chi_\phi^2,2)},
\end{align*}
and
\begin{align*}
L(\hat{\bsmu}_{\phi},0)
= \underset{a_i = 1 \text{ if $2\nmid i$}}{\sum_{a_1,\dots a_{r-1} \in A_+}} \frac{\bsmu_{\phi}(a_1, \dots, a_{r-1})}{a_1 \cdots a_{r-1}}
= \frac{ \Reg_{\altphi} \cdot \Aord{\rH(\altphi)}{A}}{L(A,\chi_\phi,1)^{\frac{(-1)^r+1}{2}}}.
\end{align*}
\end{enumerate}
\end{corollaryintro}

We now assume $r=2$ and suppose $\Log_\phi(z) = \sum_{m\geq 0}\beta_m z^{q^m}$. In \S5.2 we define $\FF_q$-linear power series $L_i(z)$, $L'_i(z)$, $\tL_i(z)$, e.g., see \eqref{eq:start1}--\eqref{eq:end1}, which combine to make the coordinate functions of the logarithm series of $\phitenpsi$, $\symphi$ and $\altphi$. The main result in this part of the paper (Corollary~\ref{C:log}) is the following. 

\begin{corollaryintro}\label{Cintro:log}
Let $\phi:\sA\to A[\tau]$ be a Drinfeld module $\phi:\sA\to A[\tau]$ given by $\phi_t=\theta+\kappa_1\tau+\kappa_2\tau^2$ with $\kappa_2\in\FF_q^{\times}$, and let
\[\bL_m\assign\begin{pmatrix}
        L_{1,m}+(\theta-\theta^{(m)})L'_{1,m}& \kappa_2L'_{2,m}& \kappa_1^{(m)}L'_{1,m}+\kappa_2L'_{2,m}& \kappa_2L'_{1,m}\\
        \frac{(\theta-\theta^{(m)})}{\kappa_2}\tL_{1,m}& \tL_{0,m}+\tL_{2,m}& \frac{\kappa_1^{(m)}}{\kappa_2}\tL_{1,m}+\tL_{2,m}& \tL_{1,m}\\
        \frac{(\theta-\theta^{(m)})}{\kappa_2}\tL_{1,m}& \tL_{2,m}& \tL_{0,m}+\frac{\kappa_1^{(m)}}{\kappa_2}\tL_{1,m}+\tL_{2,m}& \tL_{1,m}\\
        \frac{(\theta-\theta^{(m)})}{\kappa_2}L_{1,m}& L_{2,m}& \frac{\kappa_1^{(m)}}{\kappa_2}L_{1,m}+L_{2,m}& L_{1,m}
    \end{pmatrix}\in\Mat_4(K)\]  for $m\geqslant 1$, then
\begin{enumerate}
    \item $\Log_{\tenphi}\begin{pmatrix}
        z_1\\z_2\\z_3\\z_4
    \end{pmatrix}=\begin{pmatrix}
        z_1\\z_2\\z_3\\z_4
    \end{pmatrix}+\begin{pmatrix}
        1&&&\\&1&&\\&&1&\\&&-\frac{\kappa_1}{\kappa_2}&1
    \end{pmatrix}\sum_{m\geqslant 1}\bL_m\begin{pmatrix}
        z_1^{q^m}\\z_2^{q^m}\\z_3^{q^m}\\z_4^{q^m}
    \end{pmatrix},$
    \item $\Log_{\symphi}\begin{pmatrix}
        z_1\\z_2\\z_3
    \end{pmatrix}=\begin{pmatrix}
        z_1\\z_2\\z_3
    \end{pmatrix}+\begin{pmatrix}
        1&&&\\&\frac{1}{2}&\frac{1}{2}&\\&-\frac{\kappa_1}{2\kappa_2}& -\frac{\kappa_1}{2\kappa_2}&1
    \end{pmatrix}\sum_{m\geqslant 1}\bL_m\begin{pmatrix}
        z_1^{q^m}\\z_2^{q^m}\\z_2^{q^m}\\z_3^{q^m}
    \end{pmatrix},$
    \item $\Log_{\altphi}(z) = z+\tL_0(z).$
\end{enumerate}
\end{corollaryintro}

We further define the dilogarithm series $\Log_{\phi,2}(z)=\sum_{m\geqslant 0}\beta_m^2z^{q^m}$ and $\hL_i(z)$ in \eqref{eq:start2}--\eqref{eq:L2hat}, and obtain the following result on regulators (stated as Theorem~\ref{T:reg}). Note that $\hL_i(z)$ is related to $L'_i(z)$ by the chain rule {\cite[Lem.~2.4.6]{NamoijamP22}} (see also \eqref{eq:chain}).


\begin{theoremintro}\label{T:computeReg}
We have the following formulas for regulators. 
\begin{enumerate}
\item Assume that $\deg(\kappa_1)\leqslant (q+1)/2$. Then \[\Reg_{\Alt^2\phi}=\Log_{\Alt^2\phi}(1)=1+\frac{\kappa_2}{\theta-\theta^{(1)}}\sum_{m\geq 1}(\beta_m\tbeta_{m-2}-\beta_{m-1}\tbeta_{m-1}).\] 

\item Assume that $\deg(\kappa_1)\leqslant 1.$ Then
\begin{enumerate}
    \item[(i)] $\Reg_{\Sym^2\phi}=\gamma\cdot\det(M)$, where $\gamma\in\FF_q^\times$ is chosen so that it has sign $1$, and  
    \[M=
    \begin{pmatrix}1+\hL_1(\theta) & -\hL_1(\kappa_1)-2\kappa_2\hL_2(1)-\pd_\theta(\kappa_1) & -\kappa_2\hL_1(1)\\
    -\kappa_2^{-1}\tL_1(\theta) & \Log_{\altphi}(1)+\kappa_2^{-1}\tL_1(\kappa_1)+2\tL_2(1) & \tL_1(1)\\
    -\kappa_2^{-1}\Log_{\phi,2}(\theta) & \kappa_2^{-1}\Log_{\phi,2}(\kappa_1)+2L_2(1) & \Log_{\phi,2}(1)
    \end{pmatrix}.
    \]
    \item[(ii)] $\Reg_{\phi^{\otimes 2}} = \Reg_{\Sym^2\phi}\cdot\Reg_{\altphi}$.
\end{enumerate}
\end{enumerate}
\end{theoremintro}

\begin{remark}
    From Theorem \ref{T:computeReg} we see the regulators can be explicitly expressed in terms of coordinate functions of logarithms. 
\end{remark}

\subsection{Outline}
After summarizing preliminary material in~\S\ref{S:prelim},   we define tensor products, symmetric and alternating squares of Drinfeld modules from the aspect of $t$-motives and explore their properties in~\S\ref{Ch:TensorProducts}. 
In~\S\ref{S:Ltenseries} we review the theories of Goss $L$-series, as well as Fang's class module formula. We consider the $L$-function of $\phitenpsi$, $\symphi$ and $\altphi$. Then we introduce the convolution $L$-series $L(\bsmu_{\phi} \times \bsmu_{\psi},s)$, $L(\tbsmu_{\phi},s)$ and $L(\hat{\bsmu}_{\phi},s)$, relate them to $L(\phitenpsi,s)$, $L(\symphi,s)$, $L(\altphi,s)$ and twisted Carlitz $L$-series, and investigate special value identities using Fang's class module formula. We provide explicit expressions of regulators $\Reg_{\phitenpsi}$, $\Reg_{\symphi}$ and $\Reg_{\altphi}$ for rank $2$ case in \S\ref{S:regulator}.

\section{Preliminaries}\label{S:prelim}
\subsection{Notation}
We will use the following notation throughout.
\begin{longtable}{p{1.25truein}@{\hspace{5pt}$=$\hspace{5pt}}p{4.5truein}}
$A$ & $\FF_q[\theta]$, polynomial ring in variable $\theta$ over $\FF_q$. \\
$A_+$ & the monic elements of $A$. \\
$K$ & $\FF_q(\theta)$, the fraction field of $A$. \\
$K_{\infty}$ & $\laurent{\FF_q}{\theta^{-1}}$, the completion of $K$ at $\infty$. \\
$\C$ & the completion of an algebraic closure $\oK_{\infty}$ of $K_{\infty}$.\\
$\inorm{\,\cdot\,}$; $\deg$ & $\infty$-adic norm on $\C$, extended to the sup norm on a finite dimensional $\C$-vector space; $\deg = -\ord_{\infty} = \log_q\inorm{\,\cdot\,}$.\\
$\FF_f$ & $A/fA$ for $f \in A_+$ irreducible. \\
$\sA$ & $\FF_q[t]$, for a variable $t$ independent from $\theta$. \\
$\TT_t$ & Tate algebra in $t$ $= \{ \sum a_i t^i \in \power{\C}{t} \mid \inorm{a_i} \to 0 \} = $ completion of $\C[t]$ with respect to Gauss norm. \\
\end{longtable}

\subsubsection{Rings of operators} 
For a variable $t$ independent from $\theta$ we let $\sA \assign \FF_q[t]$. We let $\TT_t$ denote the standard \emph{Tate algebra}, $\TT_t \subseteq \power{\C}{t}$,
consisting of power series that converge on the closed unit disk of $\C$, and we take
\begin{equation}
\TT_t(K_\infty) \assign \TT_t \cap \power{K_{\infty}}{t} = \laurent{\FF_q[t]}{\theta^{-1}}.
\end{equation}
We let $\dnorm{\,\cdot\,}$ denote the Gauss norm on $\TT_t$, such that $
\lVert \sum_{i=0}^{\infty} a_i t^i \rVert = \max_i \{ \inorm{a_i} \}$,
under which $\TT_t$ is a complete normed $\C$-vector space, and likewise $\TT_t(K_\infty)$ is a complete normed $K_{\infty}$-vector space. We extend the degree map on $\C$ to $\TT_t$ by taking
$\deg = \log_q \dnorm{\,\cdot\,}$. We further let $\TT_\theta$ denote the Tate algebra, $\TT_\theta \subseteq \power{\C}{t}$, consisting of power series that converge on the closed disk of radius $\inorm{\theta}$, and $\dnorm{\,\cdot\,}_\theta$ denote the norm on $\TT_\theta$ such that 
$
\dnorm{ \sum_{i=0}^{\infty} a_i t^i}_\theta = \max_i \{ q^i\cdot\inorm{a_i} \}$.

\subsubsection{Frobenius operators}
We take $\tau : \C \to \C$ for the \emph{$q$-th power Frobenius automorphism}, which we extend to $\laurent{\C}{t}$ 
by requiring it to commute with $t$.
 For $g = \sum c_i t^i \in \laurent{\C}{t}$, we define the \emph{$n$-th Frobenius twist},
\[
g^{(n)} \assign \tau^n(g) = \sum c_i^{q^n} t^i, \quad \forall\, n \in \ZZ.
\]
Then $\tau$ induces an $\FF_q(t)$-linear automorphism of $\TT_t$,
and the fixed ring of $\tau$ is $\TT_t^{\tau} = \FF_q[t]$.

\subsubsection{Twisted polynomials}\label{SSS:twistedpolys}
Let $R$ be any commutative $\FF_q$-algebra, and let $\tau : R \to R$ be an injective $\FF_q$-algebra endomorphism. Let $R^{\tau}$ be the $\FF_q$-subalgebra of $R$ of elements fixed by $\tau$. For $n \in \ZZ$ for which $\tau^n$ is defined on $R$ and a matrix $B=(b_{ij})$ with entries in $R$, we let $B^{(n)}$ be defined by twisting each entry. That is, $(b_{ij})^{(n)} = (b_{ij}^{(n)})$. For $\ell \geqslant 1$ we let $\Mat_\ell(R)[\tau] = \Mat_\ell(R[\tau])$ be the ring of \emph{twisted polynomials} in $\tau$ with coefficients in $\Mat_\ell(R)$, subject to the relation $\tau B = B^{(1)} \tau$ for $B \in \Mat_\ell(R)$.
In this way, $R^{\ell}$ is a left $\Mat_{\ell}(R)[\tau]$-module, where if $\beta = B_0 + B_1 \tau + \dots + B_m \tau^m \in \Mat_{\ell}(R)[\tau]$ and $\bx \in R^{\ell}$, then
\begin{equation} \label{E:taueval}
\beta(\bx) = B_0 \bx + B_1 \bx^{(1)} + \cdots + B_m \bx^{(m)}.
\end{equation}

If furthermore $\tau$ is an automorphism of $R$, then we set $\sigma \assign \tau^{-1}$ and form the twisted polynomial ring $\Mat_{\ell}(R)[\sigma]$, subject to $\sigma B = B^{(-1)} \sigma$ for $B \in \Mat_{\ell}(R)$.
Then $R^{\ell}$ is a left $\Mat_{\ell}(R)[\sigma]$-module, where for $\gamma = C_0 + C_1\sigma + \dots + C_m\sigma^m \in \Mat_{\ell}(R)[\sigma]$ and $\bx \in R^{\ell}$,
\begin{equation} 
\gamma(\bx) = C_0 \bx + C_1 \bx^{(-1)} + \cdots + C_m \bx^{(-m)}.
\end{equation}

For $\beta \in \Mat_{\ell}(R)[\tau]$ (or $\gamma \in \Mat_{\ell}(R)[\sigma]$), we write $\pd \beta$ (or $\pd \gamma$) for the constant term with respect to $\tau$ (or $\sigma$). We have natural inclusions of $\FF_q$-algebras,
\[
\Mat_{\ell}(R)[\tau] \subseteq \power{\Mat_{\ell}(R)}{\tau}, \quad \Mat_{\ell}(R)[\sigma] \subseteq \power{\Mat_{\ell}(R)}{\sigma},
\]
into \emph{twisted power series} rings, the latter when $\tau$ is an automorphism.

\subsubsection{Ore anti-involution} 
We assume that $\tau : R \to R$ is an automorphism, and recall the anti-isomorphism $* : R[\tau] \to R[\sigma]$ of $\FF_q$-algebras originally defined by Ore~\cite{Ore33a} (see also~\cites[\S 1.7]{Goss}[\S 2.3]{NamoijamP22}{Poonen96}),  given by
\[
\biggl( \sum_{i=0}^\ell b_i\tau^i \biggr)^* = \sum_{i=0}^{\ell} b_i^{(-i)} \sigma^i.
\]
One verifies that $(\alpha\beta)^* = \beta^*\alpha^*$ for $\alpha$, $\beta \in R[\tau]$. For $B = (\beta_{ij}) \in \Mat_{k \times \ell}(R[\tau])$, we set
\[
B^* \assign \bigl( \beta_{ij}^* \bigr)^{\tr} \in \Mat_{\ell \times k}(R[\sigma]),
\]
which then satisfies
\begin{equation} \label{E:BCstar}
(BC)^* = C^*B^* \in \Mat_{m \times k}(R[\sigma]), \quad B \in \Mat_{k \times \ell}(R[\tau]),\ C \in \Mat_{\ell \times m}(R[\tau]).
\end{equation}
The inverse of $* : \Mat_{k \times \ell}(R[\tau]) \to \Mat_{\ell \times k}(R[\sigma])$ is also denoted by ``$*$.''

\subsubsection{Orders of finite 
\texorpdfstring{$F[x]$}{F[x]}-modules} \label{SSS:orders}
For $F[x]$ a polynomial ring in one variable over a field~$F$, we say that an $F[x]$-module is \emph{finite} if it is finitely generated and torsion.
Now fix a finite $F[x]$-module $M$. Then there are monic polynomials $f_1, \dots, f_\ell \in F[x]$ so that
\[
M \cong F[x]/(f_1) \oplus \dots \oplus F[x]/(f_\ell).
\]
We set $\Aord{M}{F[x]} \assign f_1 \cdots f_\ell \in F[x]$, which is a generator of the Fitting ideal of $M$, and we call $\Aord{M}{F[x]}$ the \emph{$F[x]$-order} of $M$. If $m_x : M \to M$ is left-multiplication by~$x$, then
\[
\Aord{M}{F[x]} = \Char(m_x,M,X)|_{X=x},
\]
where $\Char(m_x,M,X) \in F[X]$ is the characteristic polynomial of $m_x$ as an $F$-linear map.
For a variable $y$ independent from $x$, but $M$ still an $F[x]$-module, we will write
\[
\Aord{M}{F[y]} \assign \Aord{M}{F[x]}|_{x=y} = \Char(m_x,M,y).
\]
This will be of particular use for us when $M$ is an $\sA$-module (or $\bA$-module), where
\[
\Aord{M}{A} = \Aord{M}{\sA}|_{t=\theta} = \Char(m_t,M,\theta) \in A, \ \textup{(or $\Aord{M}{\AAA} = \Aord{M}{\bA}|_{t=\theta} = \Char(m_t,M,\theta) \in \AAA$),}
\]
coercing $A$-orders and $\AAA$-orders to be elements of our scalar fields.

\subsection{Drinfeld modules, Anderson \texorpdfstring{$t$}{t}-modules, and their adjoints} \label{SS:Drinfeld}
Given a field $F\supseteq \FF_q$ and an $\FF_q$-algebra map $\iota : \sA \to F$, we call $F$ an \emph{$\sA$-field}.  The kernel of $\iota$ is the \emph{characteristic} of $F$, and if $\iota$ is injective then the characteristic is \emph{generic}. If $F \subseteq \C$ has generic characteristic, then we always assume that $\iota(t) = \theta$. Otherwise, $\iota(t) \rassign \otheta \in F$.

\subsubsection{Drinfeld modules and Anderson \texorpdfstring{$t$}{t}-modules}
A \emph{Drinfeld module} over $F$ is defined by an $\FF_q$-algebra homomorphism $\phi : \sA \to F[\tau]$ such that
\begin{equation} \label{E:Drindef}
\phi_t = \otheta + \kappa_1 \tau + \dots + \kappa_r \tau^r, \quad \kappa_r \neq 0.
\end{equation}
We say that $\phi$ has \emph{rank $r$}. We then make $F$ into an $\sA$-module by setting
\[
t \cdot x \assign \phi_t(x) = \otheta x + \kappa_1 x^q + \cdots + \kappa_r x^{q^r}, \quad x \in F.
\]
Similarly an \emph{Anderson $t$-module of dimension~$\ell$} over $F$ is defined by an $\FF_q$-algebra homomorphism $\psi : \sA \to \Mat_{\ell}(F)[\tau]$ such that
\begin{equation} \label{E:tmoddef}
\psi_t = \pd \psi_t + E_1 \tau + \dots + E_{w} \tau^{w}, \quad E_i \in \Mat_{\ell}(F),
\end{equation}
where $\pd \psi_t - \otheta\cdot \rI_{\ell}$ is nilpotent. A Drinfeld module is then a $t$-module of dimension~$1$. 
We write $\psi(F)$ for $F^{\ell}$ with the $\sA$-module structure given by $a \cdot \bx \assign \psi_a(\bx)$ through~\eqref{E:taueval}. Similarly, we write $\Lie(\psi)(F)$ for~$F^{\ell}$ with $F[t]$-module structure defined by $\pd \psi_a$ for $a \in \sA$. For $a \in \sA$, the $a$-torsion submodule of $\psi(\oF)$ is denoted
\[
\psi[a] \assign \{ \bx \in \smash{\oF}^{\ell} \mid \psi_a(\bx) = 0 \}.
\]

Given $t$-modules $\phi : \sA \to \Mat_{k}(F)[\tau]$, $\psi : \sA \to \Mat_{\ell}(F)[\tau]$, a morphism $\eta : \phi \to \psi$ is a matrix $\eta \in \Mat_{\ell \times k}(F[\tau])$ such that
$\eta \phi_a = \psi_a \eta$ for all $a \in \sA$.
Moreover, $\eta$ induces an $\sA$-module homomorphism $\eta : \phi(F) \to \psi(F)$, and we have a functor $\psi \mapsto \psi(F)$ from the category of $t$-modules to $\sA$-modules. We also have an induced map of $F[t]$-modules, $\pd \psi : \Lie(\phi)(F) \to \Lie(\psi)(F)$.

Anderson defined $t$-modules in~\cite{And86}, and following his language we sometimes abbreviate ``Anderson $t$-module'' by ``$t$-module.'' For more information about Drinfeld modules and $t$-modules see \cite{Goss,Thakur}.

\subsubsection{Exponential and logarithm series} 
Suppose now that $F \subseteq \C$ has generic characteristic and that $\psi$ is defined over $F$. Then there is a twisted power series $\Exp_{\psi} \in \power{\Mat_{\ell}(F)}{\tau}$, called the \emph{exponential series} of~$\psi$, such that
\[
\Exp_{\psi} = \sum_{i=0}^{\infty} B_i \tau^i, \quad B_0 = \rI_{\ell},\ B_i \in \Mat_{\ell}(F),
\]
and for all $a \in \sA$, $\Exp_{\psi} {}\cdot \pd\psi_a = \psi_a \cdot \Exp_{\psi}$.
This functional identity for $a=t$ induces a recursive relation that uniquely determines $\Exp_{\psi}$. That the coefficient matrices have entries in $F$ is due to Anderson \cite[Prop.~2.1.4, Lem.~2.1.6]{And86}. 
The exponential series induces an $\FF_q$-linear and entire function,
\[
\Exp_{\psi} : \C^{\ell} \to \C^{\ell}, \quad \Exp_{\psi}(\bz) = \sum_{i=0}^{\infty} B_i \bz^{(i)}, \quad \bz \assign (z_1, \dots, z_{\ell})^{\tr},
\]
called the \emph{exponential function} of $\psi$. That $\Exp_{\psi}$ converges everywhere is equivalent to
\[\lim_{i \to \infty} \inorm{B_i}^{1/q^i} = 0 \Leftrightarrow \lim_{i \to \infty} \deg (B_i)/q^i = -\infty.\]
We also identify the exponential function with the $\FF_q$-linear formal power series $\Exp_\psi(\bz) \in \power{\C}{\bz}^\ell$. 
The functional equation for $\Exp_\psi$ induces the identities,
\[
\Exp_\psi (\pd \psi_a \bz) = \psi_a \bigl( \Exp_\psi(\bz) \bigr), \quad \forall\, a \in \sA.
\]
The exponential function of $\psi$ is always surjective for Drinfeld modules, but it may not be surjective when $\ell \geqslant 2$. We say that $\psi$ is \emph{uniformizable} if $\Exp_{\psi} : \C^{\ell} \to \C^{\ell}$ is surjective. The kernel of $\Exp_{\psi} \subseteq \C^{\ell}$,
\[
\Lambda_\psi \assign \ker \Exp_{\psi},
\]
is a finitely generated and discrete $\pd\psi(\sA)$-submodule of $\C^{\ell}$ called the \emph{period lattice} of~$\psi$.
Thus if $\psi$ is uniformizable, then we obtain an exact sequence of $\sA$-modules,
\[
0 \to \Lambda_{\psi} \to \C^{\ell} \xrightarrow{\Exp_{\psi}} \psi(\C) \to 0.
\]
As an element of $\power{\Mat_{\ell}(F)}{\tau}$ the series $\Exp_{\psi}$ is invertible, and we let \[\Log_{\psi} \assign \Exp_{\psi}^{-1} \in \power{\Mat_{\ell}(F)}{\tau}\] be the \emph{logarithm series} of~$\psi$, satisfying
\[
\Log_{\psi} = \sum_{i=0}^{\infty} C_i \tau^i, \quad C_0 = \rI_{\ell},\ C_i \in \Mat_{\ell}(F).
\]
Together with the \emph{logarithm function}, $\Log_\psi(\bz) = \sum_{i \geqslant 0} C_i \bz^{(i)} \in \power{\C}{\bz}^{\ell}$, we have
$\pd \psi_a \cdot \Log_{\psi} = \Log_{\psi} {} \cdot \psi_a$ and $\pd\psi_a ( \Log_{\psi}(\bz) ) = \Log_{\psi} (\psi_a (\bz))$, for all $a \in \sA$.
In general $\Log_\psi(\bz)$ converges only on an open polydisc in $\C^{\ell}$. For example, if $\phi : \sA \to \C[\tau]$ is a Drinfeld module as in \eqref{E:Drindef}, then $\Log_{\phi}(z)$ converges on the open disk of radius $R_\phi$, where
\begin{equation} \label{E:Rphi}
R_{\phi} = \inorm{\theta}^{-\max \{ (\deg \kappa_i - q^i)/(q^i-1)\, \mid \, 1 \leqslant i \leqslant r,\, \kappa_i \neq 0 \}}
\end{equation}
 (see \cites[Rem.~6.11]{EP14}[Cor.~4.5]{KhaochimP22}).

\subsubsection{Adjoints of \texorpdfstring{$t$}{t}-modules}
Assume now that $F$ is a perfect $\sA$-field and that $\psi : \sA \to \Mat_{\ell}(F)[\tau]$ is an Anderson $t$-module over $F$ defined as in \eqref{E:tmoddef}. The \emph{adjoint} of $\psi$ is defined to be the $\FF_q$-algebra homomorphism $\psi^* : \sA \to \Mat_{\ell}(F)[\sigma]$ defined by
\[
\psi_a^* \assign (\psi_a)^*, \quad \forall\, a \in \sA.
\]
Since for $a$, $b \in \sA$ we have $\psi_{ab} = \psi_a\psi_b = \psi_b \psi_a$, \eqref{E:BCstar} implies that $\psi^*$ respects multiplication, which is the nontrivial part of checking that $\psi^*$ is an $\FF_q$-algebra homomorphism. From~\eqref{E:tmoddef}, we have
\[
\psi^*_t = (\psi_t)^* = (\pd \psi_t)^{\tr} + \bigl(E_1^{(-1)} \bigr)^{\tr} \sigma + \cdots + \bigl( E_w^{(-w)} \bigr)^{\tr} \sigma^{w},
\]
and so for any $\bx \in F^{\ell}$, we have $\psi^*_t(\bx) = (\pd \psi_t)^{\tr} \bx + (E_1^{(-1)})^{\tr} \bx^{(-1)} + \cdots + (E_w^{(-w)})^{\tr} \bx^{(-w)}$.
In this way the map $\psi^*$ induces an $\sA$-module structure on $F^{\ell}$, which we denote $\psi^*(F)$. Similarly we denote $\Lie(\psi^*)(F) = F^{\ell}$ with an $F[t]$-module structure induced by $\pd \psi_a^{\tr}$ for $a \in \sA$. For $a \in \sA$, the $a$-torsion submodule of $\psi^*(\oF)$ is denoted
\[
\psi^*[a] \assign \{ \bx \in \smash{\oF}^{\ell} \mid \psi^*_a(\bx) = 0\}.
\]

If $\eta : \phi \to \psi$ is a morphism of $t$-modules as above, then $\eta^* \in \Mat_{k \times \ell}(F)[\sigma]$ provides a morphism $\eta^* : \psi^* \to \phi^*$ such that $\eta^* \psi_a^* = \phi_a^* \eta^*$ for all $a \in \sA$ (and vice versa). Furthermore, $\pd\eta^* : \Lie(\psi^*)(F) \to \Lie(\phi^*)(F)$ is an $F[t]$-module homomorphism. Adjoints of Drinfeld modules were investigated extensively by Goss~\cite[\S 4.14]{Goss} and Poonen~\cite{Poonen96}.

\subsection{\texorpdfstring{$t$}{t}-motives and dual \texorpdfstring{$t$}{t}-motives}

For this subsection we fix a perfect $\sA$-field $F$ and $t$-module $\psi : \sA \to \Mat_{\ell}(F)[\tau]$ as in~\eqref{E:tmoddef}. Recall that $\otheta = \iota(t) \in F$.

\subsubsection{\texorpdfstring{$t$}{t}-motive of \texorpdfstring{$\psi$}{psi}}
We let $\cM_{\psi} \assign \Mat_{1\times \ell}(F[\tau])$, and make $\cM_{\psi}$ into a left $F[t,\tau]$-module by using the inherent structure as a left $F[\tau]$-module and setting
\[
a \cdot m \assign m \psi_a, \quad m \in \cM_{\psi},\ a \in \sA.
\]
Then $\cM_{\psi}$ is called the \emph{$t$-motive} of $\psi$. We note that for any $m\in \cM_{\psi}$,
\[
(t-\otheta)^{\ell} \cdot m \in \tau \cM_{\psi},
\]
since $\pd\psi_t-\otheta \rI_{\ell}$ is nilpotent (and $F$ is perfect). If we need to emphasize the dependence on the base field~$F$, we write
\[
\cM_{\psi}(F) \assign \cM_{\psi} = \Mat_{1 \times \ell}(F[\tau]).
\]

A morphism $\eta : \phi \to \psi$ of $t$-modules over $F$ of dimensions~$k$ and $\ell$, defined as in \S\ref{SS:Drinfeld}, induces a morphism of left $F[t,\tau]$-modules $\eta^{\dagger} : \cM_{\psi} \to \cM_{\phi}$, given by $\eta^{\dagger} (m) \assign m\eta$ for $m \in \cM_{\psi}$. The functor from $t$-modules over $F$ to $t$-motives over $F$ is fully faithful, and so every left $F[t,\tau]$-module homomorphism $\cM_{\psi} \to \cM_{\phi}$ arises in this way.

By construction $\cM_{\psi}$ is free of rank~$\ell$ as a left $F[\tau]$-module, and we say $\ell$ is the \emph{dimension} of $\cM_{\psi}$. If $\cM_{\psi}$ is further free of finite rank over $F[t]$, then $\cM_{\psi}$ is said to be \emph{abelian} and $r = \rank_{F[t]} \cM_{\psi}$ is the \emph{rank} of $\cM_{\psi}$. We will say that $\psi$ is abelian or has rank~$r$ if $\cM_{\psi}$ possesses the corresponding properties. The $t$-motives in Anderson's original definition in~\cite{And86} are abelian, as will be most of the $t$-motives in this paper, but for example, see~\cites{BP20}[Ch.~5]{Goss}[Ch.~2--4]{NamoijamP22}{HartlJuschka20}~ for $t$-motives in this wider context.

\subsubsection{Dual \texorpdfstring{$t$}{t}-motive of \texorpdfstring{$\psi$}{psi}}
We let $\cN_{\psi} \assign \Mat_{1 \times \ell}(F[\sigma])$, and similar to the case of $t$-motives, we define a left $F[t,\sigma]$-module structure on $\cN_{\psi}$ by setting
\[
a \cdot n \assign n \psi_a^*, \quad n \in \cN_{\psi},\ a \in \sA.
\]
The module $\cN_{\psi}$ is the \emph{dual $t$-motive} of $\psi$. As in the case of $t$-motives, for any $n \in \cN_{\psi}$ we have $(t-\otheta)^{\ell} \cdot n \in \sigma\cN_{\psi}$. Also if we need to emphasize the dependence on~$F$, we write
\[
\cN_{\psi}(F) \assign \cN_{\psi} = \Mat_{1\times \ell}(F[\sigma]).
\]

Again for a morphism $\eta : \phi \to \psi$ of $t$-modules of dimensions $k$ and $\ell$, we obtain a morphism of left $F[t,\sigma]$-modules, $\eta^{\ddagger} : \cN_{\phi} \to \cN_{\psi}$, given by $\eta^{\ddagger}(n) \assign n \eta^*$ for $n \in \cN_{\phi}$. Also, every morphism of left $F[t,\sigma]$-modules $\cN_{\phi} \to \cN_{\psi}$ arises in this way.

The dual $t$-motive $\cN_{\psi}$ is free of rank $\ell$ as a left $F[\sigma]$-module, and $\ell$ is the \emph{dimension} of $\cN_{\psi}$. If $\cN_{\psi}$ is free of finite rank over $F[t]$, then we say $\cN_{\psi}$ is \emph{$\sA$-finite}, and we call $r = \rank_{F[t]}(\cN_{\psi})$ the \emph{rank} of $\cN_{\psi}$. It has been shown by Maurischat~\cite{Maurischat21} that for a $t$-module $\psi$, the $t$-motive $\cM_{\psi}$ is abelian if and only if the dual $t$-motive $\cN_{\psi}$ is $\sA$-finite. In this case the rank of $\cM_{\psi}$ is the same as the rank of $\cN_{\psi}$. We will say that $\psi$ is $\sA$-finite or has rank~$r$ if $\cN_{\psi}$ has those properties.
Dual $t$-motives were initially introduced in \cite{ABP04} over fields of generic characteristic. See \cites{BP20}{HartlJuschka20}{Maurischat21}[Ch.~2--4]{NamoijamP22}), for more information.

We call $\bm = (m_1, \dots, m_r)^{\tr} \in \Mat_{r\times 1}(\cM_{\psi}(F))$ a \emph{basis} of $\cM_{\psi}(F)$ if $m_1, \dots, m_r$ form an $F[t]$-basis of $\cM_{\psi}(F)$. Likewise $\bn = (n_1, \dots, n_r)^{\tr} \in \Mat_{r\times 1}(\cN_{\psi}(F))$ is a \emph{basis} of $\cN_{\psi}(F)$ if $n_1, \dots, n_r$ form an $F[t]$-basis of $\cN_{\psi}(F)$. We then define $\Gamma$, $\Phi \in \Mat_r(F[t])$ so that
\[
\tau \bm = \Gamma \bm, \quad \sigma \bn = \Phi \bn.
\]
It follows that $\det \Gamma = c(t-\theta)^\ell$, $\det \Phi = c'(t-\theta)^\ell$,
where $c$, $c' \in F^{\times}$ (e.g., see \cite[\S 3.2]{NamoijamP22}). Then $\Gamma$ \emph{represents multiplication by $\tau$} on $\cM_{\psi}$ and $\Phi$ \emph{represents multiplication by $\sigma$} on $\cN_{\psi}$.

\begin{example} \label{Ex:Carlitz1}
\emph{Carlitz module.} The Carlitz module $\sC : \sA \to F[\tau]$ over $F$ is defined by
\[
\sC_t = \otheta + \tau,
\]
and it has dimension~$1$ and rank~$1$. Then $\bm = \{1\}$ is an $F[t]$-basis for $\cM_{\sC} = F[\tau]$, and $\bn = \{1\}$ is an $F[t]$-basis for $\cN_{\sC} = F[\sigma]$. One finds that $\tau \cdot 1 = (t-\theta)\cdot 1$ in $\cM_{\sC}$ and $\sigma \cdot 1 = (t-\theta)\cdot 1$ in $\cN_{\sC}$, so $\Gamma = \Phi = t-\theta$.
\end{example}

\begin{example} \label{Ex:Drinfeld1}
\emph{Drinfeld modules.} Let $\phi : \sA \to F[\tau]$ be a Drinfeld module over $F$ of rank~$r$ defined as in~\eqref{E:Drindef}. Then $\bm = (1,\tau, \dots, \tau^{r-1})^{\tr}$ is a basis for $\cM_{\psi}$ and $\bn = (1, \sigma, \dots, \sigma^{r-1})^{\tr}$ is a basis for $\cN_{\psi}$. Furthermore, $\tau \bm = \Gamma\bm$ and $\sigma \bn = \Phi \bn$, where
\begin{equation} \label{E:Gammadef}
\Gamma = \begin{pmatrix}
0 & 1 & \cdots & 0 \\
\vdots & \vdots & \ddots & \vdots \\
0 & 0 & \cdots & 1 \\
(t-\otheta)/\kappa_r & -\kappa_1/\kappa_r & \cdots & -\kappa_{r-1}/\kappa_r
\end{pmatrix},
\end{equation}
and $\Phi$ occurs similarly. See \cites[\S 3.3--3.4]{CP12}[Ex.~3.5.14, Ex.~4.6.7]{NamoijamP22}[\S 4.2]{Pellarin08}~ for details.
\end{example}

\subsection{Tate modules and characteristic polynomials for Drinfeld modules}\label{SS:munu}

We fix a Drinfeld module $\phi : \sA \to A[\tau]$ of rank $r$ in generic characteristic, given by
\[
\phi_t = \theta + \kappa_1 \tau + \cdots + \kappa_r \tau^r, \quad \kappa_i \in A,\ \kappa_r \neq 0.
\]
Letting $f \in A_+$ be irreducible of degree~$d$, the reduction of $\phi$ modulo~$f$ is a Drinfeld module $\ophi : \sA \to \FF_f[\tau]$ of rank $r_0 \leqslant r$, where $\FF_f = A/fA$. Then $\phi$ has \emph{good reduction} modulo~$f$ if $r_0=r$ or equivalently if $f \nmid \kappa_r$.

For $\lambda \in \sA_+$ irreducible, we form the $\lambda$-adic \emph{Tate modules},
\[
T_{\lambda}(\phi) \assign \varprojlim \phi[\lambda^m], \quad T_{\lambda}(\ophi) \assign \varprojlim \ophi[\lambda^m].
\]
As an $\sA_{\lambda}$-module, $T_{\lambda}(\phi) \cong \sA_{\lambda}^{r}$, and if $\lambda(\theta) \neq f$, then likewise $T_{\lambda}(\ophi) \cong \sA_{\lambda}^{r_0}$. Fixing henceforth that $\lambda(\theta) \neq f$, we set $P_f(X) \assign \Char(\tau^d,T_{\lambda}(\ophi),X)|_{t=\theta}$ to be the characteristic polynomial of the $q^d$-th power Frobenius acting on $T_{\lambda}(\ophi)$ but, for convenience, with coefficients forced into $A$ (rather than $\sA$). Thus we have
\begin{equation} \label{E:Pfdef}
P_f(X) = X^{r_0} + c_{r_0-1} X^{r_0-1} + \cdots + c_0 \in A[X].
\end{equation}
Takahashi~\cite[Prop.~3]{Takahashi82} showed that the coefficients in~$A$ and are independent of the choice of $\lambda$ (see also Gekeler~\cite[Cor.~3.4]{Gekeler91}).
We note that if $\phi$ has good reduction modulo~$\lambda$ and if $\alpha_f \in \Gal(K^{\sep}/K)$ is a Frobenius element, then (e.g., see \cites[\S 3]{Goss92}[\S 8.6]{Goss})
\[
\Char(\tau^d,T_{\lambda}(\ophi),X) = \Char(\alpha_f,T_{\lambda}(\phi),X) \in \sA[X].
\]

\subsubsection{Properties of \texorpdfstring{$P_f(X)$}{}} \label{SSS:Pf}
The following results are due to Gekeler~\cite[Thm.~5.1]{Gekeler91} and Takahashi~\cite[Lem.~2, Prop.~3]{Takahashi82}.
\begin{itemize}
\item We have $c_0 = c_f^{-1}f$ for some $c_f \in \FF_q^{\times}$.
\item The ideal $(P_f(1)) \subseteq A$ is an Euler-Poincar\'e characteristic for $\ophi(\FF_f)$.
\item The roots $\gamma_1, \dots, \gamma_{r_0}$ of $P_f(x)$ in $\oK$ satisfy $\deg_{\theta} \gamma_i = d/r_0$.
\end{itemize}
Extending these a little further, for $1 \leqslant j \leqslant r_0$, we have $\deg_{\theta} c_{r_0-j} \leqslant jd/r_0$. Additionally,
\begin{equation}
\bigl[ \ophi(\FF_f) \bigr]_{A} = c_f P_f(1)
\end{equation}
by \cite[Cor.~3.2]{CEP18}. Here we use the convention from \S\ref{SSS:orders} that $\Aord{\ophi(\FF_f)}{A} = \Aord{\ophi(\FF_f)}{\sA}|_{t=\theta}$. Following the exposition in~\cite[\S 3]{CEP18}, we let $P_f^{\vee}(X) \assign \Char(\tau^d,T_{\lambda}(\ophi)^{\vee},X)|_{t=\theta}$ be the characteristic polynomial in $K[X]$ of $\tau^d$ acting on the the dual space of $T_{\lambda}(\ophi)$. We let $Q_f(X) = X^{r_0} P_f(1/X)$ and $Q_f^{\vee}(X) = X^{r_0}P_f^{\vee}(1/X)$ be the reciprocal polynomials of $P_f(X)$ and $P_f^{\vee}(X)$, and consider $Q_f^{\vee}(fX) = 1 + c_f c_1 X + c_f c_2 f X^2 + \cdots + c_f c_{r_0-1} f^{r_0-2} X^{r_0-1} + c_f f^{r_0-1}X^{r_0}$. To denote the dependence on $\phi$, we write $P_{\phi,f}(X)$, $Q_{\phi,f}(X)$, etc.

By varying $f$, we use $Q_f^{\vee}(fX)$ and $Q_f(X)$ to define multiplicative functions $\mu_{\phi}$, $\nu_{\phi} : A_+ \to A$ such that on powers of a given~$f$,
\begin{equation} \label{E:munugen}
\sum_{m=1}^{\infty} \mu_{\phi}(f^m) X^m \assign \frac{1}{Q_f^{\vee}(fX)}, \quad
\sum_{m=1}^{\infty} \nu_{\phi}(f^m) X^m \assign \frac{1}{Q_f(X)}.
\end{equation}

\subsubsection{Everywhere good reduction} \label{SSS:everywhere}
Hsia and Yu~\cite{HsiaYu00} have determined precise formulas for~$c_f$ in terms of the $(q-1)$-st power residue symbol. Of particular interest presently is the case that $\phi$ has everywhere good reduction, i.e., when $\kappa_r \in \FF_q^{\times}$. In this case, Hsia and Yu~\cite[Thm.~3.2, Eqs.~(2) \& (8)]{HsiaYu00} showed that $c_f = (-1)^{r + d(r+1)} \kappa_r^d$.
This prompts the definition of a completely multiplicative function $\chi_{\phi} : A_+ \to \FF_q^{\times}$,
\begin{equation} \label{E:chidef}
\chi_{\phi}(a) \assign \bigl((-1)^{r+1} \kappa_r \bigr)^{\deg_{\theta} a},
\end{equation}
for which we see that $c_f = (-1)^r \chi_{\phi}(f)$. Letting $\ochi_{\phi} : A_+ \to \FF_q^{\times}$ be the multiplicative inverse of $\chi_{\phi}$, we see that
\begin{align} \label{E:PfandPfvee}
P_f(X) &= X^r + c_{r-1} X^{r-1} + \cdots + c_1 X + (-1)^r \ochi_{\phi}(f) \cdot f, \\
P_f^{\vee}(X) &= X^r + \frac{(-1)^r \chi_{\phi}(f) c_1}{f}X^{r-1} + \cdots + \frac{(-1)^r \chi_{\phi}(f) c_{r-1}}{f} X + \frac{(-1)^r \chi_{\phi}(f)}{f}, \notag
\end{align}
and likewise
\begin{align} \label{E:QfandQfvee}
Q_f(X) &= 1 + c_{r-1} X + \cdots + c_1 X^{r-1} + (-1)^r \ochi_{\phi}(f) \cdot f X^r, \\
Q_f^{\vee}(fX) &= \begin{aligned}[t]
1 + (-1)^r \chi_{\phi}(f) c_1X + \cdots + (-1)^r &\chi_{\phi}(f) c_{r-1}f^{r-2} X^{r-1}\\
&{}+ (-1)^r \chi_{\phi}(f)f^{r-1} X^r.
\end{aligned}\notag
\end{align}
Moreover,
\begin{equation} \label{E:munuf}
\mu_{\phi}(f) = (-1)^{r+1}\chi_{\phi}(f) c_1, \quad \nu_{\phi}(f) = -c_{r-1}.
\end{equation}
We record the induced recursive relations (cf.~\cite[Lem.~3.5]{CEP18}) on $\mu_{\phi}$ and $\nu_{\phi}$, where taking $m +r \geqslant 1$ and using the convention that $\mu_{\phi}(b) = \nu_{\phi}(b) = 0$ if $b \in K \setminus A_+$,
\begin{align}
\label{E:murec}
\mu_{\phi}(f^{m+r}) &= \begin{aligned}[t]
\mu_{\phi}(f) \mu_{\phi}(f^{m+r-1}) - (-1)^r \chi_{\phi}(f) \sum_{j=2}^{r-1} &c_j f^{j-1} \mu_{\phi}(f^{m+r-j})\\
&{} - (-1)^r \chi_{\phi}(f) f^{r-1} \mu_{\phi}(f^m),
\end{aligned}
\\
\label{E:nurec}
\nu_{\phi}(f^{m+r}) &=
\nu_{\phi}(f) \nu_{\phi}(f^{m+r-1}) - \sum_{j=2}^{r-1} c_{r-j} \nu_{\phi}(f^{m+r-j}) - (-1)^r \ochi_{\phi}(f) f \nu_{\phi}(f^m).
\end{align}

\subsection{Schur polynomials} \label{SS:Schur}
We review properties of symmetric polynomials and especially Schur polynomials. For more details on symmetric polynomials see \cites[Ch.~8]{Aigner}[\S 2.5]{HP22}[Ch.~7]{Stanley}.
Letting $\bx = \{ x_1, \dots, x_n \}$ be independent variables, the \emph{elementary symmetric polynomials} $\{e_i\}_{i=0}^{n} = \{ e_{n,i}\}_{i=0}^n \subseteq \ZZ[\bx]$ are defined by
\begin{equation} \label{E:elesymm}
\sum_{i=0}^n e_i(\bx)T^i = (1+x_1 T)(1+x_2 T) \cdots (1+x_n T).
\end{equation}
We adopt the convention that $e_i = 0$ if $i < 0$ or $i > n$.
The \emph{complete homogeneous symmetric polynomials} $\{h_i\}_{i \geqslant 0} = \{ h_{n,i} \}_{i \geqslant 0} \subseteq \ZZ[x_1,\dots, x_n]$ are defined by
\begin{equation} \label{E:homsymm}
\sum_{i=0}^\infty h_i(\bx) T^i = \frac{1}{(1-x_1 T)(1-x_2 T) \cdots (1-x_n T)},
\end{equation}
and similarly if $i < 0$ then we take $h_i=0$. Then $h_i$ consists of the sum of all monomials in $x_1, \dots, x_n$ of degree~$i$. The \emph{Vandermonde determinant} is
\begin{equation}
V(\bx) = \prod_{1 \leqslant i < j \leqslant n} (x_i - x_j).
\end{equation}
When nonzero we have $\deg e_i=i$ and $\deg h_i=i$, and $\deg V = \binom{n}{2}$.

\begin{definition} \label{D:polytensor}
For polynomials $P(T) = (T-x_1) \cdots (T-x_k)$ and $Q(T) = (T - y_1) \cdots (T-y_\ell)$, we set
\[
(P \otimes Q)(T) \assign \prod_{\substack{1 \leqslant i \leqslant k \\ 1 \leqslant j \leqslant \ell}} (T- x_i y_j).
\]
We further set 
\begin{align*}
\left(\Sym^2P\right)(T) = \prod_{1\leqslant i\leqslant j \leqslant k}(T-x_ix_j)\\
\left(\Alt^2P\right)(T) = \prod_{1\leqslant i < j \leqslant k}(T-x_ix_j).
\end{align*}
Letting $B_m$ be the coefficient of $T^m$ in $(P\otimes Q)(T)$, we find that $B_m$ is symmetric in both $x_1, \dots, x_k$ and $y_1, \dots, y_{\ell}$, its total degree in $x_1, \dots, x_k$ is $k\ell - m$, and its total degree in $y_1, \dots, y_\ell$ is also $k\ell -m$. As such,
$B_m \in \ZZ[e_{k,1}(\bx), \ldots, e_{k,k\ell-m}(\bx); e_{\ell, 1}(\by), \ldots, e_{\ell,k\ell-m}(\by)]$.
The coefficients of $(P\otimes Q)(T)$ and its inverse $(P\otimes Q)(T)^{-1}$ are expressible in terms of Schur polynomials (see Theorem~\ref{T:Cauchy} and Corollary~\ref{C:Cauchynl} for $(P\otimes Q)(T)^{-1}$).
\end{definition}

\subsubsection{Schur polynomials} \label{SSS:Schur}
Let $\lambda$ denote an integer partition $\lambda_1 \geqslant \cdots \geqslant \lambda_n \geqslant 0$ of length~$n$, where $\lambda_i=0$ is allowed. We set
\begin{equation} \label{E:slambda}
s_{\lambda}(\bx) = s_{\lambda_1 \cdots \lambda_n}(\bx) \assign
V(\bx)^{-1} \cdot \det \begin{pmatrix}
x_1^{\lambda_1+n-1} & \cdots & x_n^{\lambda_n+n-1}\\
\vdots & & \vdots \\
x_1^{\lambda_{n-1}+1} & \cdots & x_n^{\lambda_{n-1}+1} \\
x_1^{\lambda_n} & \cdots & x_n^{\lambda_n}
\end{pmatrix}
\end{equation}
We have the following properties (see \cites[\S 8.3]{Aigner}[\S 7.15]{Stanley}).
\begin{itemize}
\item $s_{\lambda}(\bx)$ is a symmetric polynomial in $\ZZ[x_1, \dots, x_n]$.
\item $\deg s_{\lambda}(\bx) = \lambda_1 + \dots + \lambda_n$.
\item For $0 \leqslant i \leqslant n$ we have $s_{\underbrace{\scriptstyle 1\,\,\cdots\,\,1}_{i} \underbrace{\scriptstyle 0\,\,\cdots\,\, 0}_{n-i}}(\bx) = e_i(\bx)$.
\item For $i \geqslant 0$ we have $s_{i \underbrace{\scriptstyle 0\,\,\cdots\,\,0}_{n-1}}(\bx) = h_i(\bx)$.
\end{itemize}
The polynomial $s_{\lambda}$ is called the \emph{Schur polynomial for $\lambda$}. Following the exposition of Bump and Goldfeld~\cites{Bump89}{Goldfeld}, when $n \geqslant 2$ (which we now assume), we consider the subset of Schur polynomials where $\lambda_n=0$ as follows. For integers $k_1, \dots, k_{n-1} \geqslant 0$, form
\[
\lambda : k_1 + \cdots + k_{n-1} \geqslant k_2 + \cdots + k_{n-1} \geqslant \cdots \geqslant k_{n-1} \geqslant 0 \geqslant 0.
\]
We set $S_{k_1, \dots, k_{n-1}}(\bx)$ to be the the Schur polynomial $s_\lambda$, i.e.,
\begin{equation} \label{E:Sk}
S_{k_1, \dots, k_{n-1}}(\bx) \assign
V(\bx)^{-1} \cdot \det \begin{pmatrix}
x_1^{k_1+\cdots+k_{n-1}+n-1} & \cdots & x_n^{k_1+\cdots+k_{n-1}+n-1}\\
x_1^{k_2+\cdots+k_{n-1}+n-2} & \cdots & x_n^{k_2+\cdots+k_{n-1}+n-2}\\
\vdots & & \vdots \\
x_1^{k_{n-1}+1} & \cdots & x_n^{k_{n-1}+1} \\
1 & \cdots & 1
\end{pmatrix}.
\end{equation}
The degree of $S_{k_1, \dots, k_{n-1}}(\bx)$ is $k_1 + 2k_2 + \cdots + (n-1)k_{n-1}$.

\begin{lemma} \label{L:slambdaSk}
Let $\lambda : \lambda_1 \geqslant \cdots \geqslant \lambda_n \geqslant 0$ be an integer partition. Then
\[
s_{\lambda}(\bx) = (x_1 \cdots x_n)^{\lambda_{n}} \cdot S_{\lambda_1-\lambda_2,\lambda_2-\lambda_3, \ldots, \lambda_{n-1}-\lambda_n}(\bx).
\]
\end{lemma}


As a result, we see from the properties of $s_{\lambda}$ above that
\begin{alignat}{2}
\label{E:Sei}
S_{\underbrace{\scriptstyle 0,\,\ldots\,,0,1,0,\,\ldots\,,0}_{\textup{$i$-th place}}}(\bx) &= e_i(\bx), \qquad &&1 \leqslant i \leqslant n-1,\\
\label{E:Shi}
S_{i,\underbrace{\scriptstyle 0,\,\cdots\,, 0}_{n-2}}(\bx) &= h_i(\bx), \qquad &&i \geqslant 0.
\end{alignat}

\begin{lemma} \label{L:Skreorder}
For $k_1, \dots, k_{n-1} \geqslant 0$, we have
\[
(x_1 \cdots x_n)^{k_1 + \dots + k_{n-1}} \cdot S_{k_1, \dots, k_{n-1}} \bigl( x_1^{-1}, \dots, x_n^{-1} \bigr) = S_{k_{n-1}, \dots, k_1}(\bx).
\]
\end{lemma}


\subsubsection{Cauchy's identities}

\begin{theorem}[{Cauchy's Identity, see \cites[Cor.~8.16]{Aigner}[\S 2.2]{Bump89}[Thm.~7.12.1]{Stanley}}] \label{T:Cauchy}
For variables $\bx = \{x_1, \dots, x_n\}$ and $\by = \{y_1, \dots, y_n\}$, let $X=x_1 \cdots x_n$ and $Y=y_1 \cdots y_n$. Then as power series in $\power{\ZZ[\bx,\by]}{T}$,
\[
\prod_{1 \leqslant i,j \leqslant n} (1 - x_i y_j T)^{-1}
= (1 - XY T^n)^{-1} \underset{k=(k_1, \dots, k_{n-1})}{\sum_{k_1=0}^{\infty} \cdots \sum_{k_{n-1}=0}^{\infty}} S_k(\bx) S_k(\by) T^{k_1 + 2k_2 + \cdots + (n-1)k_{n-1}}.
\]
\end{theorem}

If instead we have $\bx = \{x_1, \dots, x_n\}$ and $\by = \{ y_1, \dots, y_{\ell} \}$ with $n < \ell$, then Cauchy's identity reduces to the following result by setting $x_{n+1} = \cdots = x_{\ell} = 0$ and simplifying.

\begin{corollary}[{Bump \cite[\S 2.2]{Bump89}}] \label{C:Cauchynl}
For variables $\bx = \{x_1, \dots, x_n\}$ and $\by = \{y_1, \dots, y_\ell\}$ with $n < \ell$, let $X=x_1 \cdots x_n$. Then as power series in $\power{\ZZ[\bx,\by]}{T}$,
\[
\prod_{\substack{1 \leqslant i \leqslant n \\ 1 \leqslant j \leqslant \ell}} (1-x_i y_j T)^{-1}
=\underset{\substack{k=(k_1, \dots, k_{n-1}) \\ k'=(k_1, \dots, k_{n},0 \dots, 0)}}{\sum_{k_1=0}^{\infty} \cdots \sum_{k_{n}=0}^{\infty}} S_{k}(\bx) S_{k'}(\by) X^{k_n} T^{k_1 +2k_2 + \cdots + n k_n}.
\]
\end{corollary}

\subsubsection{Littlewood's identities}
\begin{theorem}[{Littlewood \cite[(11.9;2), (11.9;4)]{Littlewood}}]\label{T:Littlewood}
For variables $\bx = \{x_1, \dots, x_n\}$, let $X=x_1 \cdots x_n$. Then the following identities hold as power series in $\power{\ZZ[\bx]}{T}$.
\begin{enumerate}
\item For $n\geqslant 2$, we have\[
\prod_{1 \leqslant i\leqslant j \leqslant n} (1 - x_i x_j T)^{-1}
= (1 - X^2 T^n)^{-1} \underset{2\mid k_i \text{ for all $i$}}{\underset{k=(k_1, \dots, k_{n-1})}{\sum_{k_1=0}^{\infty} \cdots \sum_{k_{n-1}=0}^{\infty}}} S_k(\bx) T^{k_1 + 2k_2 + \cdots + (n-1)k_{n-1}}.
\]
\item For $n\geqslant 2$, we have
\begin{equation*}
\prod_{1 \leqslant i< j \leqslant n} (1 - x_i x_j T)^{-1}
= 
\displaystyle (1 - X T^{n/2})^{-\varepsilon}\underset{k_i=0 \text{ if $2\nmid i$}}{\underset{k = (k_1,\dots,k_{n-1})}{\sum_{k_1,\dots,k_{n-1}\in\ZZ_+}}} S_k(\bx) T^{k_2 + 2k_4 + \cdots + \frac{n-1}{2}k_{n-1}},
\end{equation*}
where $\varepsilon = \frac{(-1)^n+1}{2}$.
\end{enumerate}
\end{theorem}

\subsection{The function \texorpdfstring{$\bsmu_{\phi}$}{mu\_\{phi\}}} \label{SS:bsmubsnu}
We review the function $\bsmu_{\phi}$ and its properties explored in \cite[\S 6.1]{HP22} by Papanikolas and the author. They in fact defined the function $\bsmu_{\phi}$ and its ``dual'' version $\bsnu_{\phi}$. For the purpose of the present paper, we only list the properties for the function $\bsmu_{\phi}$.

Let $f \in A_+$ be irreducible, and let $P_{\phi,f}(X)$ and $P_{\phi,f}^{\vee}(X)$ be defined as in~\eqref{E:PfandPfvee}. We let $\alpha_1, \dots, \alpha_r \in \oK$ be the roots of $P_{\phi,f}^{\vee}(X)$. For $k_1, \dots, k_{r-1} \geqslant 0$, we define
\begin{align} \label{E:bsmudef}
\bsmu_{\phi}\bigl( f^{k_1}, \dots, f^{k_{r-1}} \bigr) &\assign
S_{k_1, \dots, k_{r-1}} ( \alpha_1, \dots, \alpha_r ) \cdot f^{k_1 + \cdots + k_{r-1}}, 
\end{align}
where $S_{k_1, \dots, k_{r-1}}$ is the Schur polynomial of \eqref{E:Sk}. We note that by~\eqref{E:QfandQfvee} and~\eqref{E:Sei},
\begin{align}
Q_{\phi,f}^{\vee}(fX) &= \begin{aligned}[t]
1 &{}- \bsmu_{\phi}(f,1,\ldots,1) X + \bsmu_{\phi}(1,f,1,\ldots,1) f X^2 \\
&{}+ \cdots + (-1)^{r-1}\bsmu_{\phi}(1,\ldots,1,f) f^{r-2} X^{r-1} + (-1)^r \chi_{\phi}(f) f^{r-1}X^r.
\end{aligned}
\end{align}

We then extend $\bsmu_{\phi}$ uniquely to functions on $(A_+)^{r-1}$, by requiring that if $a_1, \dots, a_{r-1}$, $b_1, \dots, b_{r-1} \in A_+$ satisfy $\gcd(a_1 \cdots a_{r-1}, b_1\cdots b_{r-1}) = 1$, then
\begin{align*}
\bsmu_{\phi}(a_1b_1, \dots, a_{r-1}b_{r-1}) &= \bsmu_{\phi}(a_1, \dots, a_{r-1}) \bsmu_{\phi}(b_1, \dots, b_{r-1}).
\end{align*}

\begin{proposition}[{\cite[Prop.~6.1.5]{HP22}}] \label{P:bsmunuprops}
For $a$, $a_1, \dots, a_{r-1} \in A_+$, the following hold.
\begin{enumerate}
\item $\bsmu_{\phi}(a_1, \dots, a_{r-1}) \in A$.
\item $\bsmu_{\phi}(a,1, \dots, 1) = \mu_{\phi}(a)$.
\item We have
\begin{align*}
\deg_{\theta} \bsmu_{\phi}(a_1, \dots, a_{r-1}) &\leqslant \frac{1}{r} \bigl( (r-1)\deg_{\theta} a_1 + (r-2) \deg_{\theta} a_2 + \cdots + \deg_{\theta} a_{r-1} \bigr).
\end{align*}
\end{enumerate}
\end{proposition}

We also list some recursive relations of the function $\bsmu_{\phi}$ induced by relations on Schur polynomials (\cite[(6.1.7), (6.1.9), (6.1.11)]{HP22}, cf.~\cite[p.~278]{Goldfeld}). Fix $f \in A_+$ irreducible, and for $k$, $k_1, \dots, k_{r-1} \geqslant 0$,
\begin{align}
\bsmu_{\phi} \bigl(f^k, &\,1, \dots, 1\bigr) \bsmu_{\phi} \bigl( f^{k_1}, \dots, f^{k_{r-1}} \bigr) \\
&= \begin{aligned}[t]
\smash{\sum_{\substack{\substack{m_0 + \cdots + m_{r-1}=k \\ m_1 \leqslant k_1,\, \ldots,\, m_{r-1} \leqslant k_{r-1}}}}} \bsmu_{\phi} \bigl( f^{k_1 + m_0 - m_1}, f^{k_2 + m_1-m_2}, \dots, &\,f^{k_{r-1}+m_{r-2}-m_{r-1}} \bigr) \\
&{} \cdot \chi_{\phi}(f)^{m_{r-1}} f^{k-m_0}.
\end{aligned} \notag 
\end{align}
For $0 \leqslant k \leqslant r-1$,
\begin{align}
\bsmu_{\phi} &\bigl(\underbrace{1,\ldots,1,f,1,\ldots,1}_{\textup{$k$-th place}}) \bsmu_{\phi} \bigl( f^{k_1}, \dots, f^{k_{r-1}} \bigr) \\
&= \begin{aligned}[t]
\smash{\sum_{\substack{m_0 + \cdots + m_{r-1} = k \\ (m_0, \ldots, m_{r-1})\, \in\, \cI_{k_1, \dots, k_{r-1}}}}} \bsmu_{\phi} \bigl( f^{k_1 + m_0 - m_1}, f^{k_2 + m_1-m_2}, \dots, &\,f^{k_{r-1}+m_{r-2}-m_{r-1}} \bigr) \\
&{} \cdot \chi_{\phi}(f)^{m_{r-1}} f^{1-m_0}.
\end{aligned} \notag 
\end{align}
In particular for $k \geqslant 1$ (cf.\ \cite[p.~278]{Goldfeld}),
\begin{align}
\bsmu_{\phi}( f^k, 1, \dots, 1) \bsmu_{\phi} (f,1,\ldots, 1)
&= \begin{aligned}[t]
\bsmu_{\phi}(f^{k+1},&\,1,\ldots, 1) \\
&{}+ \bsmu_{\phi}(f^{k-1},f,1, \ldots, 1) \cdot f,
\end{aligned}
\\
\bsmu_{\phi}( f^k, 1, \dots, 1) \bsmu_{\phi} (1,f,1,\ldots, 1)
&= \begin{aligned}[t]
\bsmu_{\phi}(f^k,&\,f,1,\ldots,1) \\
&{}+ \bsmu_{\phi}(f^{k-1},1,f,1, \ldots,1) \cdot f,
\end{aligned}
\notag \\
\bsmu_{\phi}( f^k, 1, \dots, 1) \bsmu_{\phi} (1,1,f,1,\ldots, 1)
&= \begin{aligned}[t]
\bsmu_{\phi}(f^k,&\,1,f,1,\ldots,1) \\
&{}+ \bsmu_{\phi}(f^{k-1},1,1,f,1, \ldots,1) \cdot f,
\end{aligned}
\notag \\
\bsmu_{\phi}( f^k, 1, \dots, 1) \bsmu_{\phi} (1,\ldots, 1,f)
&= \begin{aligned}[t]
\bsmu_{\phi}(f^k,&\,1,\ldots,1,f) \\
&{}+ \bsmu_{\phi}(f^{k-1},1, \ldots,1) \cdot \chi_{\phi}(f) f.
\end{aligned}
\notag
\end{align}

\subsection{Matrix operations}\label{S:matrixop}
Fixing a subring $R\subseteq\LL_t$ with $1$, we say $M\in\Mat_{r}(R)$ \emph{represents an $R$-module homomorphism $f:R^r\to R^r$ with respect to a basis $\bv=(v_1,\dots,v_r)^\tr\in \Mat_{r\times 1}(R^r)$} if 
\[f\cdot \bv\assign\begin{pmatrix}
    f(v_1)\\
    \vdots\\
    f(v_r)
\end{pmatrix}=M\bv.\]

\begin{remark}\label{R:multi}
This is slightly different from the usual sense in linear algebra. For example, if we let $M_f$, $M_g\in \Mat_r(R)$ represent two $R$-module homomorphisms $f$, $g:R^r\to R^r$ respectively, then $M_gM_f$ represents $f\circ g$. In fact $M_f$ is the transpose of the matrix representation of $f$ in the usual sense.
\end{remark}

For $i$, $j=1,\dots,r$, we define, assuming characteristic of $R$ is not $2$, 
\[\alpha_{ij} = v_i\otimes v_j,\quad \beta_{ij} = \frac{1}{2}\rbracket{v_i\otimes v_j+v_j\otimes v_i},\quad \gamma_{ij}=\frac{1}{2}(v_i\otimes v_j-v_j\otimes v_i),\]
and consider the following basis of $R$-modules in lexicographical order 
\begin{equation}\label{eq:basis}
\mathfrak{Y}_1\assign\{\alpha_{ij}\}_{i,j}\subseteq \rbracket{R^r}^{\otimes 2},\quad \mathfrak{Y}_2\assign\{\beta_{ij}\}_{i,j}\subseteq \Sym^2\rbracket{R^r},\quad \mathfrak{Y}_3\assign\{\gamma_{ij}\}_{i,j}\subseteq \Alt^2\rbracket{R^r}.    
\end{equation}

\begin{definition}\label{D:matrixop}
Let $M\in \Mat_r(R)$, and let $\mathfrak{Y_1}$, $\mathfrak{Y_2}$ and $\mathfrak{Y_3}$ as in \eqref{eq:basis}. Then $M$ represents some $R$-module homomorphism $f:R^r\to R^r$. We define the following matrix operations \[M^{\otimes 2}\in\Mat_{r^2}(R), \quad \Sym^2(M)\in\Mat_{\frac{r(r+1)}{2}}(R), \quad \Alt^2(M)\in\Mat_{\frac{r(r-1)}{2}}(R)\] to be matrices representing $f^{\otimes 2}$, $\Sym^2(f)$ and $\Alt^2(f)$ with respect to $\mathfrak{Y_1}$, $\mathfrak{Y_2}$ and $\mathfrak{Y_3}$, respectively.
\end{definition}

\begin{remark}
In Definition \ref{D:matrixop}, by a direct checking, the matrix $M^{\otimes 2}$ is the Kronecker tensor square of $M$.
\end{remark}

\begin{example}
Suppose $r=2$ and $M=(M_{ij})_{i,j=1,2}\in\Mat_2(R).$ Then 
\begin{enumerate}
    \item $M^{\otimes 2}=\begin{pmatrix}
        M_{11}M_{11}&M_{11}M_{12}&M_{12}M_{11}&M_{12}M_{12}\\
        M_{11}M_{21}&M_{11}M_{22}&M_{12}M_{21}&M_{12}M_{22}\\
        M_{21}M_{11}&M_{21}M_{12}&M_{22}M_{11}&M_{22}M_{12}\\
        M_{21}M_{21}&M_{21}M_{22}&M_{22}M_{21}&M_{22}M_{22}
    \end{pmatrix}.$
    \item $\Sym^2(M) = \begin{pmatrix}
        M_{11}^2& 2M_{11}M_{12}&M_{12}^2\\
        M_{11}M_{21}&M_{11}M_{22}+M_{12}M_{21}&M_{12}M_{22}\\
        M_{21}^2& 2M_{21}M_{22}&M_{22}^2
    \end{pmatrix}.$
    \item $\Alt^2(M)=\det(M).$
\end{enumerate}
\end{example}

\begin{lemma}\label{L:matrixop}
For $M\in\Mat_r(R)$, we let $\mathcal{T}(M)$ denote $M^{\otimes 2}$, $\Sym^2(M)$ or $\Alt^2(M).$ Then for $M_1$, $M_2\in \Mat_r(R)$, we have the following properties.
\begin{enumerate}
    \item $\mathcal{T}(M_1M_2) = \mathcal{T}(M_1)\mathcal{T}(M_2)$. Futhermore, $\mathcal{T}(M^{-1}) = \mathcal{T}(M)^{-1}$ if $M\in\GL_r(R).$
    \item $\mathcal{T}(M^{(n)})=\mathcal{T}(M)^{(n)}$ for $n\geqslant 0.$
\end{enumerate}
\end{lemma}
\begin{proof}
    The first part follows from Remark \ref{R:multi}, and the second part follows from the fact that entries of $\mathcal{T}(M)$ are polynomials in entries of $M$ over $\FF_q$.
\end{proof}

\section{Tensor products of Drinfeld modules}\label{Ch:TensorProducts}

Tensor powers of Drindeld modules were initiated by Anderson \cite{And86} from the aspect of $t$-motives. Later Hamahata \cite{Hamahata} further studied symmetric powers and alternating powers of Drinfeld modules, and provided explicit models as $t$-modules. In this chapter, we recover Hamahata's models from the aspect of $t$-motives.

We fix two Drinfeld modules $\phi$, $\psi : \sA \to A[\tau]$ defined over~$A$ with everywhere good reduction as in \eqref{E:phipsidefintro},
and their $t$-motives $\cM_\phi=\C[\tau]$, $\cM_\psi=\C[\tau]$. 
For convenience, $\kappa_0=\eta_0\assign \theta.$
\subsection{Tensor products of Drinfeld modules defined over \texorpdfstring{$A$}{A}}\label{S:ten}

\emph{The tensor product of $\cM_\phi$ and $\cM_\psi$} is \[\cM_\phi\otimes \cM_\psi\assign \cM_\phi\otimes_{\C[t]}\cM_\psi\]
 equipped with a left $\C[t,\tau]$-module structure by using the $\C[t]$-module structure from the tensor product over $\C[t]$ and setting \[\tau\cdot (a_1\otimes a_2 )\assign \tau a_1\otimes\tau a_2,\quad a_1 \in \cM_\phi,~a_2 \in \cM_\psi.\] 
 In the sense of \cite{And86}, $\cM_\phi\otimes \cM_\psi$ is a \emph{pure} $t$-motive, and its \emph{weight}, denoted by $w(\cM_\phi\otimes \cM_\psi)$, is defined to be \[w(\cM_\phi\otimes \cM_\psi)=\frac{\rank_{\C[\tau]}\cM_\phi\otimes \cM_\psi}{\rank_{\C[t]}\cM_\phi\otimes \cM_\psi}.\] Combining this with \cite[Prop.~1.11.1]{And86}, the dimension of $\cM_\phi\otimes \cM_\psi$ is
\begin{equation*}
\rank_{\C[\tau]}\cM_\phi\otimes \cM_\psi = \left(w(\phi)+w(\psi)\right) \cdot \rank_{\C[t]}\cM_\phi\otimes \cM_\psi=\left(\frac{1}{r}+\frac{1}{\ell}\right)\cdot r\cdot \ell = r+\ell.  
\end{equation*}

More generally, \emph{the $n$-th tensor power of $\cM_\phi$} is 
\[\cM_\phi^{\otimes n} \assign \underbrace{\cM_\phi\otimes_{\C[t]}\cdots\otimes_{\C[t]}\cM_\phi}_{n \text{ times}}\] equipped with a left $\C[t,\tau]$-module by using the $\C[t]$-module structure from the tensor power over $\C[t]$ and setting \[\tau\cdot (a_1\otimes\cdots\otimes a_n )\assign \tau a_1\otimes\cdots\otimes \tau a_n,\quad a_i \in \cM_\phi.\] $\cM_\phi^{\otimes n}$ is a pure $t$-motive of weight 
\begin{equation}\label{w_tenn}
w(\cM_\phi^{\otimes n})=\frac{n}{r}.
\end{equation}

\begin{lemma}[{Khaochim~\cite[Lem.~4.4]{Khaochim}}]\label{ten_basis} The set \[\{s_i\}_{i=1}^{r+\ell} = \left\{1\otimes 1,\dots , 1\otimes\tau^{\ell-1},\tau\otimes 1,\dots,\tau^r\otimes 1\right\}\] is a basis of $\cM_\phi\otimes \cM_\psi$ as a $\C[\tau]$-module.    
\end{lemma}

With the $\C[\tau]$-basis in Lemma~\ref{ten_basis}, we obtain a model $E$ of the tensor product of $\phi$ and $\psi$ which is a $t$-module of dimension $r+\ell$ defined over $A$ by solving the following equation for $E_t:$
\begin{equation}\label{ten_eq}
t\cdot \bu(s_1,\dots,s_{r+\ell})^\tr=\bu E_t(s_1,\dots,s_{r+\ell})^\tr
\end{equation} for all $\bu\in \Mat_{1\times d}(\C[\tau]).$

\begin{definition}\label{ten_def}
Suppose $r\leqslant \ell$. The \emph{tensor product of $\phi$ and $\psi$} is the $t$-module $\phi\otimes\psi:\sA\to\Mat_{r+\ell}(A[\tau])$ of dimension $r+\ell$ defined over $A$ satisfying \eqref{ten_eq}. Explicitly, from \cite[Def. 4.5]{Khaochim} \[(\phi\otimes\psi)_t =\left(\begin{array}{c|c}X_1 & X_2\\ \hline
X_3& X_4    
\end{array}\right),\]
where $X_1\in\Mat_{\ell\times\ell}(A[\tau])$, $X_2\in\Mat_{\ell\times r}(A[\tau])$, $X_3\in\Mat_{r\times\ell}(A[\tau])$, $X_4\in\Mat_{r\times r}(A[\tau])$ are defined by 
\[X_1 = \begin{pmatrix}
\theta & & & & & &\\
\kappa_1\tau & \theta &&&&&\\
\vdots & \ddots & \ddots &&&&\\
\kappa_{r-1}\tau^{r-1} & \cdots & \kappa_1\tau & \theta &&&\\
\kappa_r\tau^r & \cdots & \cdots & \kappa_1\tau & \theta &&\\
& \ddots &&&& \ddots &\\
 & & \kappa_r\tau^r& \cdots & \cdots & \kappa_1\tau & \theta
\end{pmatrix}, ~X_2 = \begin{pmatrix}
\kappa_1 & \cdots & \kappa_{r-1} & \kappa_r\\
\kappa_2 \tau & \cdots & \kappa_r\tau &\\
\vdots & \reflectbox{$\ddots$} &&\\
\kappa_r\tau^{r-1} &&&\\
0&&&\\
\vdots&&&\\
0&&&
\end{pmatrix},\]
\[X_3 = \begin{pmatrix}
\eta_1\tau & \cdots & \cdots & \cdots & \eta_{\ell-1}\tau & \eta_\ell\tau\\
\eta_2\tau^2 & \cdots & \cdots & \cdots & \eta_\ell\tau^2 &\\
\vdots & & & \reflectbox{$\ddots$} &&\\
\eta_r\tau^r & \cdots & \eta_\ell\tau^r &&&
\end{pmatrix},~ X_4 = \begin{pmatrix}
\theta &&&\\    
\eta_1\tau & \theta &&\\   
\vdots & \ddots & \ddots &\\    
\eta_{r-1}\tau^{r-1} & \cdots & \eta_1\tau & \theta
\end{pmatrix}.\]
\end{definition}

\begin{remark}\label{R:teniso}
As pointed out in \cite[Rem. 4.8]{Khaochim}, our model for $\phi\otimes\psi$ is in fact Hamahata's $\psi\otimes\phi$, but they are isomorphic as $t$-modules by \cite[Prop.~2.5]{Hamahata}. 
\end{remark}

\subsection{Symmetric and alternating powers of Drinfeld modules defined over \texorpdfstring{$A$}{A}}

Let $\mathfrak{S}_n$ be the $n$-th symmetric group and $\sgn: \mathfrak{S}_n\to\{\pm 1\}$ be the sign function. \emph{The $n$-th symmetric power of $\cM_\phi$} is the $\C[t,\tau]$-submodule
\begin{equation}\label{eq:defsym} \Sym^n\cM_\phi \assign \Span_{\C[t]}\left\{ \sum_{\sigma\in \mathfrak{S}_n} a_{\sigma(1)}\otimes\cdots\otimes a_{\sigma(n)}\right\}\subseteq \cM_\phi^{\otimes n}.
\end{equation}It is also a pure $t$-motive of weight $w(\Sym^n\cM_\phi) = w(\cM_\phi^{\otimes n})$ by \cite[Prop.~1.10.2]{And86}. Together with \eqref{w_tenn}, the dimension of $\Sym^n\cM_\phi$ is
\begin{equation} \label{d_symn}
\rank_{\C[\tau]}\Sym^n\cM_\phi = \frac{n}{r} \cdot \rank_{\C[t]}\Sym^n\cM_\phi= \frac{n}{r}\cdot \binom{r+n-1}{n} = \binom{r+n-1}{n-1}.
\end{equation}

\emph{The $n$-th alternating power of $\cM_\phi$} is the $\C[t,\tau]$-submodule
\[\Alt^n\cM_\phi \assign \Span_{\C[t]}\left\{ \sum_{\sigma\in \mathfrak{S}_n} \sgn(\sigma) (a_{\sigma(1)}\otimes\cdots\otimes a_{\sigma(n)}) \right\}\subseteq \cM_\phi^{\otimes n}.\] Similar to symmetric powers, it is a pure $t$-motive of weight $w(\Alt^n\cM_\phi) = w(\cM_\phi^{\otimes n})$ and dimension 
\begin{equation} \label{d_altn}
\rank_{\C[\tau]}\Alt^n\cM_\phi = n\cdot w(\phi) \cdot \rank_{\C[t]}\Alt^n\cM_\phi= \frac{n}{r}\cdot \binom{r}{n} = \binom{r-1}{n-1}.
\end{equation}

\begin{lemma}\label{symalt_basis} Suppose $p\neq 2$.
\begin{enumerate}
    \item The set \[\mathfrak{X}_1\assign\left\{1\otimes 1, \frac{1}{2}(1\otimes \tau+\tau\otimes 1),\dots, \frac{1}{2}(1\otimes \tau^r+\tau^r\otimes 1)\right\}\] is a $\C[\tau]$-basis for $\Sym^2\cM_\phi$.
    \item The set \[\mathfrak{X}_2\assign\left\{\frac{1}{2}(1\otimes \tau -\tau\otimes 1),\dots,\frac{1}{2}(1\otimes \tau^{r-1} -\tau^{r-1}\otimes 1) \right\}\] is a $\C[\tau]$-basis for $\Alt^2\cM_\phi$.
\end{enumerate}
\end{lemma}
\begin{proof}
By \eqref{d_symn} and \eqref{d_altn}, it suffices to show that the two sets span the corresponding $t$-motives as $\C[\tau]$-modules.

For the first part, we observe that the elements in $\Sym^2\cM_\phi$ are of the form, $ a_i, b_i\in \cM_\phi=\C[\tau],~ f_i\in\C[t]$,
\begin{align*}
\sum_i& f_i\cdot(a_i\otimes b_i+b_i\otimes a_i) \\
&= \sum_i((f_i\cdot a_i)\otimes b_i+b_i\otimes (f_i\cdot a_i))\\
&=\sum_i\sum_{j_1}\sum_{j_2} h_{i,j_1,j_2}\tau^{\min(j_1,j_2)}(1\otimes \tau^{|j_1-j_2|}+\tau^{|j_1-j_2|}\otimes 1),
\end{align*}
where $h_{i,j_1,j_2}\in\C$. So it suffices to show, for $m\geqslant 0$, 
\begin{equation}\label{ind_sym}
\xi_m\assign 1\otimes \tau^m+\tau^m\otimes 1\in \Span_{\C[\tau]}\mathfrak{X}_1.
\end{equation}
We proceed by induction on $m$. It is clear that \eqref{ind_sym} holds for $m = 0,\dots,r$. For $m>r$, we consider
\begin{equation}\label{express1_sym}
t\cdot \xi_{m-r} = 1\otimes (\tau^{m-r}\phi_t)+(\tau^{m-r}\phi_t)\otimes 1=\sum_{i=0}^r\kappa_i^{(m-r)}\xi_{m-r+i}.
\end{equation}
On the other hand, 
\begin{equation}\label{express2_sym}
\begin{split}
t\cdot \xi_{m-r} &= \phi_t\otimes \tau^{m-r}+\tau^{m-r}\otimes \phi_t\\
&=
\begin{cases}
\displaystyle \sum_{i=0}^r\kappa_i\tau^i\xi_{m-r-i} & \text{if $r\leqslant m-r$}\\
\displaystyle \sum_{i=0}^{m-r}\kappa_i\tau^i\xi_{m-r-i}+\sum_{i=1}^{2r-m}\kappa_{m-r+i}\tau^{m-r}\xi_{i} & \text{if $r > m-r$} 
\end{cases}.
\end{split}
\end{equation}
Combining \eqref{express1_sym} and \eqref{express2_sym}, and using that $\kappa_r\neq 0$, in both cases $\xi_m$ is a $\C[\tau]$-linear combination of $\{\xi_i\}_{i = 0}^{m-1}$, and the result follows by the induction hypothesis. 

For the second part, by a similar observation, it suffices to show, for $m\geqslant 1$, 
\begin{equation}\label{ind_alt}
\zeta_m\assign 1\otimes \tau^m-\tau^m\otimes 1\in \Span_{\C[\tau]}\mathfrak{X}_2.
\end{equation}
We again proceed by induction on $m$. It is clear that \eqref{ind_alt} holds for $m = 1,\dots,r-1$. Note that $\zeta_0 = 0$, so we can not use the same process for the case $m=r$. In fact, it follows by expressing $t\cdot (1\otimes 1)$ in two ways. Note that 

\[t\cdot (1\otimes 1) = \phi_t\otimes 1=1\otimes \phi_t,\]
which gives
\[\zeta_r = -\kappa_r^{-1}\sum_{i=1}^{r-1}\kappa_i\zeta_i.\]
For $m>r$, we consider
\begin{equation}\label{express1_alt}
t\cdot \zeta_{m-r} = 1\otimes (\tau^{m-r}\phi_t)-(\tau^{m-r}\phi_t)\otimes 1=\sum_{i=0}^r\kappa_i^{(m-r)}\zeta_{m-r+i}.
\end{equation}
On the other hand, 
\begin{equation}\label{express2_alt}
\begin{split}
t\cdot \zeta_{m-r} &= \phi_t\otimes \tau^{m-r}-\tau^{m-r}\otimes \phi_t\\
&=
\begin{cases}
\displaystyle \sum_{i=0}^r\kappa_i\tau^i\zeta_{m-r-i} & \text{if $r\leqslant m-r$}\\
\displaystyle \sum_{i=0}^{m-r}\kappa_i\tau^i\zeta_{m-r-i}-\sum_{i=1}^{2r-m}\kappa_{m-r+i}\tau^{m-r}\zeta_{i} & \text{if $r > m-r$} 
\end{cases}.
\end{split}
\end{equation}
Similarly, combining \eqref{express1_alt} and \eqref{express2_alt}, the result follows by the induction hypothesis. 
\end{proof}

In the same fashion as for tensor products, we define symmetric and alternating squares of $\phi$ as follows:

\begin{definition}\label{def_symalt}We assume $p\neq 2$.
\begin{enumerate}
    \item The \emph{symmetric square of $\phi$} is the $t$-module $\Sym^2\phi:\sA\to\Mat_{r+1}(A[\tau])$ of dimension $r+1$ defined over $A$ given by 
    \begin{equation*}
    \begin{split}
    \left(\Sym^2\phi\right)_t &=
    \begin{pmatrix}
        \theta &&&&\\
        \kappa_1\tau & \theta &&&\\
        \kappa_2\tau^2 & \kappa_1\tau & \theta &&\\
        \vdots && \ddots & \ddots &\\
        \kappa_r\tau^r & \cdots & \cdots & \kappa_1\tau & \theta
    \end{pmatrix}+
    \begin{pmatrix}
        0 & \kappa_1 & \cdots & \kappa_{r-1} & \kappa_r\\
        0 & \kappa_2\tau & \cdots & \kappa_r\tau &\\
        \vdots & \vdots & \reflectbox{$\ddots$} &&\\
        0 & \kappa_r\tau^{r-1} &&&\\
        0 &&&&
    \end{pmatrix}.
    \end{split}
    \end{equation*}

    \item The \emph{alternating square of $\phi$} is the $t$-module $\Alt^2\phi:\sA\to \Mat_{r-1}(A[\tau])$ of dimension $r-1$ defined over $A$ given by 
    \begin{equation*}
    \begin{split}
    \left(\Alt^2\phi\right)_t &=
    \begin{pmatrix}
        \theta &&&&\\
        \kappa_1\tau & \theta &&&\\
        \kappa_2\tau^2 & \kappa_1\tau & \theta &&\\
        \vdots && \ddots & \ddots &\\
        \kappa_{r-2}\tau^{r-2} & \cdots & \cdots & \kappa_1\tau & \theta
    \end{pmatrix}-
    \begin{pmatrix}
        \kappa_2\tau & \kappa_3\tau & \cdots & \kappa_{r-1}\tau & \kappa_r\tau\\
        \kappa_3\tau^2 & \kappa_4\tau^2 & \cdots & \kappa_r\tau^2 &\\
        \vdots & \vdots & \reflectbox{$\ddots$} &&\\
        \kappa_{r-1}\tau^{r-2} & \kappa_r\tau^{r-2} &&&\\
        \kappa_r\tau^{r-1} &&&&
    \end{pmatrix}.
    \end{split}
    \end{equation*}   
\end{enumerate}
\end{definition}

In other word, the $\tau$-expansions of $\left(\Sym^2\phi\right)_t$ and $\left(\Alt^2\phi\right)_t$ are
\begin{equation*}
\begin{split}
\left(\Sym^2\phi\right)_t &=B_0+B_1\tau+\cdots+B_r\tau^r,\\
\left(\Alt^2\phi\right)_t &=\theta \rI_{r-1}+C_1\tau+\cdots+C_{r-1}\tau^{r-1},
\end{split}
\end{equation*}
where $B_i = \left(\begin{array}{c|c}0 & 0\\ \hline
B_i'& 0    
\end{array}\right)\in\Mat_{r+1}(A)$ and $C_i = \left(\begin{array}{c|c}0 & 0\\ \hline
C_i'& 0    
\end{array}\right)\in\Mat_{r-1}(A)$
with
\[B_i' = 
\begin{pmatrix}
\kappa_i & \kappa_{i+1} & \cdots & \kappa_r\\
& \kappa_i & \cdots & \kappa_{r-1} \\
&&\ddots & \vdots\\
&&& \kappa_i
\end{pmatrix}, \quad C_i' = 
\begin{pmatrix}
-\kappa_{i+1} & \cdots & -\kappa_{r-1} & -\kappa_r\\
\kappa_i &&&\\
& \ddots &&\\
&& \kappa_i & 
\end{pmatrix},\]
of sizes $(r-i+1)\times (r-i+1)$ and $(r-i)\times (r-i)$, respectively.

\begin{proposition}\label{P:uniformizable}
The tensor structures $\phi^{\otimes 2}$, $\Sym^2(\phi)$ and $\Alt^2(\phi)$ are uniformizable.
\end{proposition}
\begin{proof}
By \cite[Thm.~4]{And86}, it suffices to show that $\cM_E$ is rigid analytically trivial, in the sense of \cite[\S 2.4.6]{HP22}, for $E =\phi^{\otimes 2}$, $\Sym^2(\phi)$ and $\Alt^2(\phi)$. We let $\mathcal{T}(M)$ denote the corresponding matrix operation. Example \ref{Ex:Drinfeld1} provides that $(1,\tau,\dots,\tau^{r-1})^\tr$ is a basis for $\cM_\phi$ with $\Gamma$ in \eqref{E:Gammadef} representing multiplication by $\tau$ on $\cM_\phi$. Then by choosing a basis for $\cM_E$ in the same way as \eqref{eq:basis}, we observe that $\mathcal{T}(\Gamma)$ represents multiplication by $\tau$ on $\cM_E$.

Furthermore, by \cite[\S 4.2]{Pellarin08} (see also \cite[Ex.~2.4.12]{HP22}), the $t$-motive $\cM_\phi$ is rigid analytically trivial with rigid analytic trivialization denoted by $\Upsilon\in \GL_r(\TT_t)$. Then by Lemma~\ref{L:matrixop}, 
\[\mathcal{T}(\Upsilon^\tr)^{(1)}=\mathcal{T}(\rbracket{\Upsilon^\tr}^{(1)})=\mathcal{T}(\Gamma\Upsilon^\tr)=\mathcal{T}(\Gamma)\mathcal{T}(\Upsilon^\tr),\]
which implies $\mathcal{T}(\Upsilon^\tr)^\tr$ is a rigid analytic trivialization for $\cM_E.$
\end{proof}

\begin{remark}\label{R:almoststrictlypure}
By a direct computation, the $\tau$-expansion of $\left(\Sym^2\phi\right)_{t^2}$ is of the form \[\tB_0 +\ \tB_1\tau+\cdots+\tB_{2r}\tau^{2r},\] with $\tB_i = 0$ for $r+1\leqslant i\leqslant 2r$, and where $\tB_r$ is a lower triangular matrix with $\kappa_r^2$ on the diagonal. Similarly, the $\tau$-expansion of $\left(\Alt^2\phi\right)_{t^2}$ is of the form \[\tC_0 +\ \tC_1\tau+\cdots+\tC_{2r}\tau^{2r-2},\] with $\tC_i = 0$ for $r+1\leqslant i\leqslant 2r-2$, and where $\tC_r$ is a lower triangular matrix with $\kappa_r^2$ on the diagonal.  Therefore, the top coefficients of $\left(\Sym^2\phi\right)_{t^2}$ and $\left(\Alt^2\phi\right)_{t^2}$ are invertible ($\kappa_r\neq 0$). A similar computation can also be done for $\rbracket{\phi^{\otimes 2}}_{t^2}$, which gives the same conclusion as symmetric and alternating squares.
So $\phi^{\otimes 2}$, $\Sym^2\phi$ and $\Alt^2\phi$ are \emph{almost strictly pure} in the sense of \cite{NamoijamP22}, which implies pure by \cites[Rem. 2.2.3]{Goss94}[Rem. 5.5.5]{Goss}[Rem. 4.5.3]{NamoijamP22}. This explicates the purity of their $t$-motives.
\end{remark}

\section{Convolution \texorpdfstring{$L$}{L}-series}\label{S:Ltenseries}
In a series of articles {\cites{Goss79}{Goss83}{Goss92}[Ch.~8f]{Goss}} Goss defined and investigated function field valued $L$-series attached to Drinfeld modules and $t$-modules defined over finite extensions of~$K$. These $L$-functions possess a rich structure of special values, initiated by Carlitz~\cite[Thm.~9.3]{Carlitz35} for the eponymous Carlitz zeta function and continued by Goss~\cites{Goss92}[Ch.~8]{Goss}. Anderson and Thakur~\cite{AndThak90} further revealed the connection between Carlitz zeta values and coordinates of logarithms on tensor powers of the Carlitz module.

Taelman \cites{Taelman09}{Taelman10}{Taelman12} discovered a breakthrough on special $L$-values for Drinfeld modules that related them to the product of an analytic regulator and the $A$-order of a class module. These results have been extended in several directions, including to $t$-modules defined over $\oK$ and more refined special value identities~\cites{ANT20}{ANT22}{AnglesTaelman15}{Beaumont21}{FGHP20}{CEP18}{Fang15}{Gezmis21}{GezmisNamoijam21}.

\subsection{Goss \texorpdfstring{$L$}{L}-series} \label{SS:GossL}
Let $\phi : \sA \to A[\tau]$ be a Drinfeld module over $A$ with everywhere good reduction as in~\eqref{E:phipsidefintro}. Goss {\cites[\S 3]{Goss92}[\S 8.6]{Goss}} associated the Dirichlet series
\begin{equation}
L(\phi^{\vee},s) = \prod_{f \in A_+,\ \textup{irred.}} Q_f^{\vee} \bigl( f^{-s} \bigr)^{-1}, \quad
L(\phi,s) = \prod_{f \in A_+, \ \textup{irred.}} Q_f \bigl( f^{-s} \bigr)^{-1}.
\end{equation}
where as in~\S\ref{SS:munu}, $Q_f(X) \in A[X]$ is the reciprocal polynomial of the characteristic polynomial $P_f(X)$ of Frobenius acting on the Tate module $T_{\lambda}(\ophi)$ and $Q_f^{\vee}(X) \in K[X]$ is the reciprocal polynomial of $P_f^{\vee}(X)$ arising from $T_{\lambda}(\ophi)^{\vee}$.
In the future we will write simply ``$\prod_{f}$'' to indicate that a product is over all irreducible $f \in A_+$.

\begin{remark} \label{R:Qfcalc}
\cite[Cor.~3.7.3]{HP22} makes the calculation of $P_f(X)$ and $P_f^{\vee}(X)$ reasonable (and hence $Q_f(X)$ and $Q_f^{\vee}(X)$ also). See \cite[Rmk.~5.1.2]{HP22} for the details.

\end{remark}

The bounds on the coefficients of $P_f(X)$ from \S\ref{SSS:Pf} imply that $L(\phi,s)$ converges in $K_{\infty}$ for $s \in \ZZ_+$ and that $L(\phi^{\vee},s)$ converges for $s \in \ZZ_{\geqslant 0}$ (e.g., see \cite[\S 3]{CEP18}). Goss extended the definition of these $L$-series to $s$ in a non-archimedean analytic space, but we will not pursue these extensions here. We will henceforth assume $s \in \ZZ$.

By \eqref{E:munugen}, we find that
\begin{equation}
L(\phi^{\vee},s) = \sum_{a \in A_+} \frac{\mu_{\phi}(f)}{a^{s+1}}
\end{equation}
(see \cite[Eqs.~(12)--(14)]{CEP18}). In particular, for the Carlitz module $P_{\sC,f}^{\vee}(X) = X - 1/f$, so
\[
L(\sC^{\vee},s) = \sum_{a \in A_+} \frac{1}{a^{s+1}} = \zeta_{\sC}(s+1)
\]
is a shift of the Carlitz zeta function.

Taelman~\cite[Thm.~1]{Taelman12} proved a special value identity for $L(\phi^{\vee},0)$ as follows. First,
\begin{equation} \label{E:Qfvee1}
Q_f^{\vee}(1)^{-1} = \frac{f}{(-1)^r\, \ochi(f) \cdot P_f(1)} = \frac{\Aord{\FF_f}{A}}{\Aord{\ophi(\FF_f)}{A}},
\end{equation}
where the first equality follows from~\eqref{E:QfandQfvee} and the second from Gekeler~\cite[Thm.~5.1]{Gekeler91} (and also from \cite[Cor.~3.7.8]{HP22} combined with the definition of $P_f(X)$). We then have
\begin{equation} \label{E:Taelman}
L(\phi^{\vee},0) = \prod_f \frac{\bigAord{\FF_f}{A}}{\bigAord{\ophi(\FF_f)}{A}} = \Reg_{\phi} \cdot \rH(\phi),
\end{equation}
where the first equality follows from \eqref{E:Qfvee1} and the second is Taelman's identity.
The formula on the right contains the regulator $\Reg_{\phi} \in K_{\infty}$ and the order of the class module $\rH(\phi)\in A$ (see \cite{Taelman12} for details).
We will use Fang's generalization of Taelman's formula to $t$-modules. See Theorem~\ref{FangClass}.

\subsection{Fang's class module formula} \label{SS:Fangformula}
In \cite{Fang15}, Fang proved an extension of Taelman's class module formula to abelian Anderson $t$-modules defined over the integral closure of $A$ in a finite extension of $K$. Later Angl\`es, Ngo Dac and Tavares Ribeiro \cite{ANT22} extended the class module formula for admissible Anderson modules for more general ring, which comes from global function field over a finite field. For the purpose of the present paper, it suffices to focus on Fang's identity, thus we will mainly provide a summary of it.

Let $E : \sA \to \Mat_{\ell}(A[\tau])$ be an abelian Anderson $t$-module defined over~$A$. 
The exponential series $\Exp_{E} \in \power{K}{\tau}$ of $E$ induces an $\FF_q$-linear function
\begin{equation}
\Exp_{E,K_{\infty}} \colon \Lie(E)(K_{\infty}) \to E(K_{\infty}) \quad \Leftrightarrow \quad
\Exp_{E,K_{\infty}} \colon K_{\infty}^{\ell} \to K_{\infty}^{\ell}.
\end{equation}
Now $\Lie(E)(K_{\infty})$ has a canonical $K_{\infty}$-vector space structure, but Fang~\cite[p.~303]{Fang15} pointed out that it has another structure of a vector space over $\laurent{\FF_q}{t^{-1}}$. Namely we extend $\pd : \sA \to \Mat_{\ell}(K_{\infty})$ to an $\FF_q$-algebra homomorphism,
\begin{align}
\laurent{\FF_q}{t^{-1}} \stackrel{\pd}{\longrightarrow} \Mat_{\ell}(K_\infty) \; :\; \sum_{j \geqslant j_0} c_j t^{-j} \longmapsto \sum_{j \geqslant j_0} c_j \cdot \pd E_t^{-j}.
\end{align}
Notably the series on the right converges by~\cite[Lem.~1.7]{Fang15}. As Fang continued, $\Lie(E)(K_\infty)$ obtains an $\laurent{\FF_q}{t^{-1}}$-vector space structure via~$\pd$. 
For any $g \in \laurent{\FF_q}{t^{-q^{\ell}}}$, we have $\pd g = g\cdot \rI_{\ell}$, and so $\Lie(E)(K_{\infty})$ has dimension $\ell q^{\ell}$ as over $\laurent{\FF_q}{t^{-q^{\ell}}}$, which implies it has dimension~$\ell$ over $\laurent{\FF_q}{t^{-1}}$.
Since $K_{\infty} = \laurent{\FF_q}{\theta^{-1}} \cong \laurent{\FF_q}{t^{-1}}$, we will abuse notation and use the map~$\pd$ to define new $K_\infty$-vector space and $A$-module structures on $\Lie(E)(K_{\infty})$ that are possibly different from scalar multiplication.
With respect to this $K_{\infty}$-structure, Fang showed~\cite[Thm.~1.10]{Fang15} that \[\Lie(E)(A) \subseteq \Lie(E)(K_{\infty})\] and 

\begin{equation*}
\Exp_{E,K_{\infty}}^{-1} \bigl( E(A) \bigr) \subseteq \Lie(E)(K_{\infty})
\end{equation*}
are $A$-lattices in the sense of \cite[Def. 1.9]{Fang15}. In particular they have rank~$\ell$ as an $A$-modules via~$\pd$.

Choose $A$-bases $\{ \bv_1, \dots, \bv_{\ell} \}$ and $\{ \bslambda_1, \dots, \bslambda_{\ell} \}$ of $\Lie(E)(A)$ and $\Exp_{E,K_{\infty}}^{-1}(A^\ell)$ via $\pd$ respectively, and let $V \in \GL_{\ell}(K_{\infty})$ be chosen so that its columns are the coordinates of $\bslambda_1, \dots, \bslambda_{\ell}$ with respect to $\bv_1, \dots, \bv_{\ell}$ (via $\pd$). Following Taelman~\cites{Taelman10}{Taelman12}, Fang defined the regulator of $E$ as
\begin{equation}\label{eq:reg}
\Reg_{E} \assign \gamma \cdot \det(V) \in K_{\infty},\quad \gamma \in \FF_q^{\times},
\end{equation}
where $\gamma$ is chosen so that $\Reg_{E}$ has sign~$1$ (leading coefficient as an element of $\laurent{\FF_q}{\theta^{-1}}$ is~$1$).
This value is independent of the choice of $A$-bases.

Also following Taelman, Fang~\cite[Thm.~1.10]{Fang15} defined the class module of $E$ as
\begin{equation}
\rH(E) \assign \frac{E(K_{\infty})}{\Exp_{E,K_{\infty}}(\Lie(E)(K_{\infty})) + E(A)},
\end{equation}
and he proved that $\rH(E)$ is a finitely generated and torsion $A$-module. Fang's class module formula is the following.

\begin{theorem}[{Fang~{\cite[Thm.~1.10]{Fang15}}}] \label{FangClass}
Let $E : \sA \to \Mat_{\ell}(A[\tau])$ be an abelian Anderson $t$-module. Then
\[
\prod_f \frac{\Aord{\Lie(\soE)(\FF_f)}{A}}{\Aord{\soE(\FF_f)}{A}} = \Reg_{E} \cdot \bigAord{\rH(E)}{A},
\]
where the left-hand side converges in $K_\infty$.
\end{theorem}

\subsection{\texorpdfstring{$L$}{L}-series of Tensor product of Drinfeld modules over \texorpdfstring{$A$}{A}}
Let $\phi$, $\psi : \sA \to A[\tau]$ be Drinfeld modules defined over $A$ of ranks $r$ and $\ell$ respectively as in \eqref{E:phipsidefintro}. Recall the $t$-module $\phi \otimes \psi : \sA \to \Mat_{r+\ell}(A[\tau])$ defined over $A$ as in Definition \ref{ten_def}. 

\subsubsection{Characteristic polynomials of Frobenius and \texorpdfstring{$L$}{L}-series of \texorpdfstring{$\phi\otimes\psi$}{phitenpsi}}\label{ten_charpoly}
Let $f \in A_+$ be irreducible of degree $d$, and let $\lambda \in \sA_+$ be irreducible so that $\lambda(\theta) \neq f$. Let
\[
\rho_{\phi,\lambda} : \Gal(K^{\sep}/K) \to \Aut (T_{\lambda}(\phi)) \cong \GL_r(\sA_{\lambda})
\]
be the Galois representation associated $T_{\lambda}(\phi)$, and similarly define $\rho_{\psi,\lambda} : \Gal(K^{\sep}/K) \to \Aut(T_{\lambda}(\psi))$ for $\psi$. 
The $\lambda$-adic Tate module 
\[
T_{\lambda}(\phi\otimes\psi) = \varprojlim (\phi\otimes\psi)[\lambda^m]\cong \sA_{\lambda}^{r\ell}
\]
 induces another Galois representation
\[
\rho_{\phi\otimes\psi,\lambda} : \Gal(K^{\sep}/K) \to \Aut(T_{\lambda}(\phi\otimes\psi))\cong \GL_{r\ell}(\sA_{\lambda}).
\]
As outlined in \S\ref{SS:munu}, if $\alpha_f \in \Gal(K^{\sep}/K)$ is a Frobenius element for~$f$, then because $\phi$ and $\psi$ have good reduction at~$f$,
\[
\Char(\alpha_f,T_{\lambda}(\phi),X)|_{t=\theta} = P_{\phi,f}(X), \quad
\Char(\alpha_f,T_{\lambda}(\psi),X)|_{t=\theta} = P_{\psi,f}(X),
\]
both of which lie in $A[X]$. 
We define 
\begin{equation*}
\begin{split}
\bP_{\phi \otimes \psi,f}(X)\assign \Char(\alpha_f,T_{\lambda}(\phi\otimes\psi),X)|_{t=\theta}\\
\bP^\vee_{\phi \otimes \psi,f}(X)\assign \Char(\alpha_f,T^\vee_{\lambda}(\phi\otimes\psi),X)|_{t=\theta}
\end{split}.
\end{equation*}
Recall the notation $(P\otimes Q)(X)$ from Definition \ref{D:polytensor}. By \cite[Thm.~3.1]{Hamahata}, we then have 
\begin{equation}\label{P_ten}
\begin{split}
\bP_{\phi \otimes \psi,f}(X) = (P_{\phi,f}\otimes P_{\psi,f})(X)\in A[X]\\
\bP^\vee_{\phi \otimes \psi,f}(X) = (P^\vee_{\phi,f}\otimes P^\vee_{\psi,f})(X)\in K[X]
\end{split}.
\end{equation}
We further set $\bQ_{\phi \otimes \psi,f}(X)$ and $\bQ_{\phi \otimes \psi,f}^{\vee}(X)$ to be their reciprocal polynomials. We now consider the $L$-function
\begin{equation}\label{L_ten}
L((\phi \otimes \psi)^{\vee},s) \assign \prod_f \bQ_{\phi \otimes \psi,f}^{\vee} \bigl( f^{-s} \bigr)^{-1}, \quad s \geqslant 0.
\end{equation}


\subsubsection{Special \texorpdfstring{$L$}{L}-values of tensor product}
Fixing $f\in A_+$ irreducible, we let \[\ophitenpsi: \sA\to\Mat_{r+\ell}(\FF_f[\tau])\] denote the reduction modulo $f$. Recall the completely multiplicative functions $\chi_{\phi}$, $\chi_{\psi} : \sA_+ \to \FF_q^{\times}$ as in~\eqref{E:chidef}. We have the following proposition for determining $\bigAord{\ophitenpsi(\FF_f)}{A}$ in term of the value $\bP_{\phitenpsi, f}(1).$
\begin{proposition}\label{Aord_ten}
Let $f\in A_+$ be irreducible. For $\ophitenpsi: \sA \to \Mat_{r+\ell}(\FF_f[\tau])$ defined above,
\[
\bigAord{\ophitenpsi(\FF_f)}{A} = \chi_{\phi}(f)^{\ell}\, \chi_{\psi}(f)^r\cdot \bP_{\phitenpsi,f}(1) = \frac{\bP_{\phitenpsi,f}(1)}{\bP_{\phitenpsi,f}(0)} \cdot f^{r+\ell}.
\]
\end{proposition}
\begin{proof}
Let $\lambda\in \sA_+$ be irreducible so that $\lambda(\theta)\neq f$. Let $K_f^{\nr}$ be the maximal unramified and separable extension of $K_f$, and let $K^{\nr}_f \supseteq O^{\nr}_f \supseteq M^{\nr}_f$ be its subring of $f$-integral elements and its maximal ideal. Because $\phi$ and $\psi$ have everywhere good reduction, we see from \cite[Thm.~1]{Takahashi82}(see also \cite[Thm.~4.10.5]{Goss}) that
\[
\phi[\lambda^m], \ \psi[\lambda^m] \subseteq O^{\nr}_f, \quad \forall\, m \geqslant 1.
\]
Moreover, the natural reduction maps
\begin{equation*} \label{Tateiso_ten}
\phi(O^{\nr}_f)[\lambda^m] = \phi[\lambda^m] \iso \ophi[\lambda^m], \quad \psi(O^{\nr}_f)[\lambda^m] = \psi[\lambda^m] \iso \opsi[\lambda^m],
\quad \forall\, m \geqslant 1,
\end{equation*}
are $\sA$-module isomorphisms (e.g., see \cite[\S 2]{Takahashi82}), which gives isomorphisms on Tate modules commuting with the Galois action
\begin{equation}\label{diag_ten}
T_\lambda(\phi) \iso T_\lambda(\ophi), \quad T_\lambda(\psi) \iso T_\lambda(\opsi).
\end{equation}
Together with \cite[Thm.~3.1]{Hamahata}, the following diagram of $\sA_\lambda$-modules commutes: 
\begin{equation}
\begin{tikzcd}[column sep=large]
 T_\lambda(\phitenpsi) \arrow{r}{\sim} \isoarrow{dd} &  T_\lambda(\ophitenpsi)\arrow[equal]{d} \\
& T_\lambda(\ophi\otimes\opsi)\isoarrow{d} \\
T_\lambda(\phi)\otimes_{\sA_\lambda}T_\lambda(\psi) \arrow{r}{\sim} & T_\lambda(\ophi)\otimes_{\sA_\lambda}T_\lambda(\opsi)
\end{tikzcd}
\end{equation}
Furthermore, the maps in the diagram above commute with the Galois action, which implies \[\bP_{\phitenpsi,f}(X)=\Char(\tau^d,T_\lambda(\ophitenpsi),X).\]
Now as $\bP_f(X) = (P_{\phi,f} \otimes P_{\psi,f})(X)$ by \eqref{P_ten}, it follows from \eqref{E:PfandPfvee} and Definition~\ref{D:polytensor} that the constant term of $\bP_{f}(X)$ is
\begin{equation} \label{E:bPf0}
\bP_f(0) = \ochi_{\phi}(f)^{\ell} \ochi_{\psi}(f)^r f^{r+\ell}.
\end{equation}
By \cite[Cor.~3.7.8]{HP22}, it remains to show $\chi_{\phi}(f)^{\ell}\, \chi_{\psi}(f)^r\cdot \bP_{\phitenpsi,f}(1)\in A$ is monic. 

Indeed, writing $\bP_f(X) = \sum_{i=0}^{r\ell} b_i X^i$, $b_i \in A[X]$ and
letting $c_0, \dots, c_{r-1} \in A$ be given as in~\eqref{E:Pfdef}, Definition~\ref{D:polytensor} implies that, for $0 \leqslant m \leqslant r\ell -1$, each $b_m$ is a polynomial in $c_0, \dots, c_{r-1}$ with coefficients in $\FF_q$.
Assigning the weight $r-i$ to each $c_i$, then as formal expressions, each monomial in $c_0, \dots, c_{r-1}$ in $b_m$ has the same total weight $r\ell - m$. That is, if $c_0^{n_0} \dots c_{r-1}^{n_{r-1}}$ is a monomial in $b_m$, then
$\sum_{i=0}^{r-1} (r-i)n_i = r\ell - m$, and so by~\S\ref{SSS:Pf},
\[
\deg_{\theta} \bigl(c_0^{n_0} \cdots c_{r-1}^{n_{r-1}} \bigr) \leqslant \sum_{i=0}^{r-1} \frac{d}{r} \cdot n_i(r-i) = \frac{d}{r} (r \ell - m) = d \ell - \frac{dm}{r}.
\]
From~\eqref{E:bPf0}, this is an equality if $m=0$. On the other hand, this inequality implies,
\[
 0 < m \leqslant r\ell -1 \quad \Rightarrow \quad \deg_{\theta} b_m < d \ell.
\]
Therefore from \eqref{E:bPf0}, $\chi_{\phi}(f)^{\ell}\, \chi_{\psi}(f)^r\cdot \bP_{\phitenpsi,f}(1)\in A$ is monic.
\end{proof}

For each irreducible $f \in \sA_+$, \eqref{P_ten} implies that $\bQ_{\phitenpsi,f}^{\vee}(1) = \bP_{\phitenpsi,f}(1)/\bP_{\phitenpsi,f}(0)$. By combining Proposition~\ref{Aord_ten} and~\eqref{L_ten}, we obtain the following identity for $L(\phitenpsi^{\vee},0)$, which shows that Fang's class module formula (Theorem~\ref{FangClass}) applies to the special values we are considering.

\begin{proposition} \label{P:Lten0}
\[
L((\phitenpsi)^{\vee},0) = \prod_f \frac{\bigAord{\FF_f^{r+\ell}}{A}}{\bigAord{\ophitenpsi(\FF_f)}{A}}.
\]
\end{proposition}

\subsubsection{Convolution \texorpdfstring{$L$}{L}-series of \texorpdfstring{$\phitenpsi$}{phitenpsi}}\label{SS:tenconv}
We investigate
\[
L((\phitenpsi)^{\vee},s) = \prod_f \bQ_{\phitenpsi,f}^{\vee}\bigl( f^{-s} \bigr)^{-1}
\]
from~\eqref{L_ten}. We at first fix $f \in A_+$ irreducible, and we let $\alpha_1, \dots, \alpha_r \in \oK$ be the roots of $P_{\phi,f}^{\vee}(X)$, and we let $\beta_1, \dots, \beta_{\ell} \in \oK$ be the roots of $P^{\vee}_{\psi,f}(X)$. We split it into two cases.

\paragraph{For \texorpdfstring{$r=\ell$}{r=l}}\label{L_ten_eq}
As $\bQ_{\phitenpsi,f}^{\vee}(X)$ is the reciprocal polynomial of $\bP_{\phitenpsi,f}^{\vee}(X) = P_{\phi,f}^{\vee}(X) \otimes P^\vee_{\psi,f}(X)$, we can expand $\bQ_{\phitenpsi,f}^{\vee}(f^{-s})^{-1}$ using Cauchy's identity (Theorem~\ref{T:Cauchy}). We note from~\eqref{E:PfandPfvee} that $\alpha_1 \cdots \alpha_r = \chi_{\phi}(f) f^{-1}$ and $\beta_1 \cdots \beta_r = \chi_{\psi}(f) f^{-1}$. By the definitions of $\bsmu_{\phi}$ and $\bsmu_{\psi}$ from \eqref{E:bsmudef}, Theorem~\ref{T:Cauchy} implies
\begin{align*}
\bQ_{\phitenpsi,f}^{\vee}\bigl( &f^{-s} \bigr)^{-1} \\
&= \biggl( 1 - \frac{\chi_{\phi}(f) \chi_{\psi}(f) }{f^{rs+2}} \biggr)^{-1} 
\underset{k=(k_1, \dots, k_{r-1})}{\sum_{k_1=0}^{\infty} \cdots \sum_{k_{r-1}=0}^{\infty}} S_{k}(\bsalpha) S_{k}(\bsbeta) f^{-s(k_1 + 2k_2 + \cdots + (r-1)k_{r-1})} \\
&= 
\biggl( 1 - \frac{\chi_{\phi}(f) \chi_{\psi}(f) }{f^{rs+2}} \biggr)^{-1} \smash{\sum_{k_1, \ldots, k_{r-1} \geqslant 0}} \frac{\bsmu_{\phi} \bigl(f^{k_1}, \dots, f^{k_{r-1}} \bigr) \bsmu_{\psi} \bigl( f^{k_1}, \dots, f^{k_{r-1}} \bigr)}{ 
 f^{2(k_1+k_2 + \cdots + k_{r-1})+s(k_1 + 2k_2 + \cdots +(r-1)k_{r-1})}},
\end{align*}
where $\bsalpha = (\alpha_1, \dots, \alpha_r)$ and $\bsbeta = (\beta_1, \dots, \beta_r)$.
We define the twisted Carlitz zeta function
\begin{equation}
L(A,\chi_{\phi}\chi_{\psi},s) \assign \sum_{a \in A_+} \frac{\chi_{\phi}(a)\chi_{\psi}(a)}{a^s},
\end{equation}
and finally we define the $L$-series,
\begin{equation} \label{E:Lmunurxr}
L(\bsmu_{\phi} \times \bsmu_{\psi},s) \assign \sum_{a_1 \in A_+} \cdots \sum_{a_{r-1} \in A_+} \frac{\bsmu_{\phi}(a_1, \dots, a_{r-1}) \bsmu_{\psi}(a_1, \dots, a_{r-1})}{(a_1 \cdots a_{r-1})^2 (a_1 a_2^2 \cdots a_{r-1}^{r-1})^s}.
\end{equation}
The convergence of this series in $K_{\infty}$ can be deduced from Proposition~\ref{P:bsmunuprops}(d) for $s \geqslant 0$.
After some straightforward simplification we arrive at the following result.

\begin{theorem} \label{T:Ltenrxr}
Let $\phi$, $\psi : \sA \to A[\tau]$ be Drinfeld modules both of rank $r \geqslant 2$ with everywhere good reduction, as defined in~\eqref{E:phipsidefintro}. Then
\[
L((\phi \otimes \psi)^{\vee},s) = L(A,\chi_{\phi}\chi_{\psi},rs+2) \cdot
L(\bsmu_{\phi} \times \bsmu_{\psi},s).
\]
\end{theorem}

We can substitute $s=0$ into Theorem~\ref{T:Ltenrxr} and obtain the following special value identity.

\begin{corollary} \label{C:Lmumurxr}
Let $\phi$, $\psi : \sA \to A[\tau]$ be Drinfeld modules both of rank $r \geqslant 2$ with everywhere good reduction, as defined in~\eqref{E:phipsidefintro}.
\begin{align*}
L(\bsmu_{\phi} \times \bsmu_{\psi},0)
&= \sum_{a_1 \in A_+} \cdots \sum_{a_{r-1} \in A_+} \frac{\bsmu_{\phi}(a_1, \dots, a_{r-1}) \bsmu_{\psi}(a_1, \dots, a_{r-1})}{(a_1 \cdots a_{r-1})^2} \\[10pt]
&= \frac{ \Reg_{\phitenpsi} \cdot \Aord{\rH(\phitenpsi)}{A}}{L(A,\chi_\phi\chi_\psi,2)}.
\end{align*}

\end{corollary}

\paragraph{For \texorpdfstring{$r < \ell$}{r < l}}
Using that $\alpha_1\cdots\alpha_r=\chi_\phi(f)f^{-1}$, we apply Bump's specialization of Cauchy's identity (Corollary~\ref{C:Cauchynl}), and similar to calculations in \S\ref{L_ten_eq}, we find

\begin{align*}
\bQ_{\phitenpsi,f}^{\vee}\bigl( f^{-s} \bigr)^{-1} &= \underset{\substack{k=(k_1, \dots, k_{r-1}) \\ k'=(k_1, \dots, k_{r},0 \dots, 0)}}{\sum_{k_1=0}^{\infty} \cdots \sum_{k_{r}=0}^{\infty}} S_{k}(\bsalpha) S_{k'}(\bsbeta) \bigl( \chi_{\phi}(f) f^{-1} \bigr)^{k_r} f^{-s(k_1 + 2k_2 + \cdots + rk_r)} \\
&= 
 \smash{\sum_{k_1, \ldots, k_{r} \geqslant 0}}\frac{ \bsmu_{\phi} \bigl(f^{k_1}, \dots, f^{k_{r-1}} \bigr) \bsmu_{\psi} \bigl( f^{k_1}, \dots, f^{k_{r}},1, \ldots, 1 \bigr) 
\cdot \chi_{\phi}\bigl(f^{k_r} \bigr) }{ f^{2(k_1+k_2 + \cdots + k_{r})+s(k_1 + 2k_2 + \cdots + rk_{r})}}.
\end{align*}
The expression $\bsmu_{\psi}( f^{k_1}, \dots, f^{k_{r}},1, \ldots, 1)$ generically has $1$'s in exactly the last $\ell-1-r$ places. We thus define the $L$-series when $r < \ell$,
\begin{equation} \label{E:Lmumurxl1}
L(\bsmu_{\phi} \times \bsmu_{\psi},s) \assign
\sum_{a_1, \ldots, a_r \in A_+}
\frac{\chi_{\phi}(a_r)\bsmu_{\phi}(a_1, \dots, a_{r-1}) \bsmu_{\psi}(a_1, \dots, a_r,1, \ldots, 1)}{(a_1 \cdots a_r)^2 (a_1 a_2^2 \cdots a_r^r)^s}.
\end{equation}

By some straightforward calculation, we obtain the following theorem.

\begin{theorem} \label{T:Ltenrxl}
Let $\phi$, $\psi : \sA \to A[\tau]$ be Drinfeld modules of ranks $r$ and $\ell$ respectively with everywhere good reduction, as defined in~\eqref{E:phipsidefintro}. Assume that $r$, $\ell \geqslant 2$ and that $r < \ell$. Then
\[
L((\phi \otimes \psi)^{\vee},s) = L(\bsmu_{\phi} \times \bsmu_{\psi},s).
\]
\end{theorem}

We can substitute $s=0$ into Theorem~\ref{T:Ltenrxl} and obtain the following special value identity.

\begin{corollary} \label{C:Lmumurxl}
Let $\phi$, $\psi : \sA \to A[\tau]$ be Drinfeld modules both of rank $r \geqslant 2$ with everywhere good reduction, as defined in~\eqref{E:phipsidefintro}. Assume that $r$, $\ell \geqslant 2$ and that $r < \ell$. Then
\begin{align*}
L(\bsmu_{\phi} \times \bsmu_{\psi},0)
&= \sum_{a_1, \ldots, a_r \in A_+}
\frac{\chi_{\phi}(a_r)\bsmu_{\phi}(a_1, \dots, a_{r-1}) \bsmu_{\psi}(a_1, \dots, a_r,1, \ldots, 1)}{(a_1 \cdots a_r)^2} \\[10pt]
&= \Reg_{\phitenpsi} \cdot \Aord{\rH(\phitenpsi)}{A}.
\end{align*}

\end{corollary}

\begin{remark}
Note that the case $r>\ell$ is included in the case $r<\ell$ since $\phitenpsi$ is isomorphic to $\phi\otimes\phi$ as $t$-modules by \cite[Prop.~2.5]{Hamahata}.
\end{remark}

\subsection{\texorpdfstring{$L$}{L}-series of symmetric and alternating squares of Drinfeld module over \texorpdfstring{$A$}{A}}

We assume $p\neq 2$. Let $\phi : \sA \to A[\tau]$ be Drinfeld modules defined over $A$ of rank $r$ as in \eqref{E:phipsidefintro}. Recall the $t$-module $\Sym^2\phi : \sA \to \Mat_{r+1}(A[\tau])$ and $\Alt^2\phi : \sA \to \Mat_{r-1}(A[\tau])$ defined over $A$ as in Definition \ref{def_symalt}. Following the same process as in \S\ref{ten_charpoly}, We define 
\begin{equation*}
\begin{split}
\bP_{\Sym^2\phi,f}(X)\assign \Char(\alpha_f,T_{\lambda}(\Sym^2\phi),X)|_{t=\theta}\\
\bP^\vee_{\Sym^2\phi,f}(X)\assign \Char(\alpha_f,T^\vee_{\lambda}(\Sym^2\phi),X)|_{t=\theta}
\end{split},
\end{equation*}
and 
\begin{equation*}
\begin{split}
\bP_{\Alt^2\phi,f}(X)\assign \Char(\alpha_f,T_{\lambda}(\Alt^2\phi),X)|_{t=\theta}\\
\bP^\vee_{\Alt^2\phi,f}(X)\assign \Char(\alpha_f,T^\vee_{\lambda}(\Alt^2\phi),X)|_{t=\theta}
\end{split}.
\end{equation*}
Recall the notations $(\Sym^2P)(X)$ and $(\Alt^2(P)(X))$ from Definition \ref{D:polytensor}. By \cite[Thm.~5.9]{Hamahata}, we then have 
\begin{equation}\label{P_sym}
\begin{split}
\bP_{\Sym^2\phi,f}(X) = (\Sym^2P_{\phi,f})(X)\in A[X]\\
\bP^\vee_{\Sym^2\phi,f}(X) = (\Sym^2P^\vee_{\phi,f})(X)\in K[X]
\end{split},
\end{equation}
and 
\begin{equation}\label{P_alt}
\begin{split}
\bP_{\Alt^2\phi,f}(X) = (\Alt^2P_{\phi,f})(X)\in A[X]\\
\bP^\vee_{\Alt^2\phi,f}(X) = (\Alt^2P^\vee_{\phi,f})(X)\in K[X]
\end{split}.
\end{equation}
We further set $\bQ_{E,f}(X)$ and $\bQ_{E,f}^{\vee}(X)$ to be their reciprocal polynomials for $E=\Sym^2\phi$ or $\Alt^2\phi$. We now consider the $L$-functions
\begin{align}
L((\Sym^2\phi)^{\vee},s) \assign \prod_f \bQ_{\Sym^2\phi,f}^{\vee} \bigl( f^{-s} \bigr)^{-1}, \quad s \geqslant 0 \label{L_sym}\\
L((\Alt^2\phi)^{\vee},s) \assign \prod_f \bQ_{\Alt^2\phi,f}^{\vee} \bigl( f^{-s} \bigr)^{-1}, \quad s \geqslant 0.\label{L_alt}
\end{align}


\subsubsection{Special \texorpdfstring{$L$}{L}-values of symmetric and alternating squares}
Fixing $f\in A_+$ irreducible, similar to the tensor products, we 
denote \[\osymphi: \sA \to \Mat_{r+1}(\FF_f[\tau])\] and \[\oaltphi: \sA \to \Mat_{r-1}(\FF_f[\tau])\] the reductions modulo $f$. We also have the following proposition for determining $\bigAord{\osymphi(\FF_f)}{A}$ and $\bigAord{\oaltphi(\FF_f)}{A}$ in term of the values $\bP_{\symphi, f}(1)$ and $\bP_{\altphi, f}(1)$, respectively.
\begin{proposition}\label{Aord_symalt}
Let $f\in A_+$ be irreducible. For $\osymphi: \sA \to \Mat_{r+1}(\FF_f[\tau])$ and $\oaltphi: \sA \to \Mat_{r-1}(\FF_f[\tau])$ defined above,
\[
\bigAord{\osymphi(\FF_f)}{A} = (-1)^{\frac{r(r+1)}{2}}\chi_{\phi}(f)^{r+1}\cdot \bP_{\symphi,f}(1) = \frac{\bP_{\symphi,f}(1)}{\bP_{\symphi,f}(0)} \cdot f^{r+1},
\] and 
\[
\bigAord{\oaltphi(\FF_f)}{A} = (-1)^{\frac{r(r-1)}{2}}\chi_{\phi}(f)^{r-1}\cdot \bP_{\altphi,f}(1) = \frac{\bP_{\altphi,f}(1)}{\bP_{\altphi,f}(0)} \cdot f^{r-1}.
\]
\end{proposition}
\begin{proof}
In a similar fashion to the proof of Proposition~\ref{Aord_ten}, following from \cite[Thm.~5.9]{Hamahata} the diagrams of $\sA_\lambda$-modules below commutes:  
\[
\begin{tikzcd}[column sep=large]
 T_\lambda(\symphi) \arrow{r}{\sim} \isoarrow{dd} &  T_\lambda(\osymphi)\arrow[equal]{d} \\
& T_\lambda(\Sym^2(\ophi))\isoarrow{d} \\
\Sym^2(T_\lambda(\phi)) \arrow{r}{\sim} & \Sym^2(T_\lambda(\ophi))
\end{tikzcd}
\begin{tikzcd}[column sep=large]
 T_\lambda(\altphi) \arrow{r}{\sim} \isoarrow{dd} &  T_\lambda(\oaltphi)\arrow[equal]{d} \\
& T_\lambda(\Alt^2(\ophi))\isoarrow{d} \\
\Alt^2(T_\lambda(\phi)) \arrow{r}{\sim} & \Alt^2(T_\lambda(\ophi))
\end{tikzcd}\]
Furthermore, the maps in the diagram above commute with the Galois action, which implies \[\bP_{\symphi,f}(X)=\Char(\tau^d,T_\lambda(\osymphi),X),\]
and
\[\bP_{\altphi,f}(X)=\Char(\tau^d,T_\lambda(\oaltphi),X).\]
The remaining part follows by the same process of the proof of Proposition~\ref{Aord_ten}.
\end{proof}

For each irreducible $f \in \sA_+$, \eqref{P_sym} and \eqref{P_alt} imply that $\bQ_{\symphi,f}^{\vee}(1) = \bP_{\symphi,f}(1)/\bP_{\symphi,f}(0)$ and $\bQ_{\altphi,f}^{\vee}(1) = \bP_{\altphi,f}(1)/\bP_{\altphi,f}(0)$. By combining Proposition~\ref{Aord_symalt}, \eqref{L_sym} and \eqref{L_alt}, we obtain the following identities for $L(\symphi^{\vee},0)$ and $L(\altphi^{\vee},0)$, which show that Fang's class module formula (Theorem~\ref{FangClass}) applies to the special values we are considering.

\begin{proposition} \label{P:Lsymalt0}
\[
L((\symphi)^{\vee},0) = \prod_f \frac{\bigAord{\FF_f^{r+1}}{A}}{\bigAord{\osymphi(\FF_f)}{A}},
\]
\[
L((\altphi)^{\vee},0) = \prod_f \frac{\bigAord{\FF_f^{r-1}}{A}}{\bigAord{\oaltphi(\FF_f)}{A}}.
\]
\end{proposition}


\subsubsection{Convolution \texorpdfstring{$L$}{L}-series of symmetric and alternating squares}
We investigate
\[
L((\symphi)^{\vee},s) = \prod_f \bQ_{\symphi,f}^{\vee}\bigl( f^{-s} \bigr)^{-1}
\]
and 
\[
L((\altphi)^{\vee},s) = \prod_f \bQ_{\altphi,f}^{\vee}\bigl( f^{-s} \bigr)^{-1}
\]
from~\eqref{L_sym} and \eqref{L_alt}. Same as we did in tensor product cases, we at first fix $f \in A_+$ irreducible, and we let $\alpha_1, \dots, \alpha_r \in \oK$ be the roots of $P_{\phi,f}^{\vee}(X)$, and we let $\beta_1, \dots, \beta_{\ell} \in \oK$ be the roots of $P^{\vee}_{\psi,f}(X)$. Instead of applying Cauchy's identity, we should apply Littlewood's identities to analyze $\bQ_{\symphi,f}^{\vee}(X)^{-1}$ and $\bQ_{\altphi,f}^{\vee}(X)^{-1}.$

\paragraph{Symmetric square}\label{L_sym_cov}
As $\bQ_{\symphi,f}^{\vee}(X)$ is the reciprocal polynomial of $\bP_{\symphi,f}^{\vee}(X) = (\Sym^2 P_{\phi,f}^{\vee})(X)$, we can expand $\bQ_{\symphi,f}^{\vee}(f^{-s})^{-1}$ using Littlewood's identity (Theorem~\ref{T:Littlewood}(a)). We note from~\eqref{E:PfandPfvee} that $\alpha_1 \cdots \alpha_r = \chi_{\phi}(f) f^{-1}$. By the definitions of $\bsmu_{\phi}$ from \eqref{E:bsmudef}, Theorem~\ref{T:Littlewood}(a) implies
\begin{align*}
\bQ_{\symphi,f}^{\vee}\bigl( &f^{-s} \bigr)^{-1} \\
&= \biggl( 1 - \frac{\chi_{\phi}(f)^2}{f^{rs+2}} \biggr)^{-1} 
\underset{2\mid k_i \text{ for all $i$}}{\underset{k=(k_1, \dots, k_{r-1})}{\sum_{k_1=0}^{\infty} \cdots \sum_{k_{r-1}=0}^{\infty}}} S_{k}(\bsalpha)f^{-s(\frac{k_1}{2} + \frac{2k_2}{2} + \cdots + \frac{(r-1)k_{r-1}}{2})} \\
&= 
\biggl( 1 - \frac{\chi_{\phi}(f)^2}{f^{rs+2}} \biggr)^{-1} \smash{\sum_{k_1, \ldots, k_{r-1} \geqslant 0}} \frac{\bsmu_{\phi} \bigl(f^{2k_1}, \dots, f^{2k_{r-1}} \bigr)}{ 
 f^{2(k_1+k_2 + \cdots + k_{r-1})+s(k_1 + 2k_2 + \cdots +(r-1)k_{r-1})}},
\end{align*}
where $\bsalpha = (\alpha_1, \dots, \alpha_r)$.
Recalling the twisted Carlitz zeta function
\begin{equation}
L(A,\chi_{\phi}^2,s) \assign \sum_{a \in A_+} \frac{\chi_{\phi}(a)^2}{a^s},
\end{equation}
and finally we define the $L$-series,
\begin{equation} \label{E:Lsymmu}
L(\tbsmu_{\phi},s) \assign \sum_{a_1 \in A_+} \cdots \sum_{a_{r-1} \in A_+} \frac{\bsmu_{\phi}(a_1^2, \dots, a_{r-1}^2)}{(a_1 \cdots a_{r-1})^2 (a_1 a_2^2 \cdots a_{r-1}^{r-1})^s}.
\end{equation}
The convergence of this series in $K_{\infty}$ can be deduced from Proposition~\ref{P:bsmunuprops}(d) for $s \geqslant 0$.
After some straightforward simplification we arrive at the following result.

\begin{theorem} \label{T:Lsymrxr}
Let $\phi : \sA \to A[\tau]$ be Drinfeld module of rank $r \geqslant 2$ with everywhere good reduction, as defined in~\eqref{E:phipsidefintro}. Then
\[
L((\symphi)^{\vee},s) = L(A,\chi_{\phi}^2,rs+2) \cdot
L(\tbsmu_{\phi},s).
\]
\end{theorem}

We can substitute $s=0$ into Theorem~\ref{T:Lsymrxr} and obtain the following special value identities.

\begin{corollary} \label{C:Lsymrxr}
Let $\phi : \sA \to A[\tau]$ be Drinfeld module of rank $r \geqslant 2$ with everywhere good reduction, as defined in~\eqref{E:phipsidefintro}.
\begin{align*}
L(\tbsmu_{\phi},0)
= \sum_{a_1 \in A_+} \cdots \sum_{a_{r-1} \in A_+} \frac{\bsmu_{\phi}(a_1^2, \dots, a_{r-1}^2)}{(a_1 \cdots a_{r-1})^2} 
= \frac{ \Reg_{\symphi} \cdot \Aord{\rH(\symphi)}{A}}{L(A,\chi_\phi^2,2)}.
\end{align*}
\end{corollary}

\paragraph{Alternating square}\label{L_alt_cov}
We splits it into three cases. As $\bQ_{\altphi,f}^{\vee}(X)$ is the reciprocal polynomial of $\bP_{\altphi,f}^{\vee}(X) = (\Alt^2 P_{\phi,f}^{\vee})(X)$, we can expand $\bQ_{\altphi,f}^{\vee}(f^{-s})^{-1}$ using Littlewood's identity (Theorem~\ref{T:Littlewood}(b)) for the first two cases. We recall from~\eqref{E:PfandPfvee} that $\alpha_1 \cdots \alpha_r = \chi_{\phi}(f) f^{-1}$, and $\bsmu_{\phi}$ from \eqref{E:bsmudef}. Letting $\bsalpha = (\alpha_1, \dots, \alpha_r)$ and $w = \lfloor \frac{n}{2}\rfloor$, we have the following identities.
\begin{enumerate}
    \item[Case 1:] For $r\geqslant 3$ and $2\nmid r-1$,
    \begin{align*}
    \bQ_{\altphi,f}^{\vee}\bigl( &f^{-s} \bigr)^{-1} \\
    &= \biggl( 1 - \frac{\chi_{\phi}(f)}{f^{\frac{rs}{2}+1}} \biggr)^{-1} 
    \underset{k=(0,k_1,0,k_2, \dots, k_w,0)}{\sum_{k_1=0}^{\infty} \cdots \sum_{k_w=0}^{\infty}} S_{k}(\bsalpha)f^{-s(k_1+ 2k_2 + \cdots + wk_w)} \\
    &= \biggl( 1 - \frac{\chi_{\phi}(f)}{f^{\frac{rs}{2}+1}} \biggr)^{-1} \smash{\sum_{k_1, \ldots, k_w \geqslant 0}} \frac{\bsmu_{\phi} \bigl(1,f^{k_1},1,f^{k_2}, \dots, f^{k_w},1 \bigr)}{ 
     f^{k_1+k_2 + \cdots + k_w+s(k_1 + 2k_2 + \cdots +wk_w)}},
    \end{align*}
    \item[Case 2:] For $r\geqslant 3$ and $2\mid r-1$,
    \begin{align*}
    \bQ_{\altphi,f}^{\vee}\bigl( &f^{-s} \bigr)^{-1} \\
    &= \underset{k=(0,k_1,0,k_2, \dots, k_w)}{\sum_{k_1=0}^{\infty} \cdots \sum_{k_w=0}^{\infty}} S_{k}(\bsalpha)f^{-s(k_1+ 2k_2 + \cdots + wk_w)} \\
    &= \smash{\sum_{k_1, \ldots, k_w \geqslant 0}} \frac{\bsmu_{\phi} \bigl(1,f^{k_1},1,f^{k_2}, \dots, f^{k_w} \bigr)}{ 
     f^{k_1+k_2 + \cdots + k_w+s(k_1 + 2k_2 + \cdots +wk_w)}},
    \end{align*}
    \item[Case 3:] If $r = 2$, then \[\bQ_{\altphi,f}^{\vee}\bigl( f^{-s} \bigr)^{-1} = \biggl( 1 - \frac{\chi_{\phi}(f)}{f^{s+1}} \biggr)^{-1}.\]
\end{enumerate}


Considering the twisted Carlitz zeta function
\begin{equation*}
L(A,\chi_{\phi},s) \assign \sum_{a \in A_+} \frac{\chi_{\phi}(a)}{a^s},
\end{equation*}
and finally we define the $L$-series,
\begin{equation} \label{E:Laltmu}
L(\hat{\bsmu}_{\phi},s) \assign \underset{a_i = 1 \text{ if $2\nmid i$}}{\sum_{a_1,\dots a_{r-1} \in A_+}} \frac{\bsmu_{\phi}(a_1, \dots, a_{r-1})}{a_1 \cdots a_{r-1} (a_1 a_2^2 \cdots a_{r-1}^{r-1})^s}.
\end{equation}
The convergence of this series in $K_{\infty}$ can be deduced from Proposition~\ref{P:bsmunuprops}(d) for $s \geqslant 0$.
After some straightforward simplification we arrive at the following result.

\begin{theorem} \label{T:Laltrxr}
Let $\phi : \sA \to A[\tau]$ be Drinfeld module of rank $r \geqslant 2$ with everywhere good reduction, as defined in~\eqref{E:phipsidefintro}. Then
\[
 L((\altphi)^{\vee},s) = L(A,\chi_{\phi},\frac{rs}{2}+1)^{\frac{(-1)^r+1}{2}} \cdot
L(\hat{\bsmu}_{\phi},s).
\]
\end{theorem}

We can substitute $s=0$ into Theorem~\ref{T:Laltrxr} and obtain the following special value identities.

\begin{corollary} \label{C:Laltrxr}
Let $\phi : \sA \to A[\tau]$ be Drinfeld module of rank $r \geqslant 2$ with everywhere good reduction, as defined in~\eqref{E:phipsidefintro}.
\begin{align*}
L(\hat{\bsmu}_{\phi},0)
= \underset{a_i = 1 \text{ if $2\nmid i$}}{\sum_{a_1,\dots a_{r-1} \in A_+}} \frac{\bsmu_{\phi}(a_1, \dots, a_{r-1})}{a_1 \cdots a_{r-1}}
= \frac{ \Reg_{\altphi} \cdot \Aord{\rH(\altphi)}{A}}{L(A,\chi_\phi,1)^{\frac{(-1)^r+1}{2}}}.
\end{align*}

\end{corollary}

\begin{remark}
In the case $r=2$, from \eqref{E:Laltmu}, the $L$-series $L(\hat{\bsmu}_{\phi},s)=1$. Theorem~\ref{T:Laltrxr} and Corollary~\ref{C:Laltrxr} imply \[L((\altphi)^{\vee},s) = L(A,\chi_{\phi},s+1),\] 
and 
\[L(A,\chi_\phi,1) = \Reg_{\altphi} \cdot \Aord{\rH(\altphi)}{A}.\]
In fact, the alternating square $\altphi$ is a Drinfeld module of rank $1$ defined by $(\altphi)_t = \theta-\kappa_2\tau$. In this case, the class module $\rH(\altphi)$ is trivial, and $\Reg_{\altphi}=\Log_{\altphi}(1)$. We refer readers to the next chapter for more details on regulators of tensor products, symmetric and alternating squares. In conclusion, \[L(A,\chi_\phi,1)=\Log_{\altphi}(1).\]
\end{remark}

\section{Regulators of tensor products, symmetric and alternating squares}\label{S:regulator}

In this chapter, we provide explicit expressions of regulators of tensor, symmetric and alternating squares of Drinfeld modules of rank $2$. 

\subsection{Anderson's exponentiation theorem}

We let $E:\sA\to \Mat_\ell(\C[\tau])$ be a uniformizable almost strictly pure $t$-module of rank $r$ in the sense of \cite{NamoijamP22}, and let $\bm = (m_1,\dots,m_r)^\tr\in\Mat_{r\times 1}(\cM_E)$ be a basis of its $t$-motive $\cM_E\assign\Mat_{1\times\ell}(\C[\tau])$ with $\tPhi_E$ denoting multiplication by $\tau$ on $\cM_E$. Picking $\bn = (n_1,\dots,n_r)^\tr\in\Mat_{r\times 1}(\cN_E)$ to be a basis of its dual $t$-motive $\cN_E\assign\Mat_{1\times\ell}(\C[\sigma])$, a \emph{$t$-frame} for $E$ is a pair $(\iota_E,\Phi_E)$, where $\Phi_E$ represents multiplication by $\sigma$ with respect to $\bn$, and $\iota_E:\Mat_{1\times r}(\C[t])\to \cN_E$ is a map given by for $\bsalpha = (\alpha_1,\dots,\alpha_r)\in\Mat_{1\times r}(\C[t])$, \[\iota_E(\bsalpha) = \bsalpha\cdot\bn = \alpha_1n_1+\cdots+\alpha_rn_r.\]

For $n=\sum_{i=0}^la_i\sigma^i\in\cN_E$ with $a_i\in\Mat_{1\times\ell(\C)}$, we then define two maps $\varepsilon_0$, $\varepsilon_1:\cN_E\to \C^\ell$ by setting
\[\varepsilon_0(n)\assign a_0^\tr,\quad \varepsilon_1(n)\assign\left(\sum_{i=0}^la_i^{(i)}\right)^\tr. \]

We have the following two results by Anderson.

\begin{lemma}[Anderson {\cites[Prop. 2.5.8]{HartlJuschka20}[Lem.~3.4.1]{NamoijamP22}}]
There exists a unique bounded $\C$-linear map 
\[\cE_{0}=\cE_{0,E}:\left(\Mat_{1\times\ell}(\TT_\theta), \dnorm{\,\cdot\,}_\theta\right)\to\left(\C^\ell,\norm{\,\cdot\,}_\infty\right)\]
of normed vector spaces such that $\cE_{0}\vert_{\Mat_{1\times r}(\C[t])}=\varepsilon_0\circ\iota_E.$
\end{lemma}

We further let $\cE_1=\cE_{1,E}\assign\varepsilon_1\circ\iota_E:\Mat_{1\times r}(\C[t])\to\C^\ell.$ Then we state the Anderson's exponentiation theorem below.

\begin{theorem}[Anderson {\cites[Thm.~2.5.21]{HartlJuschka20}[Thm.~3.4.2]{NamoijamP22}}]\label{T:AET}
Let $E:\sA\to \Mat_\ell(\C[\tau])$ be an $\sA$-finite $t$-module of rank $r$ with $t$-frame $(\iota_E,\Phi_E).$ Fix $\bh\in\Mat_{1\times r}(\C[t])$, and suppose there exists $\bg\in \Mat_{1\times r}(\TT_\theta)$ such that 
\[\bg^{(-1)}\Phi_E-\bg=\bh.\]
Then 
\[\Exp_E\left(\cE_0(\bg+\bh)\right)=\cE_1(\bh).\]
\end{theorem}

Let $\xi\in\C^\ell$, and define $\bh_\xi\in\Mat_{1\times r}(\C[t])$ as in \cite[(4.4.21)]{NamoijamP22}. Considering 
\[\bg\assign\sum_{m\geqslant 0}\bh_\xi^{(m+1)}\left(\Phi_E^{-1}\right)^{(m+1)}\cdots \left(\Phi_E^{-1}\right)^{(1)}\in \Mat_{1\times r}(\TT_\theta),\]
by verifying directly, we have
\[\bg^{(-1)}\Phi_E-\bg=\bh_\xi.\] Theorem~\ref{T:AET} and \cite[Prop.~4.5.22]{NamoijamP22} imply
\begin{equation}\label{eq:exp_E}
    \Exp_E\left(\cE_{0,E}\rbracket{\bh_\xi+\sum_{m\geqslant 1}\bh_\xi^{(m)}\left(\Phi_E^{-1}\right)^{(m)}\cdots \left(\Phi_E^{-1}\right)^{(1)}}\right)=\xi.
\end{equation}

\begin{remark}
The construction above generalizes Chen's construction for Drinfeld modules \cite[Rmk. 3.1.8]{Chen22} to uniformizable almost strictly pure $t$-modules.
\end{remark}

\subsection{Logarithms of tensor structures}
We fix a Drinfeld module $\phi:\sA\to A[\tau]$ of rank $r$. In this subsection, we follow the processes in \cite{NamoijamP22} to compute $\cE_0$ and $\bh_\xi$ in \eqref{eq:exp_E}, which allow us to provide expressions of logarithms of tensor structures. It requires that the $t$-modules $E\assign \phi^{\otimes 2}$, $\Sym^2(\phi)$, $\Alt^2(\phi)$ are uniformizable and almost strictly pure, which are followed by Proposition~\ref{P:uniformizable} and 
Remark \ref{R:almoststrictlypure}.

\subsubsection{Calculation of \texorpdfstring{$\cE_0$}{E0}}
From now on, we suppose $\phi_t=\theta+\kappa_1\tau+\kappa_2\tau^2$ with $\kappa_2\in\FF_q^{\times}$. 
\paragraph{The case tensor square}\label{S:E0ten}
We let $E = \phi^{\otimes 2}$. First of all, we define a $t$-module $\tE:\sA\to\Mat_{4}(\C[\tau])$ from \[\cN_{\tE}\assign\cN_\phi^{\otimes 2} \assign \C[\sigma]\otimes_{\C[t]}\C[\sigma]\] in a similar way as in \S\ref{S:ten}. To be precise, one can verify that
\begin{equation}\label{eq:basisdualten}
    \{s_i\}_{i=1}^4=\{1\otimes 1,1\otimes \sigma,\sigma\otimes 1,\sigma^2\otimes 1\}
\end{equation} 
is a $\C[\sigma]$-basis of $\cN_{\tE}$ (cf. Lemma~\ref{ten_basis}). Then we obtain $\tE$ by solving the following equation for $\tE_t:$
\begin{equation*}
t\cdot \bu(s_1,\dots,s_{4})^\tr=\bu \tE_t^*(s_1,\dots,s_{4})^\tr
\end{equation*} for all $\bu\in \Mat_{1\times d}(\C[\sigma]).$ Explicitly, we have 
\begin{equation}
\tE_t = 
\begin{pmatrix}
    \theta& \kappa_1\tau& \kappa_1\tau& \kappa_2\tau^2\\
    & \theta& \kappa_2\tau& \\
    \kappa_1^{(-1)}& \kappa_2\tau& \theta& \kappa_1\tau\\
    \kappa_2& & & \theta
\end{pmatrix}.
\end{equation}
By calculating directly, we have an isomorphism $U_E^*=U_E^\tr$ from $\tE$ to $E$ given by
\[U_E = 
\begin{pmatrix}
    & & & 1\\
    & & 1& -\frac{\kappa_1}{\kappa_2}\\
    & 1 & &\\
    1& -\frac{\kappa_1^{(-1)}}{\kappa_2}& & &
\end{pmatrix}\in \GL_4(\C)\subseteq\GL_4(\C[\tau]).\]
By the $\C[\sigma]$-basis in \eqref{eq:basisdualten}, we identify $\cN_{\tE}$ with $\Mat_{1\times 4}(\C[\sigma])$ by setting $s_i\mapsto\bs_i$, where $\bs_1,\dots,\bs_4$ is the standard basis vectors in $\Mat_{1\times 4}(\C[\sigma])$. In this way, the $\C[t]$-basis 
\[\{1\otimes 1,\,1\otimes \sigma,\,\sigma\otimes 1,\,\sigma\otimes\sigma\}\]
is identified with 
\[\{n_i\}_{i=1}^4=\{\bs_1,\,\bs_2,\,\bs_3,\,\sigma\bs_1\}.\]
Furthermore, by Example \ref{Ex:Drinfeld1} we note that 
\begin{equation}
    \Phi_E = \Phi_{\tE}=\begin{pmatrix}
        0&1\\
        \frac{t-\theta}{\kappa_2}&-\frac{\kappa_1^{(-1)}}{\kappa_2}
    \end{pmatrix}^{\otimes 2}=\begin{pmatrix}
    &&&1\\
    &&\frac{t-\theta}{\kappa_2}&-\frac{\kappa_1^{(-1)}}{\kappa_2}\\
    &\frac{t-\theta}{\kappa_2}&&-\frac{\kappa_1^{(-1)}}{\kappa_2}\\
    \frac{(t-\theta)^2}{\kappa_2^2}&-\frac{\kappa_1^{(-1)}(t-\theta)}{\kappa_2^2}&-\frac{\kappa_1^{(-1)}(t-\theta)}{\kappa_2^2}&\frac{(\kappa_1^{(-1)})^2}{\kappa_2^2}
\end{pmatrix}
\end{equation}
represents multiplication by $\sigma$ on $\cN_E$ and $\cN_{\tE}$.

We observe that there exists $C\in \GL_4(\C[t])$ so that
\[C\Phi_E=\begin{pmatrix}
    (t-\theta)^2&&&\\
    &(t-\theta)&&\\
    &&(t-\theta)&\\
    &&&1
\end{pmatrix}.\]
To calculate $\cE_0$ by \cite[Prop.~3.5.7]{NamoijamP22}, we follow \cite[Rmk. 3.5.11]{NamoijamP22} to define
\[V_E\assign
\begin{pmatrix}
    \pd n_1(\pd\tE_t^*-\theta\rI_4)\\
    \pd n_1\\
    \pd n_2\\
    \pd n_3
\end{pmatrix}=
\begin{pmatrix}
    && \kappa_1^{(-1)}& \kappa_2\\
    1&&&\\
    & 1&&\\
    &&1&
\end{pmatrix}\in \GL_4(\C)\subseteq\GL_4(\C[\tau]).\]
Note that $V_E^\tr:\rho\to\tE$ is an isomorphism of $t$-modules with $\rho:\sA\to\Mat_4(\C[\tau])$
 defined by \[\rho_t = (V_E^\tr)^{-1}\tE_tV_E^\tr.\]

Now we have isomorphisms of $t$-modules:
\[\begin{tikzcd}
\rho\arrow{r}{V_E^T}&\tE\arrow{r}{U_E^T}&E, 
\end{tikzcd}\]
which gives isomorphisms of dual $t$-motives
\[
\begin{tikzcd}
\cN_\rho\arrow{r}{(\cdot)V_E}&\cN_{\tE}\arrow{r}{(\cdot)U_E}&\cN_E. 
\end{tikzcd}\]
Then the $\C[t]$-basis $\bn_{\tE}=(n_1,\dots,n_4)^\tr$ of $\cN_{\tE}$ gives $\C[t]$-bases 
\begin{equation}\label{eq:basisdualtenten}
\bn_E = \bn_{\tE}U_E \text{ and }\bn_{\rho}=\bn_{\tE}V_E^{-1}
\end{equation}
of $\cN_E$ and $\cN_{\rho}$, respectively. The $t$-frame $\iota_E:\Mat_{1\times 4}(\C[t])\to \cN_E$ can be computed by 
\[\iota_E(\bff) =\iota_{\tE}(\bff)U_E= \iota_\rho(\bff)V_EU_E, \text{ for } \bff\in\Mat_{1\times 4}(\C[t]).\] 

We let $\partial_t:\FF_q(t)\to \FF_q(t)$ be the first hyperderivative with respect to $t$. By applying \cite[Prop.~3.5.7]{NamoijamP22} to $\rho$ and using $U_E$, $V_E\in \Mat_4(\C)$, 
\begin{equation}\label{eq:tenE0}
\cE_0(\bff) 
= \varepsilon_0(\iota_E(\bff))
=(V_EU_E)^\tr\cE_{0,\rho}(\bff)=(V_EU_E)^\tr\begin{pmatrix}
    \pd_t(f_1)\\
    f_1\\
    f_2\\
    f_3
\end{pmatrix}|_{t=\theta}=\begin{pmatrix}
    \kappa_2\pd_t(f_1)|_{t=\theta}\\
    f_3(\theta)\\
    f_2(\theta)\\
    f_1(\theta)-\frac{\kappa_1}{\kappa_2}f_2(\theta)
\end{pmatrix},
\end{equation}
for $\bff=(f_1,\dots,f_4)\in\Mat_{1\times 4}(\C(t))$.

\paragraph{The case symmetric square}
We let $E=\Sym^2(\phi)$. Similar to the tensor square case, one check directly that 
\begin{equation}\label{eq:basisdualsym}
    \{s_i\}_{i=1}^3=\{1\otimes 1,\frac{1}{2}(1\otimes \sigma+\sigma\otimes 1),\frac{1}{2}(1\otimes \sigma^2+\sigma^2\otimes 1)\}
\end{equation} 
is a $\C[\sigma]$-basis of $\cN_{\tE}\assign\Sym^2(\cN_\phi)\subseteq\cN_\phi^{\otimes 2}$ (cf. \eqref{eq:defsym}, Lemma~\ref{symalt_basis}). In the same fashion for the tensor square, we define the $t$-module $\tE:\sA\to\Mat_3(\C[\tau])$, given by 
\begin{equation}
\tE_t = 
\begin{pmatrix}
    \theta& \kappa_1\tau& \kappa_2\tau^2\\
    \kappa_1^{(-1)}& \theta+\kappa_2\tau & \kappa_1\tau\\
    \kappa_2& & \theta
\end{pmatrix},
\end{equation}
which is isomorphic to $E$ by $U_E^*:\tE\to E$, where
\[U_E= 
\begin{pmatrix}
    & & 1\\
    & \frac{1}{2}& -\frac{\kappa_1}{2\kappa_2}\\
    1 & -\frac{\kappa_1^{(-1)}}{2\kappa_2}&\frac{\kappa_1\kappa_1^{(-1)}}{2\kappa_2^2}
\end{pmatrix}\in \GL_3(\C)\subseteq\GL_3(\C[\tau]).\]
We identify $\cN_{\tE}$ with $\Mat_{1\times 3}(\C[\sigma])$ by setting the $\C[\sigma]$-basis in \eqref{eq:basisdualsym} $s_i\mapsto\bs_i$, where $\bs_1,\bs_2,\bs_3$ is the standard basis vectors in $\Mat_{1\times 3}(\C[\sigma])$. In this way, the $\C[t]$-basis 
\[\{1\otimes 1,\,\frac{1}{2}(1\otimes \sigma+\sigma\otimes 1),\,\sigma\otimes\sigma\}\]
is identified with 
\[\{n_i\}_{i=1}^3=\{\bs_1,\,\bs_2,\,\sigma\bs_1\}.\]
Furthermore, by Example \ref{Ex:Drinfeld1} we note that 
\begin{equation}
    \Phi_E = \Phi_{\tE}=\Sym^2\begin{pmatrix}
        0&1\\
        \frac{t-\theta}{\kappa_2}&-\frac{\kappa_1^{(-1)}}{\kappa_2}
    \end{pmatrix}=\begin{pmatrix}
    &&1\\
    &\frac{t-\theta}{\kappa_2}&-\frac{\kappa_1^{(-1)}}{\kappa_2}\\
    \frac{(t-\theta)^2}{\kappa_2^2}&-2\frac{\kappa_1^{(-1)}(t-\theta)}{\kappa_2^2}&\frac{(\kappa_1^{(-1)})^2}{\kappa_2^2}
\end{pmatrix}
\end{equation}
represents multiplication by $\sigma$ on $\cN_E$ and $\cN_{\tE}$.

We observe that there exists $C\in \GL_3(\C[t])$ so that
\[C\Phi_E=\begin{pmatrix}
    (t-\theta)^2&&\\
    &(t-\theta)&\\
    &&1
\end{pmatrix}.\]
Again, we follow \cite[Rmk. 3.5.11]{NamoijamP22} to define
\[V_E\assign
\begin{pmatrix}
    \pd n_1(\pd\tE_t^*-\theta\rI_3)\\
    \pd n_1\\
    \pd n_2
\end{pmatrix}=
\begin{pmatrix}
    & \kappa_1^{(-1)}& \kappa_2\\
    1&&\\
    & 1&
\end{pmatrix}\in \GL_3(\C)\subseteq\GL_3(\C[\tau]).\]
Note that $V_E^\tr:\rho\to\tE$ is an isomorphism of $t$-modules with $\rho:\sA\to\Mat_3(\C[\tau])$
 defined by \[\rho_t = (V_E^\tr)^{-1}\tE_tV_E^\tr.\]

Now we have isomorphisms of $t$-modules:
\[\begin{tikzcd}
\rho\arrow{r}{V_E^T}&\tE\arrow{r}{U_E^T}&E, 
\end{tikzcd}\]
which gives isomorphisms of dual $t$-motives
\[
\begin{tikzcd}
\cN_\rho\arrow{r}{(\cdot)V_E}&\cN_{\tE}\arrow{r}{(\cdot)U_E}&\cN_E. 
\end{tikzcd}\]
Then the $\C[t]$-basis $\bn_{\tE}=(n_1,n_2,n_3)^\tr$ of $\cN_{\tE}$ gives $\C[t]$-bases 
\begin{equation}\label{eq:basisdualsymsym}
\bn_E = \bn_{\tE}U_E \text{ and }\bn_{\rho}=\bn_{\tE}V_E^{-1}
\end{equation}
of $\cN_E$ and $\cN_{\rho}$ respectively. So the $t$-frame $\iota_E:\Mat_{1\times 3}(\C[t])\to \cN_E$ can be computed by 
\[\iota_E(\bff) =\iota_{\tE}(\bff)U_E= \iota_\rho(\bff)V_EU_E, \text{ for } \bff\in\Mat_{1\times 3}(\C[t]).\] By applying \cite[Prop.~3.5.7]{NamoijamP22} to $\rho$ and using $U_E$, $V_E\in \Mat_3(\C)$, 
\begin{equation}
\cE_0(\bff) 
=(V_EU_E)^\tr\cE_{0,\rho}(\bff)=(V_EU_E)^\tr\begin{pmatrix}
    \pd_t(f_1)\\
    f_1\\
    f_2
\end{pmatrix}|_{t=\theta}
=\begin{pmatrix}
    \kappa_2\pd_t(f_1)|_{t=\theta}\\
    \frac{1}{2}f_2(\theta)\\
    f_1(\theta)-\frac{\kappa_1}{2\kappa_2}f_2(\theta)
\end{pmatrix},
\end{equation}
for $\bff=(f_1,f_2,f_3)\in\Mat_{1\times 3}(\C(t))$.

\paragraph{The case alternating square}
By Definition \ref{def_symalt}, the alternating square of $\phi$ is given by $\altphi_t=\theta-\kappa_2\tau$, which is a Drinfeld module of rank $1$. It follows by \cite[Ex. 3.5.14]{NamoijamP22} that, for $\bff\in\C(t)$,
\begin{equation}
\cE_0(\bff)=\bff(\theta).
\end{equation}

\subsubsection{Calculation of \texorpdfstring{$\bh_\xi$}{hxi}}
To compute $\bh_\xi$, we will first apply \cite[Cor.~4.5.20(a)]{NamoijamP22} to find the matrix $V$ representing the Hartl-Juschka's isomorphism in \cite[Thm.~2.5.13]{HartlJuschka20} (see \cite[Thm.~4.4.9]{NamoijamP22}).

\paragraph{The case tensor square}
We let $E = \phi^{\otimes 2}$. In the same fashion as the $\cN_{\tE}$ in \S\ref{S:E0ten}, we identify $\cM_E\assign\cM_\phi^{\otimes 2}$ with $\Mat_{1\times 4}(\C[\tau])$ by setting $s_i\mapsto\bs_i$, where $\{s_i\}_{i=1}^4$ is the $\C[\tau]$-basis in Lemma~\ref{ten_basis}. Then the $\C[t]$-basis 
\[\{1\otimes 1,\,1\otimes \tau,\,\tau\otimes 1,\,\tau\otimes\tau\}\]
is identified with 
\begin{equation}\label{eq:basistmotiveten}
    \{m_i\}_{i=1}^4=\{\bs_1,\,\bs_2,\,\bs_3,\,\tau\bs_1\}.
\end{equation}
On the other hand, by a direct computation, the $\C[t]$-basis of $\cN_E$ in \eqref{eq:basisdualtenten} is 
\[\{    \bs_4,\,\bs_3-\frac{\kappa_1}{\kappa_2}\bs_4,\,\bs_2,\,\sigma\bs_4 \}.\]

By Remark \ref{R:almoststrictlypure}, the top coefficient of $E_{t^2}$ is invertible, so the matrices $X$, $Y\in\Mat_{8\times 4}(\C[t])$ in \cite[Cor.~4.5.20(a)]{NamoijamP22} come from solving the following equations.
\[
\begin{pmatrix}
    \bs_1\\
    \bs_2\\
    \bs_3\\
    \bs_4\\
    \tau\bs_1\\
    \tau\bs_2\\
    \tau\bs_3\\
    \tau\bs_4    
\end{pmatrix}=X
\begin{pmatrix}
    \bs_1\\
    \bs_2\\
    \bs_3\\
    \tau\bs_1    
\end{pmatrix},\quad
\begin{pmatrix}
    \bs_1\\
    \bs_2\\
    \bs_3\\
    \bs_4\\
    \sigma\bs_1\\
    \sigma\bs_2\\
    \sigma\bs_3\\
    \sigma\bs_4    
\end{pmatrix}=Y
\begin{pmatrix}
    \bs_4\\
    \bs_3-\frac{\kappa_1}{\kappa_2}\bs_4\\
    \bs_2\\
    \sigma\bs_4    
\end{pmatrix},
\]
which give 
\[
X = 
\begin{pmatrix}
1&&&\\
&1&&\\
&&1&\\
*&*&*&*\\
&&&1\\
*&*&*&*\\
*&*&*&*\\
*&*&*&*
\end{pmatrix},\quad
Y = 
\begin{pmatrix}
*&*&*&*\\
&&1&\\
\frac{\kappa_1}{\kappa_2}&1&&\\
1&&&\\
*&*&*&*\\
*&*&*&*\\
*&*&*&*\\
&&&1
\end{pmatrix}.
\]
\begin{remark}
The entries $*$'s in $X$ and $Y$ can be found explicitly, but we just do not need them for the later computations. 
\end{remark}

By Definition \ref{ten_def}, the matrix $B\in\Mat_8(\C)$ in \cite[Cor.~4.5.20]{NamoijamP22} is given by 
\[B = \left(\begin{array}{c|c}B_1 & B_2\\ \hline
B_2& 0    
\end{array}\right),\]
where if we write $E_t=B_0+B_1\tau+B_2\tau^2$, then
\begin{align}\label{eq:deften}
\begin{split}
B_0 = 
\begin{pmatrix}
    \theta&&\kappa_1&\kappa_2\\
    &\theta&&\\
    &&\theta&\\
    &&&\theta
\end{pmatrix},\quad
B_1 = 
\begin{pmatrix}
    0&0&0&0\\
    \kappa_1& & \kappa_2&\\
    \kappa_1 &\kappa_2  &&\\
    &&\kappa_1&
\end{pmatrix},\quad
B_2 = 
\begin{pmatrix}
    0&0&0&0\\
    0&0&0&0\\
    0&0&0&0\\
    \kappa_2&0&0&0
\end{pmatrix}.  
\end{split}
\end{align}
Then by \cite[Cor.~4.5.20]{NamoijamP22},
\begin{align}\label{eq:Vten}
\begin{split}
    V&=\rbracket{X^{(1)}}^\tr B^\tr Y\\
    &=\begin{pmatrix}
        \frac{\kappa_1^2}{\kappa_2}&\kappa_1&\kappa_1&\kappa_2\\
        \kappa_1&\kappa_2&&\\
        \kappa_1& &\kappa_2&\\
        \kappa_2&&&
    \end{pmatrix}.
\end{split}
\end{align}

By the proof of Proposition~\ref{P:uniformizable}, 
\begin{align}\label{eq:tphiten}
\begin{split}
\tPhi_E=\tPhi_\phi^{\otimes 2} =\begin{pmatrix}
    0&1\\
    \frac{t-\theta}{\kappa_2}&-\frac{\kappa_1}{\kappa_2}
\end{pmatrix}^{\otimes 2}=\begin{pmatrix}
    &&&1\\
    &&\frac{t-\theta}{\kappa_2}&-\frac{\kappa_1}{\kappa_2}\\
    &\frac{t-\theta}{\kappa_2}&&-\frac{\kappa_1}{\kappa_2}\\
    \frac{(t-\theta)^2}{\kappa_2^2}&-\frac{\kappa_1(t-\theta)}{\kappa_2^2}&-\frac{\kappa_1(t-\theta)}{\kappa_2^2}&\frac{\kappa_1^2}{\kappa_2^2}
\end{pmatrix}.
\end{split}
\end{align}
We let $\bm = (m_1,\dots,m_4)^\tr$ be the $\C[t]$-basis for $\cM_E$ in \eqref{eq:basistmotiveten}, and write $\tPhi_E = \sum_{i=0}^2\tU_it^i$ with $\tU_i\in\Mat_4(\C).$ By \cite[(4.4.21)]{NamoijamP22}, together with \eqref{eq:deften}--\eqref{eq:tphiten}, and recall that $\xi=(\xi_1,\dots,\xi_4)^\tr\in\C^4,$ we have 
\begin{align}\label{eq:h}
\begin{split}
\bh_\xi&=\rbracket{\tU_1\bm\xi+t\tU_2\bm\xi+\tU_2\bm\rbracket{B_0\xi+B_1\xi^{(1)}+B_2\xi^{(2)}}}^\tr\cdot V\\
&=\rbracket{0,\,\frac{\xi_3}{\kappa_2},\,\frac{\xi_2}{\kappa_2},\,-\frac{1}{\kappa_2^2}((t-\theta)\xi_1-\kappa_1\xi_2+\kappa_2\xi_4)}\cdot V\\
&=\rbracket{\frac{(t-\theta)\xi_1+\kappa_1\xi_3+\kappa_2\xi_4}{\kappa_2},\,\xi_3,\,\xi_2,\,0}\in\Mat_{1\times 4}(\C[t]).
\end{split}
\end{align}

\paragraph{The case symmetric square}
We let $E=\symphi$. In this case, we identify $\cM_E\assign\Sym^2(\cM_\phi)$ with $\Mat_{1\times 3}(\C[\tau])$ by assigning the $\C[\tau]$-basis in Lemma~\ref{symalt_basis}(a) to the standard basis vectors $\bs_1,\bs_2,\bs_3$. Then the $\C[t]$-basis 
\[\{1\otimes 1,\,\frac{1}{2}(1\otimes \tau+\tau\otimes 1),\,\tau\otimes\tau\}\]
is identified with 
\begin{equation}\label{eq:basistmotivesym}
    \{m_i\}_{i=1}^3=\{\bs_1,\,\bs_2,\,\tau\bs_1\}.
\end{equation}
On the other hand, by a direct computation, the $\C[t]$-basis of $\cN_E$ in \eqref{eq:basisdualsymsym} is 
\[\{ \bs_3,\,\frac{1}{2}\bs_2-\frac{\kappa_1}{2\kappa_2}\bs_3,\,\sigma\bs_3 \}.\]

By Remark \ref{R:almoststrictlypure}, the top coefficient of $E_{t^2}$ is invertible, so the matrices $X$, $Y\in\Mat_{6\times 3}(\C[t])$ in \cite[Cor.~4.5.20(a)]{NamoijamP22} come from solving the following equations.
\[
\begin{pmatrix}
    \bs_1\\
    \bs_2\\
    \bs_3\\
    \tau\bs_1\\
    \tau\bs_2\\
    \tau\bs_3  
\end{pmatrix}=X
\begin{pmatrix}
    \bs_1\\
    \bs_2\\
    \tau\bs_1    
\end{pmatrix},\quad
\begin{pmatrix}
    \bs_1\\
    \bs_2\\
    \bs_3\\
    \sigma\bs_1\\
    \sigma\bs_2\\
    \sigma\bs_3    
\end{pmatrix}=Y
\begin{pmatrix}
    \bs_3\\
    \frac{1}{2}\bs_2-\frac{\kappa_1}{2\kappa_2}\bs_3\\
    \sigma\bs_3    
\end{pmatrix},
\]
which give 
\[
X = 
\begin{pmatrix}
1&&\\
&1&\\
*&*&*\\
&&1\\
*&*&*\\
*&*&*
\end{pmatrix},\quad
Y = 
\begin{pmatrix}
*&*&*\\
\frac{\kappa_1}{\kappa_2}&2&\\
1&&\\
*&*&*\\
*&*&*\\
&&1
\end{pmatrix}.
\]

By Definition \ref{def_symalt}(a), the matrix $B\in\Mat_6(\C)$ in \cite[Cor.~4.5.20]{NamoijamP22} is given by 
\[B = \left(\begin{array}{c|c}B_1 & B_2\\ \hline
B_2& 0    
\end{array}\right),\]
where if we write $E_t=B_0+B_1\tau+B_2\tau^2$, then
\begin{align}\label{eq:deftmodsym}
\begin{split}
B_0 = 
\begin{pmatrix}
    \theta&\kappa_1&\kappa_2\\
    &\theta&\\
    &&\theta
\end{pmatrix},\quad
B_1 = 
\begin{pmatrix}
    0&0&0\\
    \kappa_1 & \kappa_2&0\\
    0&\kappa_1&0
\end{pmatrix},\quad
B_2 = 
\begin{pmatrix}
    0&0&0\\
    0&0&0\\
    \kappa_2&0&0
\end{pmatrix}.  
\end{split}
\end{align}
Then by \cite[Cor.~4.5.20]{NamoijamP22},
\begin{align}\label{eq:Vsym}
\begin{split}
    V&=\rbracket{X^{(1)}}^\tr B^\tr Y\\
    &=\begin{pmatrix}
        \frac{\kappa_1^2}{\kappa_2}&2\kappa_1&\kappa_2\\
        2\kappa_1&2\kappa_2&\\
        \kappa_2&&
    \end{pmatrix}.
\end{split}
\end{align}

By the proof of Proposition~\ref{P:uniformizable}, 
\begin{align}\label{eq:tphisym}
\begin{split}
\tPhi_E=\Sym^2(\tPhi_\phi) =\Sym^2\begin{pmatrix}
    0&1\\
    \frac{t-\theta}{\kappa_2}&-\frac{\kappa_1}{\kappa_2}
\end{pmatrix}=\begin{pmatrix}
    &&1\\
    &\frac{t-\theta}{\kappa_2}&-\frac{\kappa_1}{\kappa_2}\\
    \frac{(t-\theta)^2}{\kappa_2^2}&-2\frac{\kappa_1(t-\theta)}{\kappa_2^2}&\frac{\kappa_1^2}{\kappa_2^2}
\end{pmatrix}.
\end{split}
\end{align}
We let $\bm = (m_1,m_2,m_3)^\tr$ be the $\C[t]$-basis for $\cM_E$ in \eqref{eq:basistmotivesym}, and write $\tPhi_E = \sum_{i=0}^2\tU_it^i$ with $\tU_i\in\Mat_3(\C).$ By \cite[(4.4.21)]{NamoijamP22}, together with \eqref{eq:deftmodsym}--\eqref{eq:tphisym}, and recall that $\xi=(\xi_1,\xi_2,\xi_3)^\tr\in\C^3,$ we have 
\begin{align}
\begin{split}
\bh_\xi&=\rbracket{\tU_1\bm\xi+t\tU_2\bm\xi+\tU_2\bm\rbracket{B_0\xi+B_1\xi^{(1)}+B_2\xi^{(2)}}}^\tr\cdot V\\
&=\rbracket{0,\,\frac{\xi_2}{\kappa_2},\,-\frac{1}{\kappa_2^2}((t-\theta)\xi_1-\kappa_1\xi_2+\kappa_2\xi_3)}\cdot V\\
&=\rbracket{\frac{(t-\theta)\xi_1+\kappa_1\xi_2+\kappa_2\xi_3}{\kappa_2},\,2\xi_2,\,0}\in\Mat_{1\times 3}(\C[t]).
\end{split}
\end{align}

\paragraph{The case alternating square}
By {\cite[(4.6.12)]{NamoijamP22}}, for $\xi\in\C$, \[\bh_\xi=\xi\] since $\altphi$ is a Drinfeld module of rank $1$.

\subsubsection{Conclusion}\label{S:con}
We let $E=\tenphi$, $\symphi$ or $\altphi$ of dimension $\ell$ with the Drinfeld module $\phi:\sA\to A[\tau]$ given by $\phi_t=\theta+\kappa_1\tau+\kappa_2\tau^2$ with $\kappa_2\in\FF_q^{\times}$, and let $\cT_E$ be the corresponding matrix operator defined in \S\ref{S:matrixop}, and write $\Log_\phi(z) = \sum_{m=0}^\infty \beta_mz^{q^m}\in\power{K}{z}$. 

El-Guindy-Papanikolas \cite[(6.4), (6.5)]{EP14} constructed rational functions $\cB_m(t)\in K(t)$ such that $\cB_m(\theta)=\beta_m$ for $m\geqslant 0$. By their construction, for $m<0$, $\cB_m(t)=0$ and $\cB_0(t)=1$. We state a recursive formula for $\cB_m(t)$ with rank $2$ assumption below.

\begin{lemma}[El-Guindy-Papanikolas {\cite[Lem.6.12(b)]{EP14}}]\label{L:recur}
For $m\geqslant 1$, the sequence $\cB_m(t)$ satisfies the following recurrence:
\[\cB_m(t)=\frac{\kappa_1^{(m-1)}}{t-\theta^{(m)}}\cB_{m-1}(t)+\frac{\kappa_2}{t-\theta^{(m)}}\cB_{m-2}(t).\]
\end{lemma}

The proposition below provides an expression of  
\[\cR_{E,m}\assign(\Phi_E^{-1})^{(m)}\cdots(\Phi_E^{-1})^{(1)}\]
in terms of these rational functions.

\begin{proposition}\label{P:Rm}
For $m\geqslant 1$,
    \[\cR_{E,m}=\cT_E\begin{pmatrix}
        \cB_m(t)   & \frac{\kappa_2}{t-\theta^{(1)}}\cB_{m-1}^{(1)}(t)\\
        \cB_{m-1}(t)   &\frac{\kappa_2}{t-\theta^{(1)}}\cB_{m-2}^{(1)}(t)
    \end{pmatrix}.\]
\end{proposition}
\begin{proof}
By Lemma~\ref{L:matrixop}, it suffice to show that
\[\cR_{\phi,m}\assign(\Phi_\phi^{-1})^{(m)}\cdots(\Phi_\phi^{-1})^{(1)}=
\begin{pmatrix}
\cB_m(t)   & \frac{\kappa_2}{t-\theta^{(1)}}\cB_{m-1}^{(1)}(t)\\
\cB_{m-1}(t)   &\frac{\kappa_2}{t-\theta^{(1)}}\cB_{m-2}^{(1)}(t)
\end{pmatrix}.\]
Indeed, a direct calculation shows that
\[\Phi_\phi^{-1}=\begin{pmatrix}
    \frac{\kappa_1^{(-1)}}{t-\theta}&\frac{\kappa_2}{t-\theta}\\
    1&0
\end{pmatrix},\]
which gives, by applying Lemma~\ref{L:recur} with $m=1$, 
\[\cR_{\phi,1}=\rbracket{\Phi_\phi^{-1}}^{(1)}=\begin{pmatrix}
    \cB_1(t)&\frac{\kappa_2}{t-\theta^{(1)}}\cB_0^{(1)}(t)\\
    \cB_0(t)&0
\end{pmatrix}.\]
For $m\geqslant 2$, the first column of $\cR_{\phi,m}$ follows by \cite[Lem.~3.1.4]{Chen22},
and the second column of $\cR_{\phi,m}$ follows from the results of the first column and the following formula:
\[\bracket{\cR_{\phi,m}}_{i2}=\bracket{\cR_{\phi,m-1}^{(1)}\rbracket{\Phi_\phi^{-1}}^{(1)}}_{i2}=\bracket{\cR_{\phi,m-1}^{(1)}}_{i1}\bracket{\rbracket{\Phi_\phi^{-1}}^{(1)}}_{12}=\cB^{(1)}_{m-i}\frac{\kappa_2}{t-\theta^{(1)}}.\]\end{proof}

\begin{remark}
    Instead of looking at dual $t$-motives, Khaochim-Papanikolas \cite[Lem.~4.2]{KhaochimP22} provided a similar expression for the first column of, for $m\geqslant 0$,
    \[\frac{1}{t-\theta^{(m)}}\tPhi^{-1}\rbracket{\tPhi^{-1}}^{(1)}\cdots\rbracket{\tPhi^{-1}}^{(m-1)},\]
    for Drinfeld modules.
\end{remark}

\begin{theorem}\label{T:log}
Suppose that $E=\tenphi$, $\symphi$ or $\altphi$ of dimension $\ell$ with the Drinfeld module $\phi:\sA\to A[\tau]$ given by $\phi_t=\theta+\kappa_1\tau+\kappa_2\tau^2$ with $\kappa_2\in\FF_q^{\times}$. We let $\bz = (z_1,\dots,z_\ell)^\tr$. Then 
\[\Log_E(\bz)=\bz+\sum_{m\geqslant 1}\cE_{0,E}\rbracket{\bh_\bz^{(m)}\cT_E\begin{pmatrix}
        \cB_m(t)   & \frac{\kappa_2}{t-\theta^{(1)}}\cB_{m-1}^{(1)}(t)\\
        \cB_{m-1}(t)   &\frac{\kappa_2}{t-\theta^{(1)}}\cB_{m-2}^{(1)}(t)
    \end{pmatrix}}\in\power{K}{\bz}^\ell.\]
\end{theorem}
\begin{proof}
Note that $\cE_{0,E}$ is additive and by a direct computation that $\cE_{0,E}(\bh_{\bz})=\bz$ in each case. Since $\beta_m\in K$, the result follows by \eqref{eq:exp_E}, Proposition~\ref{P:Rm}.
\end{proof}

We define $\tbeta_m \assign \cB_m^{(1)}(t)|_{t=\theta}\in K$, and $\beta_m' \assign \pd_t(\cB_m(t))|_{t=\theta}\in K$, and the following $\FF_q$-linear series in $\power{K}{z}$:
\begin{align}
\label{eq:start1}L_1(z)&\assign\sum_{m\geqslant 1} L_{1,m}z^{q^m}\assign\sum_{m\geqslant 1}\beta_{m}^2z^{q^m},\\
L_2(z)&\assign\sum_{m\geqslant 1} L_{2,m}z^{q^m}\assign\sum_{m\geqslant 1}\beta_m\beta_{m-1}z^{q^m},\\
\label{eq: L1}L'_1(z)&\assign\sum_{m\geqslant 1} L'_{1,m}z^{q^m}\assign\sum_{m\geqslant 1}2\beta_{m}\beta'_{m}z^{q^m},\\
\label{eq: L2}L'_2(z)&\assign\sum_{m\geqslant 1} L'_{2,m}z^{q^m}\assign\sum_{m\geqslant 1}(\beta'_{m}\beta_{m-1}+\beta_{m}\beta'_{m-1})z^{q^m},\\
\tL_0(z)&\assign\sum_{m\geqslant 1} \tL_{0,m}z^{q^m}\assign\frac{\kappa_2}{\theta-\theta^{(1)}}\sum_{m\geqslant 1}(\beta_{m}\tbeta_{m-2}-\beta_{m-1}\tbeta_{m-1})z^{q^m},\\
\tL_1(z)&\assign\sum_{m\geqslant 1} \tL_{1,m}z^{q^m}\assign\frac{\kappa_2}{\theta-\theta^{(1)}}\sum_{m\geqslant 1}\beta_{m}\tbeta_{m-1}z^{q^m},\\
\label{eq:end1}\tL_2(z)&\assign\sum_{m\geqslant 1} \tL_{2,m}z^{q^m}\assign\frac{\kappa_2}{\theta-\theta^{(1)}}\sum_{m\geqslant 1}\beta_{m-1}\tbeta_{m-1}z^{q^m}.
\end{align}
Combining Theorem~\ref{T:log} and the formulas for $\cT$, $\cE_0$, $\bh_{\bz}$ in previous subsubsections, we obtain explicit expressions for logarithms of tensor structures.
\begin{comment}
Need to revise
\end{comment}

\begin{corollary}\label{C:log}
Let $\phi:\sA\to A[\tau]$ be a Drinfeld module $\phi:\sA\to A[\tau]$ given by $\phi_t=\theta+\kappa_1\tau+\kappa_2\tau^2$ with $\kappa_2\in\FF_q^{\times}$, and let
\[\bL_m\assign\begin{pmatrix}
        L_{1,m}+(\theta-\theta^{(m)})L'_{1,m}& \kappa_2L'_{2,m}& \kappa_1^{(m)}L'_{1,m}+\kappa_2L'_{2,m}& \kappa_2L'_{1,m}\\
        \frac{(\theta-\theta^{(m)})}{\kappa_2}\tL_{1,m}& \tL_{0,m}+\tL_{2,m}& \frac{\kappa_1^{(m)}}{\kappa_2}\tL_{1,m}+\tL_{2,m}& \tL_{1,m}\\
        \frac{(\theta-\theta^{(m)})}{\kappa_2}\tL_{1,m}& \tL_{2,m}& \tL_{0,m}+\frac{\kappa_1^{(m)}}{\kappa_2}\tL_{1,m}+\tL_{2,m}& \tL_{1,m}\\
        \frac{(\theta-\theta^{(m)})}{\kappa_2}L_{1,m}& L_{2,m}& \frac{\kappa_1^{(m)}}{\kappa_2}L_{1,m}+L_{2,m}& L_{1,m}
    \end{pmatrix}\in\Mat_4(K)\]  for $m\geqslant 1$, then
\begin{enumerate}
    \item $\Log_{\tenphi}\begin{pmatrix}
        z_1\\z_2\\z_3\\z_4
    \end{pmatrix}=\begin{pmatrix}
        z_1\\z_2\\z_3\\z_4
    \end{pmatrix}+\begin{pmatrix}
        1&&&\\&1&&\\&&1&\\&&-\frac{\kappa_1}{\kappa_2}&1
    \end{pmatrix}\sum_{m\geqslant 1}\bL_m\begin{pmatrix}
        z_1^{q^m}\\z_2^{q^m}\\z_3^{q^m}\\z_4^{q^m}
    \end{pmatrix},$
    \item $\Log_{\symphi}\begin{pmatrix}
        z_1\\z_2\\z_3
    \end{pmatrix}=\begin{pmatrix}
        z_1\\z_2\\z_3
    \end{pmatrix}+\begin{pmatrix}
        1&&&\\&\frac{1}{2}&\frac{1}{2}&\\&-\frac{\kappa_1}{2\kappa_2}& -\frac{\kappa_1}{2\kappa_2}&1
    \end{pmatrix}\sum_{m\geqslant 1}\bL_m\begin{pmatrix}
        z_1^{q^m}\\z_2^{q^m}\\z_2^{q^m}\\z_3^{q^m}
    \end{pmatrix},$
    \item $\Log_{\altphi}(z) = z+\tL_0(z).$
\end{enumerate}
\end{corollary}

\subsection{Expressions for regulators}
We maintain the same notation as in \S\ref{S:con}, let $\be_1,\dots\be_\ell$ be standard basis vectors of $\C^\ell$, and denote $\partial_\theta:K\to K$ to be the first hyperderivative with respect to $\theta$. 

\begin{lemma}[cf. {\cite[Prop.~2.5]{Demeslay15}}]\label{L:basislie}
The set $\{\be_i\}_{i=1}^\ell$ is an $A$-basis of $\Lie(E)(A)$ via $\pd$.
\end{lemma}
\begin{proof}
It suffices to show that, for $\bsalpha\in \Lie(E)(A)=A^\ell$, there exist $b_1,\dots,b_\ell\in A$ such that
\begin{equation}\label{eq:basislie}
\bsalpha=\sum_{i=1}^\ell\pd E_{b_i(t)}\be_i.
\end{equation}
We only show the case $E=\symphi$. The remaining two cases follow by similar methods.

Observe that 
    \[\pd E_{t^i}=(\pd E_{t})^i = \begin{pmatrix}
        \theta & \kappa_1 & \kappa_2\\
        & \theta &  \\
        & & \theta \\
    \end{pmatrix}^i=\begin{pmatrix}
        \theta^i & i\theta^{i-1}\kappa_1 & i\theta^{i-1}\kappa_2\\
        & \theta^i &  \\
        & & \theta^i \\
    \end{pmatrix},\]
which implies, for $b\in A,$
\[    \pd E_{b(t)}= \begin{pmatrix}
        b & \kappa_1\pd_\theta(b) & \kappa_2\pd_\theta(b)\\
        & b &  \\
        & & b \\
    \end{pmatrix}.\]
Therefore, the existence of $b_i$ follows by rewriting \eqref{eq:basislie} to 
\[\bsalpha=\begin{pmatrix}
    b_1+\kappa_1\pd_\theta(b_2)+\kappa_2\pd_\theta(b_3)\\
    b_2\\
    b_3
\end{pmatrix}.\]
\end{proof}

\begin{lemma}\label{L:reg}
    Suppose that $\{\bv_1,\dots,\bv_\ell\}\subseteq K_\infty^\ell$ is an $A$-basis of $\Exp_{E,K_\infty}^{-1}(A^\ell)$ via $\pd$. If we write \[(\br_1,\dots,\br_\ell)^\tr\assign(\bv_1,\dots,\bv_\ell)\in\Mat_\ell(K_\infty),\] then we have 
\begin{enumerate}
    \item $\Reg_{\phi^{\otimes 2}} = \gamma\cdot \det\begin{pmatrix}
    \br_1-\kappa_1\pd_\theta(\br_3)-\kappa_2\pd_\theta(\br_4)\\
    \br_2\\
    \br_3\\
    \br_4
\end{pmatrix}.$
    \item $\Reg_{\Sym^2\phi} = \gamma\cdot \det\begin{pmatrix}
    \br_1-\kappa_1\pd_\theta(\br_2)-\kappa_2\pd_\theta(\br_3)\\
    \br_2\\
    \br_3
\end{pmatrix}.$
    \item $\Reg_{\Alt^2\phi} = \gamma\cdot \det(\br_1).$
\end{enumerate}
\end{lemma}
\begin{proof}
By the definition of $\Reg_E$ in \eqref{eq:reg}, the lemma follows by expressing $\bv_i$ in terms of standard basis vectors via $\pd$ using the same method as in the proof of Lemma~\ref{L:basislie}.
\end{proof}


Similar to \eqref{eq: L1} and \eqref{eq: L2}, we further define $\hbeta_m \assign \pd_\theta(\cB_m(t))|_{t=\theta}\in K$, and the following $\FF_q$-linear series in $\power{K}{z}$:
\begin{align}
\label{eq:start2}\Log_{\phi,2}(z)&\assign z+L_1(z)=\sum_{m\geqslant 0} \beta_{m}^2z^{q^m},\\
\label{eq:L1hat}\hL_1(z)&\assign\sum_{m\geqslant 1} \hL_{1,m}z^{q^m}\assign\sum_{m\geqslant 1}2\beta_{m}\hbeta_{m}z^{q^m},\\
\label{eq:L2hat}\hL_2(z)&\assign\sum_{m\geqslant 1} \hL_{2,m}z^{q^m}\assign\sum_{m\geqslant 1}(\hbeta_{m}\beta_{m-1}+\beta_{m}\hbeta_{m-1})z^{q^m}.
\end{align}

\begin{theorem}\label{T:reg} We have the following formulas for regulators. 
\begin{enumerate}
\item Assume that $\deg(\kappa_1)\leqslant (q+1)/2$. Then \[\Reg_{\Alt^2\phi}=\Log_{\Alt^2\phi}(1)=1+\frac{\kappa_2}{\theta-\theta^{(1)}}\sum_{m\geq 1}(\beta_m\tbeta_{m-2}-\beta_{m-1}\tbeta_{m-1}).\] 

\item Assume that $\deg(\kappa_1)\leqslant 1.$ Then
\begin{enumerate}
    \item[(i)] $\Reg_{\Sym^2\phi}=\gamma\cdot\det(M)$, where $\gamma\in\FF_q^\times$ is chosen so that it has sign $1$, and  
    \[M=
    \begin{pmatrix}1+\hL_1(\theta) & -\hL_1(\kappa_1)-2\kappa_2\hL_2(1)-\pd_\theta(\kappa_1) & -\kappa_2\hL_1(1)\\
    -\kappa_2^{-1}\tL_1(\theta) & \Log_{\altphi}(1)+\kappa_2^{-1}\tL_1(\kappa_1)+2\tL_2(1) & \tL_1(1)\\
    -\kappa_2^{-1}\Log_{\phi,2}(\theta) & \kappa_2^{-1}\Log_{\phi,2}(\kappa_1)+2L_2(1) & \Log_{\phi,2}(1)
    \end{pmatrix}.
    \]
    \item[(ii)] $\Reg_{\phi^{\otimes 2}} = \Reg_{\Sym^2\phi}\cdot\Reg_{\altphi}$.
\end{enumerate}
\end{enumerate}
\end{theorem}
\begin{proof}
By \cite[Cor.~6.9]{EP14}, the assumptions on $\deg(\kappa_1)$ imply that the logarithm series $\Log_{\tenphi}(\bz)$, $\Log_{\symphi}(\bz)$ and $\Log_{\altphi}(\bz)$ converge at standard basis vectors. We may choose 
\begin{equation}\label{eq:basisExpinv}
\{\Log_E(\be_1),\dots,\Log_E(\be_\ell)\}    
\end{equation}
as an $A$-basis of $\Exp_{E,K_\infty}^{-1}(A^\ell)$. We start with calculating $\Reg_{\tenphi}$. By Corollary~\ref{C:log}(a), the matrix
\begin{multline}\label{eq:logI}
    \left(\Log_{\tenphi}(\be_1),\dots,\Log_{\tenphi}(\be_4)\right)=
\begin{pmatrix}
        1&&&\\&1&&\\&&1&\\&&-\frac{\kappa_1}{\kappa_2}&1
    \end{pmatrix}\times\\
    \scalebox{0.85}{\parbox{\linewidth}{$\begin{pmatrix}
        \Log_{\phi,2}(1)+\theta L'_1(1)-L'_1(\theta)& \kappa_2L'_{2}(1)& L'_{1}(\kappa_1)+\kappa_2L'_{2}(1)& \kappa_2L'_{1}(1)\\
        \frac{1}{\kappa_2}(\theta\tL_{1}(1)-\tL_{1}(\theta))& \Log_{\altphi}(1)+\tL_{2}(1)& \frac{1}{\kappa_2}\tL_{1}(\kappa_1)+\tL_{2}(1)& \tL_{1}(1)\\
        \frac{1}{\kappa_2}(\theta\tL_{1}(1)-\tL_{1}(\theta))& \tL_{2}(1)& \Log_{\altphi}(1)+\frac{1}{\kappa_2}\tL_{1}(\kappa_1)+\tL_{2}(1)& \tL_{1}(1)\\
        \frac{1}{\kappa_2}(\theta L_{1}(1)-L_{1}(\theta))& L_{2}(1)& \frac{\kappa_1}{\kappa_2}+\frac{1}{\kappa_2}L_{1}(\kappa_1)+L_{2}(1)& \Log_{\phi,2}(1)
    \end{pmatrix}$}}.
\end{multline}

For $f\in K(t)$, by {\cite[Lem.~2.4.6]{NamoijamP22}}, we have the chain rule
\begin{equation}\label{eq:chain}
\pd_t(f)|_{t=\theta}-\pd_\theta(f(\theta))=-\pd_\theta(f)|_{t=\theta}.
\end{equation}
Then by applying Lemma~\ref{L:reg}(a) with the $A$-basis of $\Exp_{E,K_\infty}^{-1}(A^4)$ given in \eqref{eq:basisExpinv}, \eqref{eq:logI} and \eqref{eq:chain} as well as \eqref{eq:L1hat} and \eqref{eq:L2hat} give the following expression for $\Reg_{\tenphi}$, for some $\gamma_{\tenphi}\in\FF_q^{\times}$,
\begin{multline}\label{eq:reg1}
    \Reg_{\tenphi}\\
    \scalebox{0.85}{\parbox{\linewidth}{$=\gamma_{\tenphi}\cdot\begin{vmatrix}
        1-\theta \hL_1(1)+\hL_1(\theta)& -\kappa_2\hL_{2}(1)& -\hL_{1}(\kappa_1)-\kappa_2\hL_{2}(1)-\pd_\theta(\kappa_1)& -\kappa_2\hL_{1}(1)\\
        \frac{1}{\kappa_2}(\theta\tL_{1}(1)-\tL_{1}(\theta))& \Log_{\altphi}(1)+\tL_{2}(1)& \frac{1}{\kappa_2}\tL_{1}(\kappa_1)+\tL_{2}(1)& \tL_{1}(1)\\
        \frac{1}{\kappa_2}(\theta\tL_{1}(1)-\tL_{1}(\theta))& \tL_{2}(1)& \Log_{\altphi}(1)+\frac{1}{\kappa_2}\tL_{1}(\kappa_1)+\tL_{2}(1)& \tL_{1}(1)\\
        \frac{1}{\kappa_2}(\theta L_{1}(1)-L_{1}(\theta))& L_{2}(1)& \frac{\kappa_1}{\kappa_2}+\frac{1}{\kappa_2}L_{1}(\kappa_1)+L_{2}(1)& \Log_{\phi,2}(1)
    \end{vmatrix}$,}}
\end{multline}

We then proceed the following row and column operations
\begin{enumerate}
    \item $\rR_2\mapsto\rR_2-\rR_3$,
    \item $\rC_3\mapsto\rC_2+\rC_3$. 
    \item $\rC_1\mapsto\rC_1-\frac{\theta}{\kappa_2}\rC_4$. 
\end{enumerate}
Then the determinant becomes 
\[
\Log_{\altphi}(1)\cdot\det(M).
\]
One can check, by similar calculations for $\symphi$ and $\altphi$, that the first factor gives $\Reg_{\altphi}$, and the second factor gives $\Reg_{\symphi}$, which complete the proof.\end{proof}

\begin{remark}
\begin{enumerate}
    \item Theorem~\ref{T:reg} shows that special values of the dilogarithm function $\Log_{\phi,2}(z)$ of $\phi$ appear in the regulators of $\tenphi$ and $\altphi$.
    \item The formulas in Theorem~\ref{T:reg} give explicit expressions for special values of convolution $L$-series appearing in Corollaries \ref{C:Lmumurxr}, \ref{C:Lmumurxl}, \ref{C:Lsymrxr}, \ref{C:Laltrxr}.
\end{enumerate}
\end{remark}

\newpage
\bibliographystyle{siam}
\bibliography{myReference}

\begin{thebibliography}{10}

\bibitem{Aigner}
{\sc M.~Aigner}, {\em A course in enumeration}, vol.~238 of Graduate Texts in
  Mathematics, Springer, Berlin, 2007.

\bibitem{And86}
{\sc G.~W. Anderson}, {\em {$t$}-motives}, Duke Math. J., 53 (1986),
  pp.~457--502.

\bibitem{ABP04}
{\sc G.~W. Anderson, W.~D. Brownawell, and M.~A. Papanikolas}, {\em
  Determination of the algebraic relations among special {$\Gamma$}-values in
  positive characteristic}, Ann. of Math. (2), 160 (2004), pp.~237--313.

\bibitem{AndThak90}
{\sc G.~W. Anderson and D.~S. Thakur}, {\em Tensor powers of the {C}arlitz
  module and zeta values}, Ann. of Math. (2), 132 (1990), pp.~159--191.

\bibitem{ANT20}
{\sc B.~Angl\`es, T.~Ngo~Dac, and F.~Tavares~Ribeiro}, {\em On special
  {$L$}-values of {$t$}-modules}, Adv. Math., 372 (2020), pp.~107313, 33.

\bibitem{ANT22}
\leavevmode\vrule height 2pt depth -1.6pt width 23pt, {\em A class formula for
  admissible {A}nderson modules}, Invent. Math., 229 (2022), pp.~563--606.

\bibitem{APT16}
{\sc B.~Angl\`es, F.~Pellarin, and F.~Tavares~Ribeiro}, {\em Arithmetic of
  positive characteristic {$L$}-series values in {T}ate algebras}, Compos.
  Math., 152 (2016), pp.~1--61.
\newblock With an appendix by F. Demeslay.

\bibitem{AnglesTaelman15}
{\sc B.~Angl\`es and L.~Taelman}, {\em Arithmetic of characteristic {$p$}
  special {$L$}-values}, Proc. Lond. Math. Soc. (3), 110 (2015),
  pp.~1000--1032.
\newblock With an appendix by Vincent Bosser.

\bibitem{AT17}
{\sc B.~Angl\`es and F.~Tavares~Ribeiro}, {\em Arithmetic of function field
  units}, Math. Ann., 367 (2017), pp.~501--579.

\bibitem{Beaumont21}
{\sc T.~Beaumont}, {\em On equivariant class formulas for $t$-modules}, (2021).
\newblock arXiv:2110.15696.

\bibitem{BP20}
{\sc W.~D. Brownawell and M.~A. Papanikolas}, {\em A rapid introduction to
  {D}rinfeld modules, {$t$}-modules, and {$t$}-motives}, in {$t$}-motives:
  {H}odge structures, transcendence and other motivic aspects, EMS Ser. Congr.
  Rep., EMS Publ. House, Berlin, [2020] \copyright 2020, pp.~3--30.

\bibitem{Bump89}
{\sc D.~Bump}, {\em The {R}ankin-{S}elberg method: a survey}, in Number theory,
  trace formulas and discrete groups ({O}slo, 1987), Academic Press, Boston,
  MA, 1989, pp.~49--109.

\bibitem{Carlitz35}
{\sc L.~Carlitz}, {\em On certain functions connected with polynomials in a
  {G}alois field}, Duke Math. J., 1 (1935), pp.~137--168.

\bibitem{CEP18}
{\sc C.-Y. Chang, A.~El-Guindy, and M.~A. Papanikolas}, {\em Log-algebraic
  identities on {D}rinfeld modules and special {$L$}-values}, J. Lond. Math.
  Soc. (2), 97 (2018), pp.~125--144.

\bibitem{CP12}
{\sc C.-Y. Chang and M.~A. Papanikolas}, {\em Algebraic independence of periods
  and logarithms of {D}rinfeld modules}, J. Amer. Math. Soc., 25 (2012),
  pp.~123--150.
\newblock With an appendix by Brian Conrad.

\bibitem{Chen22}
{\sc Y.-T. Chen}, {\em On {F}urusho's analytic continuation of {D}rinfeld
  logarithms}, (2022).
\newblock arXiv:2210.05277.

\bibitem{Demeslay15}
{\sc F.~Demeslay}, {\em A class formula for $l$-series in positive
  characteristic}, 2015.
\newblock arXiv:1412.3704.

\bibitem{DemeslayPhD}
\leavevmode\vrule height 2pt depth -1.6pt width 23pt, {\em Formules de classes
  en caract\'{e}ristique positive}, 2015.
\newblock Th\`{e}se de doctorat, Universit\'{e} de Caen Basse-Normandie.

\bibitem{EP14}
{\sc A.~El-Guindy and M.~A. Papanikolas}, {\em Identities for {A}nderson
  generating functions for {D}rinfeld modules}, Monatsh. Math., 173 (2014),
  pp.~471--493.

\bibitem{Fang15}
{\sc J.~Fang}, {\em Special {$L$}-values of abelian {$t$}-modules}, J. Number
  Theory, 147 (2015), pp.~300--325.

\bibitem{FGHP20}
{\sc J.~Ferrara, N.~Green, Z.~Higgins, and C.~D. Popescu}, {\em An equivariant
  {T}amagawa number formula for {D}rinfeld modules and applications}, Algebra
  Number Theory, 16 (2022), pp.~2215--2264.

\bibitem{Gekeler91}
{\sc E.-U. Gekeler}, {\em On finite {D}rinfeld modules}, J. Algebra, 141
  (1991), pp.~187--203.

\bibitem{Gezmis19}
{\sc O.~Gezmi\c{s}}, {\em Taelman {$L$}-values for {D}rinfeld modules over
  {T}ate algebras}, Res. Math. Sci., 6 (2019), pp.~Paper No. 18, 25.

\bibitem{Gezmis21}
\leavevmode\vrule height 2pt depth -1.6pt width 23pt, {\em Special values of
  {G}oss {$L$}-series attached to {D}rinfeld modules of rank 2}, J. Th\'{e}or.
  Nombres Bordeaux, 33 (2021), pp.~511--552.

\bibitem{GezmisNamoijam21}
{\sc O.~Gezmi\c{s} and C.~Namoijam}, {\em On the transcendence of special
  values of goss $l$-functions attached to drinfeld modules}, 2021.
\newblock arXiv:2110.02569.

\bibitem{Goldfeld}
{\sc D.~Goldfeld}, {\em Automorphic forms and {$L$}-functions for the group
  {${\rm GL}(n,\mathbf{R})$}}, vol.~99 of Cambridge Studies in Advanced
  Mathematics, Cambridge University Press, Cambridge, 2006.
\newblock With an appendix by Kevin A. Broughan.

\bibitem{Goss79}
{\sc D.~Goss}, {\em {$v$}-adic zeta functions, {$L$}-series and measures for
  function fields}, Invent. Math., 55 (1979), pp.~107--119.
\newblock With an addendum.

\bibitem{Goss83}
\leavevmode\vrule height 2pt depth -1.6pt width 23pt, {\em On a new type of
  {$L$}-function for algebraic curves over finite fields}, Pacific J. Math.,
  105 (1983), pp.~143--181.

\bibitem{Goss92}
\leavevmode\vrule height 2pt depth -1.6pt width 23pt, {\em {$L$}-series of
  {$t$}-motives and {D}rinfeld modules}, in The arithmetic of function fields
  ({C}olumbus, {OH}, 1991), vol.~2 of Ohio State Univ. Math. Res. Inst. Publ.,
  de Gruyter, Berlin, 1992, pp.~313--402.

\bibitem{Goss94}
\leavevmode\vrule height 2pt depth -1.6pt width 23pt, {\em Drinfeld modules:
  cohomology and special functions}, in Motives ({S}eattle, {WA}, 1991),
  vol.~55 of Proc. Sympos. Pure Math., Amer. Math. Soc., Providence, RI, 1994,
  pp.~309--362.

\bibitem{Goss}
{\sc D.~{Goss}}, {\em Basic structures of function field arithmetic}, vol.~35,
  Springer-Verlag, Berlin, 1996.

\bibitem{Hamahata}
{\sc Y.~Hamahata}, {\em Tensor products of {D}rinfeld modules and {$v$}-adic
  representations}, Manuscripta Math., 79 (1993), pp.~307--327.

\bibitem{HartlJuschka20}
{\sc U.~Hartl and A.-K. Juschka}, {\em Pink's theory of {H}odge structures and
  the {H}odge conjecture over function fields}, in {$t$}-motives: {H}odge
  structures, transcendence and other motivic aspects, EMS Ser. Congr. Rep.,
  EMS Publ. House, Berlin, [2020] \copyright 2020, pp.~31--182.

\bibitem{HsiaYu00}
{\sc L.-C. Hsia and J.~Yu}, {\em On characteristic polynomials of geometric
  {F}robenius associated to {D}rinfeld modules}, Compositio Math., 122 (2000),
  pp.~261--280.

\bibitem{HP22}
{\sc W.-C. Huang and M.~A. Papanikolas}, {\em Convolutions of {G}oss and
  {P}ellarin {$L$}-series}, 2022.
\newblock arXiv: 2206.14931.

\bibitem{Khaochim}
{\sc C.~Khaochim}, {\em Rigid analytic trivializations and periods of
  {D}rinfeld modules and their tensor products}, (2021).
\newblock {PhD} dissertation, Texas A\&M University.

\bibitem{KhaochimP22}
{\sc C.~Khaochim and M.~A. Papanikolas}, {\em Effective rigid analytic
  trivializations for {D}rinfeld modules}, Canad. J. Math.,  (2022), p.~1–30.

\bibitem{Littlewood}
{\sc D.~E. Littlewood}, {\em The theory of group characters and matrix
  representations of groups}, AMS Chelsea Publishing, Providence, RI, 2006.
\newblock Reprint of the second (1950) edition.

\bibitem{Maurischat21}
{\sc A.~Maurischat}, {\em Abelian equals $\mathbf{A}$-finite for anderson
  $\mathbf{A}$-modules}, (2021).
\newblock arXiv:2110.11114.

\bibitem{NamoijamP22}
{\sc C.~Namoijam and M.~A. Papanikolas}, {\em Hyperderivatives of periods and
  quasi-periods for anderson $t$-modules}, Mem. Amer. Math. Soc.,  (to appear).

\bibitem{Ore33a}
{\sc O.~Ore}, {\em On a special class of polynomials}, Trans. Amer. Math. Soc.,
  35 (1933), pp.~559--584.

\bibitem{Pellarin08}
{\sc F.~Pellarin}, {\em Aspects de l'ind\'{e}pendance alg\'{e}brique en
  caract\'{e}ristique non nulle (d'apr\`es {A}nderson, {B}rownawell, {D}enis,
  {P}apanikolas, {T}hakur, {Y}u, et al.)}, no.~317, 2008, pp.~Exp. No. 973,
  viii, 205--242.
\newblock S\'{e}minaire Bourbaki. Vol. 2006/2007.

\bibitem{Pellarin12}
\leavevmode\vrule height 2pt depth -1.6pt width 23pt, {\em Values of certain
  {$L$}-series in positive characteristic}, Ann. of Math. (2), 176 (2012),
  pp.~2055--2093.

\bibitem{Poonen96}
{\sc B.~Poonen}, {\em Fractional power series and pairings on {D}rinfeld
  modules}, J. Amer. Math. Soc., 9 (1996), pp.~783--812.

\bibitem{Stanley}
{\sc R.~P. Stanley}, {\em Enumerative combinatorics. {V}ol. 2}, vol.~62 of
  Cambridge Studies in Advanced Mathematics, Cambridge University Press,
  Cambridge, 1999.
\newblock With a foreword by Gian-Carlo Rota and appendix 1 by Sergey Fomin.

\bibitem{Taelman09}
{\sc L.~Taelman}, {\em Special {$L$}-values of {$t$}-motives: a conjecture},
  Int. Math. Res. Not. IMRN,  (2009), pp.~2957--2977.

\bibitem{Taelman10}
\leavevmode\vrule height 2pt depth -1.6pt width 23pt, {\em A {D}irichlet unit
  theorem for {D}rinfeld modules}, Math. Ann., 348 (2010), pp.~899--907.

\bibitem{Taelman12}
\leavevmode\vrule height 2pt depth -1.6pt width 23pt, {\em Special {$L$}-values
  of {D}rinfeld modules}, Ann. of Math. (2), 175 (2012), pp.~369--391.

\bibitem{Takahashi82}
{\sc T.~Takahashi}, {\em Good reduction of elliptic modules}, J. Math. Soc.
  Japan, 34 (1982), pp.~475--487.

\bibitem{Thakur}
{\sc D.~S. Thakur}, {\em Function field arithmetic}, World Scientific
  Publishing Co., Inc., River Edge, NJ, 2004.

\end{thebibliography}

\end{document}